\documentclass[reqno, 11pt]{amsart}
\usepackage{hyphenat}
\usepackage{amssymb}
\usepackage{hyperref}
\usepackage{amsrefs}
\usepackage{color}
\usepackage{orcidlink}
\usepackage{academicons}
\usepackage{xcolor} 
\usepackage{multicol}
\usepackage{todonotes}
\usepackage[normalem]{ulem}
\allowdisplaybreaks
 
\theoremstyle{plain}
\newtheorem{thm}{Theorem}[section]
\newtheorem{lemma}[thm]{Lemma}

\theoremstyle{definition}

\newtheorem{remark}[thm]{Remark}

\newtheorem{ex}[thm]{Example}

\def\Xint#1{\mathchoice
   {\XXint\displaystyle\textstyle{#1}}%
   {\XXint\textstyle\scriptstyle{#1}}%
   {\XXint\scriptstyle\scriptscriptstyle{#1}}%
   {\XXint\scriptscriptstyle\scriptscriptstyle{#1}}%
   \!\int}
\def\XXint#1#2#3{{\setbox0=\hbox{$#1{#2#3}{\int}$}
     \vcenter{\hbox{$#2#3$}}\kern-.5\wd0}}
\def\dashint{\Xint-}

\def\d{{\fam0 d}}
\def\BMO{\operatorname{BMO}}
\def\N{\mathbb N}

\def\R{\mathbb R}
 
 \def\Rn{\mathbb R\sp n}
 
\newcommand{\overbar}[1]{\mkern 1.5mu\overline{\mkern-1.5mu#1\mkern-1.5mu}\mkern 1.5mu}

\numberwithin{equation}{section}

\begin{document}

\title[Optimal Sobolev embeddings  for generalized Lorentz\hyp{}Zygmund spaces]{Optimal Sobolev embeddings \\ for generalized Lorentz\hyp{}Zygmund spaces}
 
\author{Paola Cavaliere}
\address{Dipartimento di Matematica, Università di Salerno, Via Giovanni Paolo II, 84084 Fisciano (SA), Italy}
\email[Corresponding author]{pcavaliere@unisa.it}
\thanks{Research Project of the University of Salerno \lq \lq Regolarità di soluzioni di equazioni alle derivate parziali non lineari e metodi funzionali'', grant number ORSA240999; GNAMPA   of the Italian INdAM - National Institute of High Mathematics (grant number not available)}

\author{Ladislav Drážný}
\address{Department of Mathematical Analysis, Faculty of Mathematics and Physics, Charles University, Sokolovská~83, 186~75, Praha~8, Czech Republic}
\address{Department of Mathematics, Faculty of Electrical Engineering, Czech Technical University in Prague, Technická~2, 166~27, Praha~6, Czech Republic}
\email{draznyl2@seznam.cz}
\thanks{Grant number 314325 of the Charles University Grant Agency; Grant number 23-04720S of the Czech Science Foundation;
 Primus research programme PRIMUS/21/SCI/002 of Charles University;  Charles University research programme number UNCE24/SCI/005} 

\numberwithin{equation}{section}




\begin{abstract} This work deals with embeddings, of any integer order, for generalized Lorentz\hyp{}Zygmund\hyp{}Sobolev spaces on Euclidean domains satisfying minimal regularity assumptions. Explicit re\-ar\-ran\-ge\-ment\hyp{}invariant,  H\"older, Morrey and Campanato optimal target spaces are exhibited. 
\end{abstract}
 
\maketitle

\section{Introduction}\label{Intro}
  
Arbitrary order Sobolev spaces on bounded domains naturally come into play in the study of  variational problems and PDEs involving weak derivatives, especially for those arising from geometry and physics  (see e.g.~\cites{BB, Mabook, Pedregal, Sobook}). These analytical problems benefit not only from information arising from the seamless combination of weak differentiability and integrability properties required to  functions  within a Sobolev space,  but also from the extra regularity that functions gain through such a membership whenever the underlying domain is regular enough, as a result of the so\hyp{}called Sobolev embeddings.

To be more specific, and keep things simple, let $\Omega$ be a \lq nice' bounded domain. Here, and in what follows, the term domain and the symbol $\Omega$ will be reserved for a connected open set in  the $n$\hyp{}dimensional real Euclidean space, with $n \geq 2$.  
 
\noindent Given any positive integer  $m$ and $p \in [1, \infty]$,  let $W^{m,p}(\Omega)$ be   the usual Sobolev space of order $m$ (and of integrability $p$) of those functions in $\Omega$ which together with all their (existing) weak  derivatives up to the order $m$  belong to  the Lebesgue space  $L^p(\Omega)$. 
Then, the outstanding results of S.~Sobolev \cite{sobolev1938theorem}, E.~Gagliardo \cite{Gagliardo}, L.~Nirenberg \cite{Nir}, and C.B.~Morrey \cite{MOR}  -typically merged in the so\hyp{}called Sobolev Embedding Theorem in spite of their  different nature and methods of proof (see e.g.~\cites{AF, Bressan})- reveal  that functions in $W^{m,p}(\Omega)$ always enjoy extra regularity.  This is dictated  -solely, if  $p \in (1, \infty)$- by the behavior of the quantity $m - \frac np$ with respect to $0$, in the usual trichotomy sense. The latter quantity is therefore referred to as the  net smoothness number of functions in $W^{m,p}(\Omega)$ (see e.g.~\cites{Henry, Bressan}), where the standard convention that $\frac n\infty = 0$ is adopted.
  
More precisely, one has that  
\begin{equation}\label{CS1} W^{m,p}(\Omega) \ \hookrightarrow  \ 
\begin{cases} L^{p^{\ast}}\!(\Omega) \quad \ &\text{if} \ m - \tfrac n p< 0\\[0.2ex]  L^{q}(\Omega) \quad  \text{for every} \ q \in \left[\tfrac nm,  \infty\right) \quad & \text{if} \ m - \tfrac np= 0, \   p  \neq 1  \\[0.2ex]L^{\infty}(\Omega) \quad \ &\text{if} \, \begin{cases}m - \tfrac n p= 0, \ p=1\\[0.2ex]  m - \tfrac n p> 0. 
\end{cases} 
\end{cases}
\end{equation} 
Here, and in what follows, the arrow '$\hookrightarrow$\label{ARROW}' between two spaces stands for a continuous embedding, and $p^{\ast}= \tfrac{np}{n-mp}$\label{Sconj}, the Sobolev conjugate of $p \in \left[1, \tfrac nm \right)$. 

\noindent Then, the (appropriate) regularity of $\Omega$  guarantees the embeddings \eqref{CS1}\textsubscript{3} and \eqref{CS1}\textsubscript{4} to be  refined, respectively, as follows. First,
\begin{equation}\label{CS2}
W^{m,p}(\Omega) \ \hookrightarrow  \  C^0_b(\Omega) \quad \  \text{if} \  m - \tfrac n p= 0, \ p=1,
\end{equation} 
where  $C^0_b(\Omega)$\label{CON} denotes the space of bounded continuous functions on $\Omega$ endowed with the usual norm. 
Next, when $m - \tfrac np> 0$,  functions in $W^{m,p}(\Omega)$ (or pointwise a.e.~equivalent versions of them) are H\"older continuous, and the  order of Hölder continuity  does depend on how their net smoothness number behaves with respect to $1$. 
Exactly, 
\begin{equation}\label{CS3} 
W^{m,p}(\Omega) \ \hookrightarrow \ 
\begin{cases}    C^{0, m - \frac np } (\Omega)  & \text{if} \ m - \tfrac np \in (0,1)\\[0.2ex] 
C^{0, \nu } (\Omega) \quad  \text{for every} \ \nu \in  (0,1) \quad & \text{if} \ m - \tfrac np= 1, \   p \in (1,\infty)  \\
C^{0, 1 } (\Omega)   \quad & \text{if} \ \begin{cases}m - \tfrac np= 1, \   p \in \{ 1, \infty\}  \\[0.2ex]  m - \tfrac np> 1, \   p \in(1, \infty], \end{cases}
\end{cases}
\end{equation}
where $C^{0, \nu} (\Omega)$,  for $\nu \in (0, 1]$, is the space of H\"older continuous functions with exponent $\nu$. 

\noindent It is worth noting that both the spaces ${L}^{p^{\ast}}\!(\Omega)$ and $L^{\infty}(\Omega)$ are known to be  optimal (i.e.~smallest possible) among all Lebesgue target spaces on $\Omega$ in the embeddings \eqref{CS1}\textsubscript{1} and \eqref{CS1}\textsubscript{3,4}, respectively.
By contrast,  no optimal Lebesgue target space exists for the embedding \eqref{CS1}\textsubscript{2}.  
Indeed, when $m<n$, functions in $W^{m,\frac nm}(\Omega)$ need not be bounded (and may not be even continuous), as simple counterexamples 
-typically with logarithmic singularities-  show  (see e.g.~\cites{Moser, TRU} for $m=1$).    
A parallel scenario emerges in the embeddings  \eqref{CS3}. The target spaces in \eqref{CS3}\textsubscript{1} and   \eqref{CS3}\textsubscript{3,4} are optimal among all H\"older spaces on $\Omega$, whereas no optimal (Hölder) target space for $W^{m,p}(\Omega)$  exists when  $p=\frac n{m-1} \in (1,\infty)$. Indeed, it is well known that functions in $W^{m,\frac n{m-1}}(\Omega)$ are not necessarily  Lipschitz continuous for $m \in (1,n]$. 

\noindent A  way to overcome these difficulties has been to enlarge the classes of admissible target spaces. As a consequence of this, improved versions of the above\hyp{}mentioned optimal embeddings have also been  derived. For instance, as far as the Sobolev\hyp{}Gagliardo\hyp{}Nirenberg embeddings \eqref{CS1} are concerned, R.~O’Neil \cite{ONeil} and J.~Peetre \cite{Peetre} independently  provided  a non\hyp{}trivial strengthening of the  embedding \eqref{CS1}\textsubscript{1} when $p >1$ through the replacement of the Lebesgue space $L^{p^{\ast}}\!(\Omega)$ with the (smaller) Lorentz space $L^{p^{\ast}\!, p} (\Omega)$.   
Notably, the Lorentz space $L^{p^{\ast}\!, p} (\Omega)$ turns out to be optimal (i.e.~the smallest possible) among  all rearrangement\hyp{}invariant target spaces in \eqref{CS1}\textsubscript{1}, as shown in \cite{EKP}*{Theorem~5.11 and Example~7(1)} (and recovered by Theorem~3.1 below). Rearrangement\hyp{}invariant spaces are, in a sense, spaces of measurable functions whose norms depend  merely on the size of the functions, or, more precisely, on the measure of their level sets.  
A precise definition of these spaces, as well as other notions employed hereafter, can be found in Section~\ref{sec2} below, where the necessary background material is collected. 
Let us warn that the Lorentz  functional framework is, however,  not suited to restoring the existence of the optimal target space in embedding \eqref{CS1}\textsubscript{2}  for $W^{m,\frac nm}(\Omega)$ when $m < n$, which indeed calls for the Lorentz\hyp{}Zygmund space $L^{\infty, \frac nm; -1} (\Omega)$ (see \cites{BW, Hansson} for the improved target space and \cite{EKP}*{Theorem~5.11 and Example~7(2)} -also  Theorem~3.1 below-  for its optimality in the class of rearrangement\hyp{}invariant spaces on $\Omega$). 

\noindent Hence, heuristically speaking, the quest for optimal Sobolev embeddings depends  heavily on the chosen class of function target spaces.
Consequently,  in response to  specific applied problems, a modern theory of Sobolev spaces  has developed in the last five decades.
The crux of the matter is to replace the Lebesgue class with an ad hoc functional framework allowing optimal integrability and sharp oscillation properties of functions within these Sobolev\hyp{}type spaces. There is nowadays an extensive literature devoted to these issues (see e.g.~\cites{CCPS2, Cia_cont, Cia_sharp, Cia_opt, Cia_forum, Cia_Maz, CPbmo, CPcamp, CP_trans, CPS, CPS2, Monia1, Costea, DT, Lada, EGO1, EGO2, EGO3, EKP, FLS, GNO, JN, KP1, KP2, Jan, Moser,  Poho, Talenti, TRU}, also \cites{ED, Kol, Lubos, Mabook} and  the  references therein). 

 In particular, combining deep results from \cites{CPS, Monia1} and the more recent contribution \cite{CCPS2} yields a comprehensive picture of
optimal integrability and sharp oscillation properties of functions within Sobolev
spaces built upon arbitrary rearrangement-invariant spaces on domains $\Omega$ more general than Lipschitz. The Sobolev\hyp{}type spaces $W^{m}X(\Omega)$ are defined in analogy with the standard Sobolev spaces, where the role of the Lebesgue spaces $L^p(\Omega)$ is played by general rearrangement\hyp{}invariant spaces $X(\Omega)$. 

\noindent Specifically, \cite{CPS}*{Theorem~6.9} provides us with a sharp counterpart of embeddings \eqref{CS1} for any  rearrangement\hyp{}invariant space $X(\Omega)$ whenever $\Omega$ is a John domain. Loosely speaking, the domain $\Omega$ allows one to travel from one of its points to another without getting too close to the boundary. The optimal (i.e.~smallest possible) rearrangement\hyp{}invariant target space in Sobolev\hyp{}type embedding  for $W^mX(\Omega)$ always exists, and it is described in an implicit way (see~\eqref{E:Opt_S} -and  also \eqref{E:EOPJ2}- in Section~\ref{Sec:John}). \\ 
Sharp counterparts of embeddings \eqref{CS2} and \eqref{CS3} for general Sobolev spaces $W^mX(\Omega)$ are established in \cite{Monia1}. As we shall point out in  Section~\ref{Sec:Jones}, although stated for bounded Lipschitz domains, results in \cite{Monia1}*{Section~3} continue to hold for any bounded Jones domain, namely for those John domains fulfilling an appropriate double twisted cone property. The general setting under consideration naturally calls for H\"older target spaces $C^{0,\sigma(\cdot)}(\Omega)$ from broader classes than those employed in \eqref{CS3}, defined in terms of a general modulus of continuity  $\sigma$, which need not be just a power. Then, \cite{Monia1}*{Theorem~3.4} tells us that the optimal modulus of continuity for the Sobolev\hyp{}type embedding  
\begin{equation}\label{EmbHolder} W^mX(\Omega)\, \hookrightarrow \, C^{0, \widehat \sigma_{m,X}(\cdot)}(\Omega) 
\end{equation}
to hold does exist, and provides its rather implicit  description as well (see~\eqref{feb31} in Section~\ref{HOL}).
 
\noindent Lastly, \cite{CCPS2} deals with optimal  embeddings for general Sobolev spaces $W^mX(\Omega)$ into Morrey and Campanato type spaces. 
Loosely speaking, Morrey  norms measure the degree of decay of norms of functions over balls when their radius approaches zero whereas Campanato  (semi\hyp{})norms provide information on the degree of decay over balls of the oscillation of functions in norm. Both decay rates on balls are described by means of a general positive function $\varphi$ on $(0, \infty)$, that is not necessarily either continuous or monotone.
In the case when $\varphi$ is a power, these generalized spaces reproduce the classical ones \cites{MOR, Ca0, Ca, JN} (also e.g.~\cites{AX, AlS, AlS2, DDY,  Man:21, Samko} and the monographs \cites{Korenobook, Saw_DiF, Stbook, Trie}). Let us mention that Morrey and Campanato type spaces were introduced and thoroughly analyzed in \cite{Sp}, and they are of use, for instance, in the description of sharp assumptions and regularity properties of solutions to nonlinear elliptic problems (see e.g.~\cites{BCDS,BySo, CW,CIS}).
More precisely, \cite{CCPS2}*{Theorem~2.2} exhibits   the optimal (smallest) Morrey space $\mathcal M^{\widetilde\varphi_{m,X}(\cdot)}(\Omega)$ such that 
\begin{equation}\label{EmbMor}
W^{m}X(\Omega) \to \mathcal M^{\widetilde\varphi_{m,X}(\cdot)}(\Omega), 
\end{equation}
and \cite{CCPS2}*{Theorem~2.6}  provides us with  the optimal (smallest) Campanato space $\mathcal L^{\widehat \varphi_{m, X}(\cdot)}(\Omega)$  such that 
\begin{equation}\label{EmbCam}
W^{m}X(\Omega) \to \mathcal L^{\widehat \varphi_{m, X}(\cdot)}(\Omega).
\end{equation} 
The optimal Morrey and Campanato target spaces in the embeddings \eqref{EmbMor} and \eqref{EmbCam} are associated with the functions  $\widetilde\varphi_{m,X}$ and $\widehat\varphi_{m,X}$, respectively, that are almost explicitly described by \eqref{E:optMo} and \eqref{E:optCa} in Section~\ref{P:MorreyCamp}, respectively. 

The present paper aims at offering  a neat explicit description of the quoted results  in the generalized Lorentz\hyp{}Zygmund realm. We shall consider here functions belonging to Sobolev\hyp{}type spaces, defined in analogy with the standard Sobolev spaces, where the role of the Lebesgue spaces $L^p(\Omega)$ is  played by the finer  class of generalized Lorentz\hyp{}Zygmund spaces  (GLZ~spaces, for short)  $L^{p, q; \alpha, \beta}(\Omega)$ and  $L^{(1, q; \alpha, \beta)}(\Omega)$.  The parameters $p, q \in [1, \infty]$ and $\alpha, \beta \in \R$  will be   suitably chosen in order to ensure that they are (equivalent to) non\hyp{}trivial   Banach function spaces (see~\eqref{c:LZnorm} and \eqref{c:L1} below, respectively), which, up to equivalent norms, are re\-ar\-ran\-ge\-ment\hyp{}in\-va\-ri\-ant spaces.
 
\noindent  Much of the interest in GLZ~spaces -introduced by D.E. Edmunds, P. Gurka, B. Opic in \cites{EGO1, EGO2} and subsequently extensively analyzed in \cite{OP}-  stems from the fact that they allow one more tier of logarithms than the Lorentz\hyp{}Zygmund spaces, proposed and investigated in \cite{BR}.   
This addition of a logarithmic factor entails a finer distinction between functions; this provides, among other things, sharp versions of  embeddings  \eqref{CS1}\textsubscript{2} and \eqref{CS3}\textsubscript{2} above, as we will exhibit in Section~\ref{S:MAIN}. 
In particular, the GLZ~class $L^{(1, q; \alpha, \beta)}(\Omega)$ allows a perturbation of the $L^{1}$\hyp{}norm by two logarithmic factors, determining domain spaces  for Sobolev\hyp{}type spaces close (but not equal) to $L^{1}(\Omega)$, which are not usually considered in the huge literature on the above\hyp{}mentioned subject. Furthermore, the GLZ~class $L^{p, q; \alpha, \beta}(\Omega)$ encompasses a very wide variety of customary Banach function spaces: the Lebesgue spaces ($p = q$, $\alpha=\beta= 0$),  the classical  Lorentz spaces ($\alpha = \beta = 0$),  the  Zygmund spaces   ($p = q$, $\beta = 0$), and Lorentz\hyp{}Zygmund spaces ($\beta = 0$).
Importantly, the GLZ~class overlaps but does not coincide with that of the Orlicz spaces. 
For a detailed treatment of this matter we refer to \cite{OP}*{Section~8}. Here, we just recall that, loosely speaking, the  definition of the Orlicz scale is based on the simple idea of replacing the power function which generates Lebesgue spaces with a more general non\hyp{}negative convex function vanishing at the origin, namely, a Young function.
The space $L^{p, q; \alpha, \beta}(\Omega)$ actually reproduces (up to equivalent norms) the Orlicz space  $L^p(\log  L)^a(\Omega)$, with $a >0$, associated with any Young function equivalent to $t^p(\log  t)^a$ near infinity ($1 < p = q$, $\alpha= a/p$, $\beta = 0$), and the Orlicz space  $\text{exp}L^a(\Omega)$ ($p = q = \infty$, $\alpha = -1/a$, $\beta = 0$) associated with any Young function equivalent to $\text{exp}(t^a)$  near infinity.  By contrast, the GLZ~spaces $L^{p, q; \alpha, \beta}(\Omega)$ and $L^{(1, q; \alpha, \beta)}(\Omega)$ are never Orlicz spaces, even up to equivalent norms, if $q \neq p$ and $q \neq 1$, respectively.

 This motivates the quest for an explicit description of the aforementioned results  from \cites{CPS, Monia1, CCPS2} within the GLZ~framework. As far as we are aware, the novelty of our contribution is twofold  and is of interest in view of the  applications mentioned  above. On the one hand, our discussion encompasses spaces $L^{(1, q; \alpha, \beta)}$ neighboring $L^{1}$. On the other hand, we allow for one more tier of logarithms by considering $L^{p, q; \alpha, \beta}(\Omega)$ spaces with $\beta \neq 0$.  The advantage of using an additional tier of logarithms becomes apparent  when limiting embeddings are taken into account, as clearly witnessed by \eqref{E:OptRI1} below. An additional outcome of our analysis on GLZ~spaces is a precise comparison between the Sobolev embeddings into Hölder type spaces and the parallel embeddings into Campanato type spaces as well as between the latter embeddings and the Sobolev embeddings into Morrey type spaces (see Section~\ref{comparison}), which complements and augments well\hyp{}known classical results. 

The paper is structured as follows. Section~\ref{sec2} is devoted to the necessary background on rearrangement\hyp{}invariant and GLZ~spaces, the regularity conditions on the ground domain $\Omega$ and to the introduction of H\"older, Morrey and Campanato type spaces that play a role in our embeddings. We aim, indeed, at making this paper self-contained for the most part. The main results are stated in Section~\ref{S:MAIN}, organized, for ease of exposition, into different subsections according to the different families of spaces taken into account.   Section~\ref{secTec} contains two technical results designed to unify the several parameter regimes, and repeatedly employed in Section~\ref{proof}, where the proofs of the main theorems are given.  
For the reader’s convenience, we add a  concise list of the main
symbols  used throughout the paper, along with a reference to the page (and to the equation when possible) where they are first introduced or best defined.

\section{Background}\label{sec2}
In what follows and throughout the paper, we set $p'=\frac{p}{p-1}$\label{Hco} for $p \in (1,\infty)$, $1'=\infty$ and $\infty'=1$. The convention that $\frac 1\infty = 0$  and  $0 \cdot \infty=0$ is adopted without further explicit reference. For two positive expressions $A$ and $B$, we will write  $A\lesssim B$\label{A} to denote that $A\leq c\, B$ for some positive constant $c$ independent of appropriate quantities appearing in the expressions $A$ and $B$. The relation  \lq \lq$\gtrsim$\label{B}” is defined accordingly. Then, we will write $A\approx B$\label{C}  when $A\lesssim B$ and $A\gtrsim B$ simultaneously. Moreover, when  $A$ and $B$ are non\hyp{}negative functions defined on an interval $(0,a)$ the above relations are said to hold near $0$ if the pertaining inequalities just hold on $(0,r_0)$ for some $r_0<a$.

\subsection{Rearrangement\hyp{}invariant spaces}\label{sec2.1}
Let $E$ be a non-negligible Lebesgue measurable subset of $\mathbb R^n$, with $n \geq 1$. We denote by $\chi_{_E}$  the characteristic function of $E$, and by $|E|$ its Lebesgue measure.
The notation $L^0(E)$\label{D} is adopted for the Riesz space of all (equivalence classes  of)  Lebesgue measurable functions   $u\colon E \to [-\infty , \infty]$,  and $L^0_+(E)= \{u \in L^0(E) \colon u \geq 0 \ \hbox{a.e.~in} \, E\}$\label{E}.
   
The \textit{decreasing rearrangement}~$u^*\colon [0, \infty) \rightarrow [0, \infty]$ of a function  $u\in L^0(E)$  is defined as
\begin{equation}\label{decreasing rearrangementE}
u^*(s)= \inf\left\{t \geq 0\colon  \left|\{x\in E\colon |u(x)|>t\}\right|\leq s\right\} \quad \   \text{for} \ s\in[0,\infty),
\end{equation} 
and the \textit{maximal non\hyp{}increasing rearrangement} of $u$ is the function $u^{**}\colon (0,\infty)\to[0,\infty ]$  defined as
\begin{equation}\label{u^**}
u^{\ast \ast}(s)\, = \, {\frac 1 s}\, \int_0^{s}u^*(t)  \, \d t \quad \   \text{for} \ s\in(0,\infty).
\end{equation}
Both the functions $u^*$ and $u^{**}$ are non\hyp{}increasing,  $u^* \leq u^{**}$ on $(0,\infty)$, and $u^*(s) =0$ if $s \geq |E|$.  The operation $u \mapsto u^{**}$ is subadditive in the sense that 
\begin{equation*}\label{**}
(|u|+|v|)^{\ast \ast}  \, \leq \, u^{\ast \ast}  + v^{\ast \ast} \quad \   \text{for} \ u,v \in L^0(E),
\end{equation*}
and the \textit{Hardy\hyp{}Littlewood inequality} tells us that
\begin{equation}\label{HL.0}
 \int_{E}|u(x) v(x) |\, \d x\leq \int_{0}^{|E|}u^*(s)v^*(s)\, \d s  \quad \   \text{for} \ u,v \in L^0(E).
\end{equation}
 
A \textit{rearrangement\hyp{}invariant Banach (extended) function norm} on $L^0_+(0,1)$ -an \textit{r.i.~function norm}, for short- is a functional $\|\cdot\|_{X(0,1)}\colon L^0_+(0,1) \rightarrow [0,\infty]$\label{F} such that
\begin{align}
&\|f+g\|_{X(0,1)}     \leq  \|f\|_{X(0,1)} + \|g\|_{X(0,1)}  \ \text{and} \notag\\& \qquad \qquad \qquad \quad \|\lambda f\|_{X(0,1)}      =    \lambda  \, \|f\|_{X(0,1)}    \ \text{for all} \   f,g \in L^0_+(0,1)   \ \text{and} \     \lambda \geq 0;   \label{N1}      \\
&  \|f\|_{X(0,1)}     >  0    \ \text{for every non\hyp{}zero} \ f   \ \text{in} \  L^0_+(0,1);    \label{N2} \\
& \|f\|_{X(0,1)}   \leq   \|g\|_{X(0,1)}  \ \text{whenever} \       f   \leq   g   \ \text{in} \ L^0_+(0,1);        \label{N3}  \\
& \sup \|f_k\|_{X(0,1)}   =   \| f\|_{X(0,1)} \ \text{whenever} \ (f_k)_{k \in \N} \subset L^0_{+}(0,1)    \ \text{with} \   f_k \nearrow f    \ \text{a.e.~in}  \   (0,1); \label{N4} \\
& \|1\|_{{X(0,1)}}   <   \infty;  \label{N5}\\
& \text{there exists a positive  constant} \ c \  \text{such that}  \   \int_0^1 f (s)\, \d s  \leq c  \|f\|_{{X(0,1)}} \   \text{for all} \  
     f  \in L^0_{+}(0,1);  \label{N6} \\
& \|f\|_{{X(0,1)}} \, = \, \| g\|_{{X(0,1)}}\   \text{for all} \     f, g \in   L^0_{+}(0,1)    \  \text{such that}  \ f^\ast = g^\ast. \label{N7}  
\end{align}
With any r.i.~function norm $\|\cdot\|_{{X(0,1)}}$, it is associated another functional on $L^0_+(0,1)$, denoted by $\|\cdot\|_{{X'(0,1)}}$, and defined as
\begin{equation} \label{n.assoc.}
\| g\|_{X'(0,1)}=  \sup \left\{   \int_0^1 f (s)g(s)\, \d s  \colon f \in L^0_+(0,1), \| f \|_{X(0,1)} \leq 1\right\} \quad \   \text{for} \ g \in
    L^0_+(0,1).
\end{equation}   
It turns out that $\|\cdot\|_{{X'(0,1)}}$ is also an r.i.~function norm, which is called the \textit{associate function norm of}~$\|\cdot\|_{{X(0,1)}}$.   
From~\eqref{HL.0}, one has that
\begin{equation*}\label{n.assoc.2}
\| g\|_{{X'(0,1)}}=  \sup\left\{\int_0^1 f^\ast (s)g^\ast(s)\, \d s  \colon f \in  L^0_+(0,1),  \| f \|_{X(0,1)} \leq 1 \right\}  \quad \   \text{for} \ g \in L^0_+(0,1).
\end{equation*}   

The \textit{fundamental function} of an r.i.~function norm $\|\cdot\|_{{X(0,1)}}$ is the function  $\phi_X \colon [0, 1] \to [0, \infty)$ defined as
\begin{equation}\label{feb99}\phi_X(t)= \|\chi_{(0,t)}\|_{{X(0,1)}} \quad \   \text{for} \ t \in (0,1] \end{equation} and $\phi_X(0)=0$. By 
  \cite{BS}*{Chapter~2, Corollary~5.3}, the function  $\phi_X$
 is  quasiconcave in $[0, 1]$, continuous except perhaps at the
origin, and satisfies
\begin{equation}\label{ASS_funda} 
\phi_X(t)\,\phi_{X'}(t) = t \quad \ \quad \   \text{for} \ t \in  [0,1].
\end{equation}
 Recall that  a function $\phi \colon [0, 1] \to [0, \infty)$ is called    \textit{quasiconcave} if  it vanishes only at $0$,  $\phi $ and the function   $\overbar{\phi} $, defined as
 \begin{equation}\label{ASS_q-conc} 
\overbar{\phi} (t)  = \frac{t}{\phi (t)} \quad \ \quad \   \text{for} \ t \in  (0,1] 
\end{equation} and $\overbar{\phi}(0)=0$, are both
 non-decreasing on $(0, 1]$. Notice that the function $\overbar{\phi}$ is quasiconcave as well.  \\
Quasiconcave functions need not be concave, but    every quasiconcave function is however always equivalent, up to multiplicative constants, to a concave function \cite{BS}*{Chapter~2, Proposition~5.10}.
 
Given an r.i.~function norm $\|\cdot\|_{{X(0,1)}}$ and a non-negligible Lebesgue measurable set $E$ of finite measure, the {\it rearrangement\hyp{}invariant space} $X(E)$ built upon  $\|\cdot\|_{X(0,1)}$ is defined as the collection of all functions $u \in L^0(E)$ such that the quantity  
\begin{equation}\label{norm}
\|u\|_{{X(E)}}=\|u\sp*(|E| \,\cdot) \|_{{X(0,1)}}
\end{equation}
is finite.
The space $X(E)$ is a Banach space, endowed with the r.i.~ function norm  given by \eqref{norm}, and the space $X(0,1)$ is called the \textit{representation space} of $X(E)$.
Note that the quantity $\|u\|_{{X(E)}}$ is defined for every $u\in L^{0}(E)$, and it is finite if and only if $u\in X(E)$. 
Moreover,
$$X(E) \subset \{u \in L^0(E) \colon u   \ \text{is finite a.e.~in} \ E\} .$$

\noindent The \textit{associate space} $X'(E)$ of $X(E)$ is the rearrangement\hyp{}invariant space built upon the r.i.~function norm $\|\cdot \|_{{X'(0,1)}}$.
In particular, $(X')'(E) = X(E)$ (see \cite{BS}*{Chapter~1, Theorem~2.7}), and we will write $X''(E)$ instead of  $(X')'(E)$. 

Given any $\lambda>0$, the \textit{dilation operator} $\textup D_{\lambda}$, defined at each $f\in L^0(0,1)$ by
\begin{equation}\label{E:dilation}  
(\textup D_{\lambda}f)(s)= 
f\left(\frac{s}{\lambda}\right) \chi_{(0, \lambda)}(s) \quad \   \text{for} \ s\in(0,1), 
 \end{equation}
is bounded on every rearrangement\hyp{}invariant~space $X(0,1)$, with operator norm not exceeding $\max\{1, \lambda^{-1}\}$. As a consequence, when $|E| \neq 1$, the norm $\|\cdot\|_{X(E)}$ is equivalent (but possibly different from) more customary norms, because the measure of $E$ may exceed $1$.

If $X(E)$ and $Y(E)$ are rearrangement\hyp{}invariant spaces, we write $X(E)\hookrightarrow Y(E)$ to denote that $X(E)$ is continuously embedded into $Y(E)$, in the sense that there exists a constant $c$ such that $\|u\|_{Y(E)}\le c\|u\|_{X(E)}$ for every $u\in X(E)$.  Note that a property of r.i.~function norms ensures that 
\begin{equation}\label{E:embequ2}  
X(E) \, \hookrightarrow \, Y(E) \quad \ \Longleftrightarrow \quad \ X(E) \subseteq  Y(E).
\end{equation} 
In particular, the equality of sets $X(E)=Y(E)$ implies that the norms of $X(E)$ and $Y(E)$ are equivalent. 
Moreover,
\begin{align}\label{E:embequ2a}
X(E) \, \hookrightarrow \, Y(E) \quad \ &\Longleftrightarrow \quad \ X(0,1)\,  \hookrightarrow \, Y(0,1), \\
\label{E:embequ2b}X(E) \, \hookrightarrow \, Y(E) \quad \ &\Longleftrightarrow \quad \ Y'(E) \, \hookrightarrow \, X'(E),
\end{align}
and the norms of the two embeddings coincide.

A pivotal class of examples of (Banach) rearrangement\hyp{}invariant spaces on $E$ are the Lebesgue spaces $L^p(E)$, $p \in [1, \infty]$. They are generated by the functional 
\begin{equation*}
\|f\|_{{L^p(0,1)}} = \begin{cases}\|f^\ast\|_{{L^p(0,1)}}&\quad  \text{if} \ \, p\in (0,\infty) \\[0.2ex] f^\ast(0^+)&\quad  \text{if} \ \, p=\infty  
\end{cases}
\end{equation*} 
for $f \in L^0_+(0,1)$, which is an r.i.~function norm if and only if $p \in [1, \infty]$. 
Since $|E|<\infty$, one has that
\begin{equation}\label{E:imm}
L^\infty (E) \, \hookrightarrow \, X(E) \, \hookrightarrow \, L^1(E) 
\end{equation} 
for every rearrangement\hyp{}invariant space $X(E)$. 
 
Two different scales of  rearrangement\hyp{}invariant spaces, which generalize, along different directions, the class of Lebesgue spaces are the Orlicz spaces  and the (generalized) Lorentz\hyp{}Zygmund spaces. The latter ones will be discussed in the next section. 

\noindent The Orlicz spaces are generated by the so-called Luxemburg functionals, whose definition, in turn, rests upon that of Young function.  
A \textit{Young function} $A\colon [0, \infty) \rightarrow [0, \infty]$ is a convex, left\hyp{}continuous function such that $A(0)=0$, and is non\hyp{}constant in $(0, \infty)$. The \textit{Luxemburg function norm} associated with a Young function $A$ is defined as
\begin{equation}\label{Orlicz}
\| f \|_{{L^{A}(0,1)}} =  \ \inf \left\{ \lambda > 0 \colon \int_0^1
A \left(\frac{f(s)}{\lambda} \right)\, \d s   \leq  1 \right\} \quad \   \text{for} \ f \in L^0_+(0,1),
\end{equation}
with the convention that $\inf \emptyset = \infty$.
The space $L^{A}(E)$ is called an \textit{Orlicz space} on $E$. Clearly, any Lebesgue space $L^{p}(E)$, with $p \in [1, \infty]$, is reproduced
as an Orlicz space $L^{A}(E)$, with the choice $A(t)=t^p$ when $p \in  [1, \infty)$ and $A(t)=\infty \chi_{_{(1, \infty)}}(t)$  when $p= \infty$. 

\noindent Given two Young functions $A$ and $B$, the Luxemburg function norms $\| \cdot \|_{{L^{A}(0,1)}}$ and $\| \cdot \|_{{L^{B}(0,1)}}$ are equivalent if and only if  $A$ and $B$ are \textit{equivalent near infinity} ($A\simeq B$ near $\infty$, for short), in the sense that there exist positive constants $c$, $C$ and $t_0$  such that
\begin{equation}\label{B.5bis}
A(ct) \leq B(t)\leq A(C t) \quad \ 
\text{for} \ t \in [t_0, \infty).  
\end{equation}
Hence, as $|E|<\infty$, 
\begin{equation}\label{B.6}
 L^A(E)= L^B(E) \ \  \text{ up to equivalent norms} \quad  \Longleftrightarrow \quad  \ A\simeq B \ \    \text{near} \ \infty.
\end{equation} 
Then, we employ the alternative notation $B(L)(E)$ to denote an Orlicz space associated to a Young function $A$ such that $A \simeq B$ near infinity.  
The {\it Orlicz spaces of exponential type},  denoted by $\text{exp}L^a(0,1)$ for $a > 0$, are those built upon the Young function $A(t) = \text{exp}(t^a) - 1$ for $t\in [0, \infty)$. The {\it Orlicz spaces of logarithmic type}, denoted by $L^p(\log L)^\alpha(0,1)$, are built upon a Young function $A$ equivalent to $t^p(1+|\log t|)^a$ near infinity, where either $p=1$, $a \geq 0$ or $p> 1$, $a \in \R$.

\noindent We refer the reader to~\cites{BS, Lubook} for a comprehensive treatment of rearrangement\hyp{}invariant spaces.
 
\subsection{Generalized Lorentz\hyp{}Zygmund spaces}\label{sec2.2} 

\noindent This section is devoted to some basic definitions and properties from the theory of GLZ~spaces. We rely upon the modern foundations of the theory of these spaces developed in the papers \cites{EOP, OP}, to which we refer the reader for more details.

Assume that $p, q\in (0,\infty]$ and $\alpha, \beta \in\R$. Let  $\|\cdot\|_{L^{p, q; \alpha, \beta}(0,1)}$ and $\|\cdot\|_{L^{(p, q; \alpha, \beta)}(0,1)}$ be the functionals   defined on functions $f \in L^{0}_{+}(0,1)$ as
\begin{equation}\label{c:LZfunct}
\|f\|_{L^{p, q; \alpha, \beta}(0,1)} = \left\|s^{\frac{1}{p}-\frac{1}{q}}\ell^\alpha(s)\ell\ell^\beta(s) f^*(s)\right\|_{L^q(0,1)}  
\end{equation}
and   -on replacing $f^{*}$ by $f^{**}$ on the right-hand side of \eqref{c:LZfunct}- 
\begin{equation}\label{c:LZfunct2}
\|f\|_{L^{(p, q; \alpha, \beta)}(0,1)} = \left\|s^{\frac{1}{p}-\frac{1}{q}}\ell^\alpha(s)\ell\ell^\beta(s) f^{**}(s)\right\|_{L^q(0,1)} .
\end{equation} 
Here, 
\begin{equation}\label{c:convLeLL}\ell(s) = 1 + |\log {s}|  \quad \ \text{and}  \quad \  \ell\ell(s) = 1 + \log\ell(s) \quad \  \text{for} \   s\in(0,\infty).\end{equation}  

\noindent  They are equivalent to r.i.~function norms if and only if the parameters $p, q, \alpha$ and $\beta$ fulfill one of the following  alternatives
\begin{equation}\label{c:LZnorm}
\begin{cases}  p=  q=1,  \    \alpha=0, \    \beta\geq 0   \\   p= q=1, \     \alpha>0, \ \beta\in \R   \\  
p\in(1,\infty), \     q\in[1,\infty], \ \alpha,\, \beta\in \R   \\ 
p=\infty, \    q\in[1,\infty], \   \alpha  < - \frac{1}{q}, \   \beta\in \R  \\ 
p=\infty,  \    q\in[1,\infty],  \   \alpha   = - \frac{1}{q}, \      \beta   < - \frac{1}{q}    \\
p=  q=\infty, \      \alpha =  \beta=0   
\end{cases}
\end{equation} 
for $\|\cdot\|_{L^{p, q; \alpha, \beta}(0,1)}$ \cite{OP}*{Theorem~7.4},   and
\begin{equation}\label{c:LZnorm2}
\begin{cases} 
p\in(0,\infty), \ q \in [1, \infty],  \ \alpha, \beta \in \R    \\ 
p=\infty, \  q \in [1, \infty],\  \alpha   < - \frac1{q}, \  \beta \in \R    \\
p=\infty, \  q \in [1, \infty], \ \alpha   = - \frac1{q}, \      \beta   < - \frac{1}{q}    \\
p=  q=\infty,\ \alpha = \beta=0  
\end{cases}
\end{equation}
for  $\|\cdot\|_{L^{(p, q; \alpha, \beta)}(0,1)}$ \cite{OP}*{Theorem~7.5}, respectively.
In these cases, with slight abuse of terminology, the functionals $\|\cdot\|_{L^{p, q; \alpha, \beta}(0,1)}$  and  $\|\cdot\|_{L^{(p, q; \alpha, \beta)}(0,1)}$ are called \textit{generalized  Lorentz\hyp{}Zygmund r.i.~function norms} and the \textit{generalized Lorentz\hyp{}Zygmund spaces} $L^{p, q; \alpha, \beta}(0,1)$ and $L^{(p, q; \alpha, \beta)}(0,1)$  -\textit{GLZ~spaces}, for short- are the corresponding rearrangement\hyp{}invariant spaces on $(0,1)$.

Let us stress that
\begin{equation}\label{c:equivLZspA}L\sp{(p, q; \alpha, \beta)}(0,1)=L\sp{p, q; \alpha, \beta}(0,1)  \quad  \text{if one of the alternatives}  \   (\ref{c:LZnorm2}) \    \text{is in force, with} \   p\in(1,\infty] 
\end{equation} 
\cite{OP}*{Theorem~3.16}, 
whereas
\begin{equation}\label{c:equivLZ2} 
 L\sp{(p, q; \alpha, \beta)}(0,1)=L\sp{1}(0,1)  \quad \   \text{if} \,  
 \begin{cases}  p\in(0,1), \     q\in[1,\infty], \ \alpha,\, \beta\in \R   \\[0.2ex]    p=1, \     q\in[1,\infty], \ \alpha   <-\frac 1q, \  \beta\in \R \\[0.2ex] p=1, \     q\in[1,\infty], \ \alpha =-\frac 1q,  \ \beta <-\frac 1q
 \\[0.2ex] 
p=1, \ q= \infty, \ \alpha =  \beta=0
 \end{cases}  
 \end{equation} 
\cite{OP}*{Lemma~3.15}, up to equivalent norms. Moreover, \begin{equation}\label{c:equivLZ2a} 
 L\sp{(1, 1; \alpha, \beta)}(0,1)=\begin{cases} 
 L^{1, 1;  \alpha+1, \beta}(0,1)&\quad   \text{if} \ \, \alpha   > - 1, \ \beta \in \R  \\[0.2ex] 
 L^{1, 1;  0, \beta+1}(0,1)&\quad   \text{if} \ \,  \alpha   =- 1, \    \beta > - 1   
 \end{cases} \end{equation}
\cite{OP}*{Theorem~3.16}

 It is worth remarking that, due to \eqref{c:equivLZspA}  and \eqref{c:equivLZ2}, as far as the $L^{(p, q; \alpha, \beta)}$ class is concerned, it suffices to deal with the spaces $L^{(1, q; \alpha, \beta)}$ under one of the present alternatives:
 \begin{equation}\label{c:L1}\begin{cases}  q\in[1,\infty], \ \alpha   > - \frac{1}{q}, \ \beta\in \R   \\[0.2ex]  
 q\in[1,\infty], \ \alpha  = - \frac{1}{q}, \ \beta > - \frac{1}{q}  \\[0.2ex] 
   q \in [1, \infty), \ \alpha= \beta = - \frac{1}{q}.  
 \end{cases}   \end{equation}

\noindent  In some borderline and limiting embeddings,  five\hyp{}parameter  Lorentz\hyp{}Zygmund spaces   $L^{p, q; \alpha, \beta, \gamma}(0,1)$ will also emerge. They are built upon the functional 
\begin{equation}\label{3log}
\| f \|_{L^{p, q; \alpha, \beta, \gamma}(0,1)} = \left\| s^{\frac{1}{p}-\frac{1}{q}}\ell^{\alpha}(s)\ell\ell^{\beta}(s) \ell\ell\ell^{\gamma}(s) f^*(s)\right\|_{L^{q}(0,1)} \quad \ \text{for} \ f \in  L^0_+(0,1),
\end{equation} 
where 
\begin{equation}\label{c:convLLL}\ell \ell \ell(s) = 1 + \log \ell \ell(s) \quad \  \text{for} \   s\in(0,1).\end{equation}
When $ q \in [1, \infty]$ and   
the parameters $p \in (1, \infty]$,  $\alpha, \beta, \gamma \in \R$ fulfill one of the following  alternatives
\begin{equation}\label{c:L1five}\begin{cases}  p \in (1, \infty), \ \alpha, \, \beta, \, \gamma \in \R   \\[0.2ex]  
 p=\infty, \ \left\| s^{-\frac{1}{q}}\ell^{\alpha}(s)\ell\ell^{\beta}(s) \ell\ell\ell^{\gamma}(s) \right\|_{L^{q}(0,1)} < \infty,  
 \end{cases}   \end{equation} 
     the functional $\| \cdot \|_{L^{p, q; \alpha, \beta, \gamma}(0,1)}$ is equivalent to an r.i.~function norm  (see e.g.~\cite{PLK}*{Theorem~3.33}).

For the reader's convenience, we  also recall well\hyp{}known characterizations of the associate spaces of GLZ~spaces, that we shall repeatedly employ in our proofs. Specifically,
 \begin{align}
 (L^{p, q; \alpha, \beta})'(0,1) &=  L^{p',q'; -\alpha, -\beta}(0,1) &  \text{if}   &\begin{cases}  \,  p = q = 1,  \ \alpha=0, \  \beta  \geq 0\\   \, p = q =1,  \ \alpha>0,   \ \beta \in \R     
 \\ 
\,  p\in(1,\infty), \ q\in[1,\infty],   \  \alpha, \beta \in \R  
\\ 
\,  p = q =\infty, \    \alpha<0, \ \beta \in \R 
\\ 
\,  p =  q =\infty, \ \alpha=0, \  \beta  \leq 0,
 \end{cases}    \label{c:ASSnorm}
 \\ 
 (L^{\infty, q; \alpha, \beta})'(0,1) &=    L^{(1, q'; -\alpha-1, -\beta)}(0,1)   & \text{if}  & \ \, q\in [1,\infty), \ \alpha<  - \tfrac 1q,   \ \beta \in \R ,         &\label{c:ASSnormA}
\\ 
 (L^{\infty, q; \alpha, \beta})'(0,1) &=   L^{(1, q'; -\frac{1}{q'}, -\beta -1)}(0,1)  &\text{if}   &   \  \, q \in [1,\infty), \ \alpha  =- \tfrac 1q, \   \beta <    - \tfrac 1q      \label{c:ASSnormB}
 \end{align}
 \cite{OP}*{Theorem~6.11}, and
 \begin{equation}\label{c:ASSnorm2}
 (L^{(1, q; \alpha, \beta)})'(0,1) = \begin{cases} 
 L^{\infty, q'; -\alpha-1, -\beta}(0,1)&\quad   \text{if} \ \, q\in[1,\infty], \ \alpha   > - \tfrac 1q, \ \beta \in \R  \\[0.2ex] 
 L^{\infty, q'; -\frac{1}{q'}, -\beta-1}(0,1)&\quad   \text{if} \ \, q \in [1, \infty],  \ \alpha   =- \tfrac 1q, \    \beta > - \tfrac 1q  \\[0.2ex]
  L^{\infty, q'; -\frac{1}{q'}, -\frac{1}{q'}, -1}(0,1)&\quad   \text{if} \ \,q\in [1,\infty), \ \alpha=   \beta   = -   \frac 1q  
 \end{cases} 
 \end{equation}\cite{OP}*{Theorem~6.12}, up to equivalent norms. 
  
As already mentioned in the Introduction, several customary families of function spaces are contained in the GLZ~class. This is the case of  the  Lebesgue spaces,  the Lorentz spaces, the spaces of Brezis\hyp{}Wainger type appearing in \cites{BW, Hansson, Poho, Yudo} in the study of refinements for embedding \eqref{CS1}\textsubscript{1,2} or of all types of exponential and logarithmic Orlicz spaces. Concretely, for $p = q$, the space $L^{p,p; \alpha, \beta}$ is equal to the Orlicz space $L^p(\log L)^\alpha(\log \log L)^\beta$, whenever one of the alternatives~\eqref{c:LZnorm} holds (cf.~\cite{OP}*{Lemma~8.1}). 
In particular, the  Lorentz\hyp{}Zygmund spaces $L^{p,q; \alpha}(0,1)$ are the spaces $L^{p,q;\alpha,0}(0,1)$ (corresponding to $\beta=0$), and the Lorentz spaces $L^{p,q}(0,1)$ are their special cases, corresponding to $\alpha=0$. 

\subsection{Endpoint spaces}\label{sec2.3} 
The Marcinkiewicz and the Lorentz endpoint rearrangement-invariant spaces built upon a general quasiconcave function will also come into play in our discussion.

Given a quasiconcave function $\phi \colon [0, 1] \to [0, \infty)$, let  
$\|\cdot\|_{\mathsf{M}_{\phi}(0,1)}$ and $\|\cdot\|_{\Lambda_\phi(0,1)}$ be the functionals   defined on functions $f \in L^{0}_{+}(0,1)$ as
 \begin{equation}\label{EndMarc_no}
\|f\|_{\mathsf{M}_{\phi}(0,1)}= \sup_{s\in (0,1)} \phi(s) f^{**}(s)
\end{equation}	
and
\begin{equation}\label{EndLor_no}
\|f\|_{\Lambda_\phi(0,1)}= \int_0^{1} f^*(s) \,d\phi (s),
\end{equation}
where $d\phi$ stands for the Lebesgue–Stieltjes measure associated with (the non-decreasing function) $\phi$.    
 The functional $\|\cdot\|_{\mathsf{M}_{\phi}(0,1)}$ is an r.i.~function norm on $L^0_+(0,1)$, whereas the functional $\|\cdot\|_{\Lambda_\phi(0,1)}$ is equivalent to an r.i.~function norm thanks to \cite{BS}*{Chapter~2, Propositions~5.10 and 5.11}. If $\phi$ is only quasiconcave,  the functional   $\|\cdot\|_{\Lambda_\phi(0,1)}$
may indeed  fail to be subadditive (see e.g.~\cite{Lorentz}). When this lack of concavity occurs,  it suffices to consider the r.i.~function norm $\|\cdot\|_{\Lambda_\phi(0,1)}$  defined with $\phi$ replaced by its least nondecreasing concave majorant $\phi_c$, which fulfills
\begin{equation}\label{EndLor_LNDCM}
\frac12 \phi_c(t) \leq \phi(t) \leq \phi_c(t) \quad \ \quad \   \text{for} \ t \in  (0,1].  	
\end{equation}  
 We shall adopt the convention that $\|\cdot\|_{\Lambda_\phi(0,1)}$ is defined according to this procedure in what follows,
whenever needed.  

The \textit{Marcinkiewicz endpoint space} $\mathsf{M}_{\phi}(0,1)$ and the \textit{Lorentz endpoint space}  $\Lambda_\phi(0,1)$ associated with a quasiconcave function $\phi$ are the corresponding rearrangement\hyp{}invariant spaces on $(0,1)$. Plainly, two  quasiconcave functions which are equivalent (up to multiplicative constants) on $(0,1)$ yield the same space  (up to equivalent norms) in each of   these  cases. In particular,  the GLZ~spaces $ L\sp{(p, \infty; \alpha, \beta)}(0,1)$ defined in Section~\ref{sec2.2} are Marcinkiewicz endpoint spaces associated with the concave function $\phi (s)=s^{\frac{1}{p}}\ell^\alpha(s)\ell\ell^\beta(s)$. 

\noindent The  fundamental functions  of the r.i.~function norms \eqref{EndMarc_no}  and  \eqref{EndLor_no} coincide with the function $\phi$, namely
\begin{equation}\label{EndFF}
\phi_{\mathsf{M}_{\phi}} =\phi   \quad \text{and} \quad   \phi_{\Lambda_{\phi}}= \phi; 
\end{equation}moreover the associate spaces of $\mathsf{M}_{\phi}(0,1)$ and    $\Lambda_\phi(0,1)$ satisfy the following properties:
\begin{equation}\label{Endsp_ASS}
\Lambda'_\phi(0,1)= \mathsf{M}_{\overbar{\phi}}(0,1) \quad \text{and} \quad \mathsf{M}_{\phi}'(0,1)= \Lambda_{\overbar{\phi}}(0,1),
\end{equation} where $\overbar{\phi}$ is the quasiconcave function defined as in \eqref{ASS_q-conc} (see e.g.~\cite{Krein}*{Chapter~2, Section~5}, \cite{Carro}*{Chapter~2, Section~4}).
Furthermore, given any rearrangement\hyp{}invariant space  $X(0,1)$,     
\begin{equation}\label{E:immEPS}
\Lambda_{\phi_{X}}(0,1) \, \hookrightarrow \, X(0,1) \, \hookrightarrow \, \mathsf{M}_{\phi_{X}}(0,1),
\end{equation}
and the spaces $\mathsf{M}_{\phi_{X}}(0,1)$ and $\Lambda_{\phi_{X}}(0,1)$ are respectively the largest and the smallest rearrangement\hyp{}invariant spaces whose fundamental function is $\phi_{X}$. This accounts for the expression \lq endpoint' which is usually attached to their names.    
 
\subsection{Spaces of Sobolev\hyp{}type and geometric properties of domains}\label{sec2.4}

Let  $\Omega$ be a bounded domain in $\R^n$, with $n \geq 2$.  Given any r.i.~function norm $\|\cdot\|_{X(0,1)}$ and $m \in \mathbb N$, the $m$\hyp{}th order Sobolev\hyp{}type space built on $X(\Omega)$, denoted by $W^mX(\Omega)$, is defined as
\begin{equation}\nonumber 
W^mX(\Omega) = \{ u  \colon   u\    \text{is} \  \text{$m$\hyp{}times}  \,  \text{weakly differentiable in}\ 
 \Omega,   \text{and} \   |\nabla^ku| \in X(\Omega) \  \text{for}  \  k \in \{0, \dots,m\}\}.
 \end{equation} 
 It is a Banach space endowed with the norm 
 \begin{equation}\label{SOB_norm}
 \| u  \|_{W^mX(\Omega)} = \sum_{k=0}^m\|\, | \nabla^k u|\, \|_{X(\Omega)} \quad \ \text{for} \ \, u\in W^m
X(\Omega). 
\end{equation}
 Here,  $\nabla^k u$  stands for the vector of all $k$\hyp{}th weak derivatives of $u$, with the convention that $\nabla^0u =u$, and $|\nabla^k u|$ for its Euclidean norm.

Sobolev embeddings for functions with unrestricted boundary values require some regularity on $\Omega$. The notions of regularity for domains employed hereafter are thus recalled, and a brief discussion of the properties that Sobolev-type spaces inherited from them is presented.

\subsubsection{Sobolev\hyp{}type spaces on John domains}\label{Sec:John} A bounded domain $\Omega$ is called a~\textit{John domain} \cites{John,Martio} if there exist a point $x_0 \in \Omega$ and a  constant  $c\in (0,1)$    such that the following internal twisted-cone condition is satisfied: for all points $x \in \Omega$ one can find a rectifiable curve $\gamma_x \colon [0, \ell_x] \to \Omega$, parametrized by  its arc\hyp{}length,   such that $\gamma_x (0)=x$, $\gamma_x (\ell_x) = x_0$ and the distance to the boundary satisfies
 \begin{equation}\label{John}{\rm dist}\, (\gamma_x (r) , \partial \Omega ) \geq c \, r \quad \ \hbox{for} \ r \in [0, \ell_x]. 
 \end{equation}  
 Roughly speaking, a John domain allows one to connect every point $x \in \Omega$ to a distinguished point $x_0 \in \Omega$ by a twisted cone whose size is suitably comparable to the Euclidean distance between $x$ and $x_0$.  John domains may have fractal boundaries as well as inward cusps, whereas outward cusps are forbidden.  
 Also, every bounded domain that satisfies the interior cone condition, and thus, in particular, any bounded Lipschitz domain,  is a John domain (see e.g.~\cite{Die}*{Section~7.4}). 

Let us stress that the notion of John domain is invariant under dilation. That is, if $\Omega$ is a  John domain, then so is $\lambda \Omega$ for all $\lambda > 0$. Therefore, as noticed above, we  will consider only John domains whose Lebesgue measure is equal to $1$, without loss of generality.  

When $\Omega$ is a John domain,~\cite{CPS}*{Theorem~6.2} tells us that the optimal (i.e.~smallest possible) rearrangement\hyp{}invariant target space in Sobolev\hyp{}type embeddings for $W^mX(\Omega)$ exists. Moreover, it is built on the r.i.~function norm $\| \cdot \|_{{X_{m, \text{opt}} (0,1)}}$  whose associate function norm is given by
\begin{equation}\label{E:Opt_S}
\|f\|_{{X_{m, \text{opt}}'(0,1)}} = \left\|s\sp{\frac mn}  f\sp{**}(s)\right\|_{{X'(0,1)}} \quad \ \text{for} \ \, f \in L^0_+ (0,1).
\end{equation}
The norm of the embedding
\begin{equation}\label{E:KP3}
W^{m}X(\Omega)\, \hookrightarrow \, X_{m, \text{opt}}(\Omega) 
\end{equation}
depends only on $n, m$ and $\Omega$. Thanks to \cite{KP1}*{Theorem~3.9}, the r.i.~function norm \eqref{E:Opt_S} admits the following (still implicit) equivalent description 
\begin{equation}\label{E:EOPJ2}
\|f\|_{{X_{m, \text{opt}}'(0,1)}} \,  \approx \, \left\|S_mf\right\|_{{X'(0,1)}} \quad \ \text{for} \ \, f \in L^0_+ (0,1),
\end{equation}  
up to multiplicative constants independent of $f$, where 
\begin{equation}\label{E:EOPJC}
S_mf = \sup_{t \in [\cdot,1)} \ t\sp{\frac mn}  f\sp{**}(t)  \quad \    \text{for} \ f \in  L^0_+(0,1).
\end{equation} 

\subsubsection{Sobolev\hyp{}type spaces on Jones domains}\label{Sec:Jones}
A domain $\Omega$ is called a~\textit{Jones domain} (or an $(\varepsilon, \delta)$\hyp{}\textit{domain})  if there exist  $\varepsilon \in (0,1)$ and $\delta \in (1, \infty]$  such that for all $x,y \in \Omega$, with $|x- y|< \delta$,  there is a rectifiable arc $\gamma$ lying in $\Omega$  and joining  $x$ to $y$ such that
 \begin{equation}\label{E:Jones}
 \ell(\gamma)   \leq   \frac{|x-y|}{\varepsilon}  \quad \ \text{and}  \quad \  {\rm dist}\, (z , \partial \Omega )   \geq  \varepsilon \, \frac{|x-z|\, |y-z|}{|x-y|}  \quad \ \text{for} \ z \in \gamma,
 \end{equation}
 where $\ell(\gamma)$ is the length of $\gamma$.  Roughly speaking, a  Jones domain enables to connect any two of its points   by a double twisted cone whose
size is suitably comparable to their Euclidean distance. 

\noindent Let us warn the reader that Jones domains are also called uniform domains when $\delta > {\rm diam} \, \Omega$ and  locally uniform domains otherwise. 

The class of Jones domains was introduced  in \cite{Jo}, and has been variously characterized (see e.g.~\cites{Die, ED, Va},  and also \cite{Daf}*{Theorem~2.8}). Each bounded Jones domain is a John domain. On the other hand, the Jones class strictly contains domains with minimally smooth boundary  \cite{Stbook}*{Chapter~6, Section~3}  and, in particular, Lipschitz domains, as observed in \cite{Jo}.   

 Jones domains shall come into play in our results of Section~\ref{SS:Holder}.
For Lipschitz domains, well-known classical results  \cites{Cal, Stbook} ($p \in (1, \infty)$ and $p\in \{1,\infty\}$, respectively) guarantee the existence of a bounded linear extension operator $\mathcal E_m\colon W^{m,p}(\Omega)   \to W^{m,p}(\R^n)$ for each $m \in \N $ and all  $p \in [1, \infty]$.  
 In \cite{Jo}, P.W.~Jones introduced the notion of $(\varepsilon, \delta)$-domain (called Jones domains after him)  to  extend that result. In  \cite{Jo}*{Theorem~1}, he   showed that,  for every $(\varepsilon, \delta)$\hyp{}domain, there always exists a bounded linear extension operator $\mathcal E_m\colon W^{m,p}(\Omega)   \to W^{m,p}(\R^n)$ for each $m \in \N $ and all  $p \in [1, \infty]$. Loosely speaking, every Sobolev function on a Jones domain is the restriction of a global Sobolev function (of the same order and exponent). 

 Let us now point out that, when $\Omega$ is a bounded Jones domain, coupling \cite{Jo}*{Theorem~1} with \cite{Monia1}*{Theorem~4.1}  yields the existence of a bounded linear extension operator $\mathcal E_m\colon W^{m}X(\Omega)   \to W^{m}X_e(\R^n)$ for each $m \in \N $ and any rearrangement\hyp{}invariant space $X(\Omega)$. Here, $X_e(\R^n)$ is an extension of $X(\Omega)$. Hence, in accordance with  the authors' remark at the beginning of \cite{Monia1}*{Section~3}, their embedding theorems for $W^mX(\Omega)$ into spaces of continuous functions -that we are going to exploit in Section~\ref{SS:Holder}-, although stated for bounded Lipschitz domains, continue to hold for general bounded Jones domains $\Omega$.
 
\subsection{H\"older  type spaces}\label{sec2.5} 
 A function $\sigma \colon (0, \infty) \to (0, \infty)$ is a {\it modulus of continuity} if it is equivalent, up to positive multiplicative constants, near $0$ to a  non\hyp{}decreasing function and such that
 \begin{equation}\label{E:modulus} 
 \lim_{r\to0^+}\sigma(r)=0 \quad \, \text{and} \quad   \, \liminf_{r\to0^+} \,  \frac{\sigma(r)}{r}>0.
 \end{equation}  
The \textit{H\"older  type space} $C^{0,   \sigma(\cdot)}(\Omega)$, associated to a  modulus of continuity  $ \sigma$,  is the non\hyp{}trivial Banach space of all functions $u\in C^{0}_b(\Omega)$ for which  the norm  
\begin{equation}\label{E:Hnorm} 
\|u\|_{C^{0,\sigma(\cdot)}(\Omega)} =  \|u\|_{L^\infty(\Omega)} +   \sup_{\substack{x,y\in\Omega,\\ x\neq y}}\frac{|u(x)-u(y)|}{\sigma(|x-y|)} 
\end{equation} 
is finite. Here, $C^{0}_b(\Omega)$ denotes the space of continuous bounded functions with the usual $L^\infty$\hyp{}norm. The classical H\"older space $C^{0, \nu}(\Omega)$ simply corresponds to the choice $ \sigma(r)= r^\nu$, for $\nu \in (0,1]$.
     
Let us emphasize that, since $\Omega$ is bounded, the behavior near $0$ is the only piece of information about a modulus of continuity  $\sigma$ which is needed for the definition of the space $C^{0, \sigma(\cdot)}(\Omega)$.
Moreover, given two moduli of continuity $\sigma, \omega \colon (0, \infty) \to (0, \infty)$, if $\sigma \lesssim \omega$ near $0$, then $C^{0, \sigma(\cdot)}(\Omega) \hookrightarrow C^{0, \omega(\cdot)}(\Omega)$, whence two moduli of continuity which are equivalent (up to multiplicative constants) near $0$ yield the same space (up to equivalent norms). 

\subsection{Morrey and Campanato  type spaces}\label{sec2.6} 

A function $\varphi\colon(0,\infty)\to(0,\infty)$  is called \textit{admissible} if
\begin{equation}\label{E:admfunc}
\inf_{r \in [a,\infty)} \varphi(r)>0\quad \ \hbox{for} \ a\in (0,\infty).
\end{equation}
The \textit{Morrey type space}~$\mathcal M^{\varphi(\cdot)}(\Omega)$, associated to an admissible function $\varphi$, is the space of all functions $u\in L^0(\Omega)$ for which  the norm
\begin{equation}\label{morreydef}
\|u\|_{\mathcal M^{\varphi(\cdot)}(\Omega)}=\sup_{B\subset \Omega} \, \frac{1}{\varphi(|B|^{\frac{1}{n}})} \, \dashint_{B}|u|\,\d x
\end{equation}
is finite, and the \textit{Campanato type space}~$\mathcal L^{\varphi(\cdot)}(\Omega)$, associated to $\varphi$, is the space of all functions  $u\in L^1_{\rm loc}(\Omega)$ such that the semi\hyp{}norm
\begin{equation}\label{E:Campsem}
|u|_{\mathcal L^{\varphi(\cdot)}(\Omega)} = \sup_{B\subset \Omega} \, \frac{1}{\varphi(|B|^{\frac{1}{n}})}\,
 \dashint_{B}|u -u_{B}|\,\d x
\end{equation}
is finite. 
Here, $B$ stands for an open ball in $\Rn$, $\dashint$ for the averaged integral, and $u_B$ denotes the average of $u$ over $B$.

\noindent Let us notice that, when $ \inf  \varphi = 0$, then $\mathcal M^{ \varphi (\cdot)} (\Omega) =  \{0\}$.
In a similar fashion,  $\mathcal L^{ \varphi(\cdot)}(\Omega)$ consists only of locally constant functions when $\liminf_{r\to0^+}  r^{-1}{\varphi(r)} =0$. Indeed, the latter equality implies that $\liminf_{r\to0^+}\varphi(r) =0$, owing to the admissibility of $\varphi$. 

A remark parallel to that made at the end of the previous subsection reads here as follows. The boundedness of $\Omega$ guarantees that the behavior near $0$ is the only piece of information about an admissible function   $\varphi$ which is needed for the definition of both the spaces $\mathcal M^{\varphi(\cdot)}(\Omega)$ and  $\mathcal L^{\varphi(\cdot)}(\Omega)$. Moreover, given two  admissible functions  $\varphi, \widehat \varphi \colon (0, \infty) \to (0, \infty)$, if $\varphi \lesssim \widehat \varphi $ near  $0$, then $\mathcal M^{\varphi(\cdot)}(\Omega) \hookrightarrow \mathcal M^{\widehat\varphi(\cdot)}(\Omega)$ and  $\mathcal L^{\varphi(\cdot)}(\Omega) \hookrightarrow \mathcal L^{\widehat\varphi(\cdot)}(\Omega)$. Consequently, two  admissible functions which are equivalent (up to multiplicative constants) near $0$ yield the same space  (up to equivalent norms) in each of these cases. 

When the admissible function $\varphi$ is a power, namely,  $\varphi (r)=r^\nu$ for some    $\nu \in \R$,  we will simply write $\mathcal M^{ \nu} (\Omega)$ and $\mathcal L^{\nu} (\Omega)$ instead of $\mathcal M^{\varphi (\cdot)} (\Omega)$ and $\mathcal L^{\varphi (\cdot)} (\Omega)$ in the above notation, respectively.
It is worth pointing out that the presence of the averaged integral on the right-hand side in \eqref{morreydef} and \eqref{E:Campsem} actually causes a shift by  $n$ in the exponent of the spaces $\mathcal M^\nu (\Omega)$ and $\mathcal L^\nu (\Omega)$ when compared with the customary notation for classical Morrey and Campanato spaces \cites{MOR, Ca0, Ca}. 
One has in fact that $\mathcal L^{\nu}(\Omega)= \mathcal M^{\nu}(\Omega)$ if $\nu<0$,  $\mathcal M^{ \nu} (\Omega)=  \{0\}$ and $\mathcal L^{\nu}(\Omega)=C^{0,\nu}(\Omega)$ if $\nu>0$, and  $\mathcal M^{0}(\Omega)=L^\infty(\Omega)$ and $\mathcal L^{0}(\Omega)={\rm BMO}(\Omega)$, where ${\rm BMO}(\Omega)$ denotes the John\hyp{}Nirenberg space of functions with bounded mean oscillation in $\Omega$.

\section{Main results}\label{S:MAIN}  

The present section is split into subsections. Each of them will be titled with the target spaces for arbitrary integer order GLZ\hyp{}Sobolev  embeddings involved. Neat explicit forms of their optimal targets are detected. Such forms shall depend  on the domain spaces, namely on whether the $L^{p, q; \alpha, \beta}$  spaces or the GLZ~spaces neighboring $L^{1}$ (i.e.~the $L^{(1,q; \alpha, \beta)}$ spaces) are dealt with. Roughly speaking, our main results shall exhibit the net smoothness of functions (with unrestricted boundary values) belonging to $m$\hyp{}th order GLZ\hyp{}Sobolev spaces on bounded domains $\Omega$ of $\R^n$, $n \geq 2$, under  minimal regularity assumptions on $\Omega$. As will be clear from our statements, this net smoothness does depend on a delicate combination of the dimension $n$ and all the parameters $m, p, q, \alpha, \beta$ defining the Sobolev\hyp{}type space in question.  
As a consequence, the gap between the embedding properties of the classical Sobolev spaces (recalled in Section~\ref{Intro}) and those of the Sobolev\hyp{}type spaces built on the broad and finely tuned class of GLZ~spaces becomes apparent, thus confirming the advantage in the use of the latter framework.  

\noindent In what follows we may always assume, without loss of generality, that the Lebesgue measure of the underlying bounded domain $\Omega$ is equal to $1$, as outlined in Section~\ref{sec2.4}.

\subsection{Embeddings into rearrangement\hyp{}invariant spaces}\label{SS:ri}

We begin by focusing on embeddings of $m$\hyp{}order GLZ\hyp{}Sobolev spaces into  rearrangement\hyp{}invariant spaces on John domains $\Omega$ of $\R^n$, $n \geq 2$, and a GLZ\hyp{}Sobolev counterpart of \eqref{CS1} is presented. As pointed out above, the two domain space cases $L^{p, q; \alpha, \beta}(\Omega)$   and $L^{(1,q; \alpha, \beta)}(\Omega)$ translate to essentially different information on the behavior of the GLZ\hyp{}Sobolev functions, and these two cases are tackled in Parts~{\sc(A)} and~{\sc(B)} of Theorem~\ref{T:OptRI}, respectively.

\noindent As mentioned in Sections~\ref{Intro} and~\ref{Sec:John}, when $\Omega$ is a John domain, the optimal target spaces in this class always exist by virtue of \cite{CPS}*{Theorem~6.2} and they are just implicitly described by \eqref{E:Opt_S}  (see e.g.~\cites{KP3, EKP, Lubos}). 
The GLZ\hyp{}Sobolev spaces provide an appropriate functional framework for their explicit description.
One of the remarkable properties of the GLZ~realm is that the  associate function norms appearing in \eqref{E:Opt_S} can be explicitly computed. 
 
Part~{\sc(A)} of Theorem~\ref{T:OptRI} ensures that the class of generalized Lorentz\hyp{}Zygmund spaces $L^{p, q; \alpha, \beta}(\Omega)$ fails to be closed under the operation of associating an optimal range in Sobolev embeddings only in the exceptional cases  $p=q=1$, $\alpha>0$ and $\beta < -\alpha$ and $p=\frac nm$, $q \in (1, \infty]$, $\alpha=\beta = \frac {1}{q'}$. Indeed, the   optimal rearrangement\hyp{}invariant target spaces for embeddings of $W^mL^{1, 1; \alpha, \beta}(\Omega)$, with $\alpha>0$ and $\beta < -\alpha$, agree (up to equivalent norms) with the  Lorentz endpoint spaces   $\Lambda_{\overbar{\phi}_{\alpha, \beta}}(\Omega)$  associated with the quasiconcave functions $\overbar{\phi}_{\alpha, \beta}$ explicitly described in \eqref{E:OptRI1bis} in the former case. In the latter ones,  the   optimal rearrangement\hyp{}invariant target spaces are    five\hyp{}parameter Lorentz\hyp{}Zygmund spaces (or, equivalently, specific instances of Lorentz\hyp{}Karamata spaces, see e.g.~\cite{ED}*{Section~3.4.3}, \cite{PLK}), having different form depending on $q$. To summarize, the fine\hyp{}grained scale of GLZ~spaces allows us to pinpoint when and how the property of its being closed under the operation of associating an optimal target in arbitrary integer order Sobolev breaks down. Loosely speaking, this happens when the GLZ~domain space is \lq very close' either to $L^1$ or to $L^{\frac nm}$ (cf.~\eqref{E:OptRI1}\textsubscript{4} or \eqref{E:OptRI1}\textsubscript{7}).
 
Next, Part~{\sc(B)}  exhibits that the optimal rearrangement\hyp{}invariant target spaces for embeddings of $W^mL^{(1, q; \alpha, \beta)}(\Omega)$ are GLZ~spaces only  when   $q=1$ and either $\alpha >-1$, $\beta \geq - \alpha -1$ or  $\alpha=-1$, $\beta < - \alpha -1$. The explicit description of the optimal rearrangement\hyp{}invariant targets for embeddings of $W^mL^{(1, 1; \alpha, \beta)}(\Omega)$ is also singled out. When $q \in (1, \infty]$,  we are able to offer the corresponding optimal GLZ~target spaces. The explicit description of the optimal (existing) rearrangement\hyp{}in\-va\-ri\-ant target spaces for $q>1$ remains an open problem.  

In the light of \eqref{CS1} and the right-hand side of \eqref{E:imm}, the embeddings in question become trivial for $m \geq n$, since any $m$\hyp{}order GLZ\hyp{}Sobolev space is embedded into $L^\infty(\Omega)$. This means that  we may here restrict ourselves to the case when $m < n$. Then, $p^{\ast}$ stands for the Sobolev conjugate of $p \in \left[1, \tfrac nm \right)$.  

\begin{thm}[Optimal embeddings into rearrangement\hyp{}invariant spaces] \label{T:OptRI}  
Let $\Omega$ be a John domain in $\mathbb R^n$, with  $n \geq 2$.

\vspace{1mm}
    
\noindent {\sc(A)}   Assume that $p, q\in[1,\infty]$, $\alpha, \beta \in \R$ fulfill one of the  alternatives~\eqref{c:LZnorm}.         
If $m \in \N$, with $m < n$, then 
\begin{equation}\label{E:OptRI1} 
W^{m}L^{p, q; \alpha, \beta}(\Omega)  \, \hookrightarrow  \,
\begin{cases} 
L^{p^\ast \!,q; \alpha , \beta}(\Omega)& \quad      \text{if} \,   \begin{cases}p =   q=1, \    \alpha =0, \    \beta \geq 0 \\    p =   q=1, \    \alpha >0, \ \beta \geq - \alpha
\\  
p \in \left(1, \frac nm\right) \!,   \    q\in[1,\infty], \     \alpha,  \beta  \in \R 
\end{cases}   
\\[0.5ex] 
\Lambda_{\overbar{\phi}_{\alpha, \beta}}(\Omega) 
 &\quad    \text{if} \ \, p = q= 1,  \   \alpha   > 0, \ \beta < - \alpha
\\[0.2ex] 
L^{\infty, q; \alpha-1 , \beta}(\Omega) 
 &\quad    \text{if} \ \, p=\frac nm,  \ q\in[1,\infty],  \   \alpha   < \frac {1}{q'}, \ \beta \in \R 
\\[0.2ex] 
 L^{\infty, q; - \frac 1q , \beta-1}(\Omega)  &\quad    \text{if} \ \, p=\frac nm, \ q\in[1,\infty],  \   \alpha   = \frac {1}{q'}, \    \beta < \frac {1}{q'} \\[0.2ex] 
 L^{\infty, q; -\frac 1q , -\frac 1q , -1}(\Omega)  &\quad   \text{if} \ \, p=\frac nm, \ q\in (1,\infty], \   \alpha   =   \beta   =\frac {1}{q'} \\[0.2ex] 
L^{\infty}(\Omega)&\quad        \text{if} \,   
\begin{cases}
p =   \frac nm, \  q=1, \  \alpha   =   \beta   = 0  \\[0.2ex]   
p =   \frac nm,   \  q\in[1,\infty], \   \alpha   = \frac {1}{q'},  \  \beta > \frac {1}{q'}   \\[0.2ex]  
p =   \frac nm, \  q\in[1,\infty],  \   \alpha   > \frac {1}{q'}, \ \beta \in \R \\[0.2ex]   
p \in \left(\frac nm, \infty\right) \!,   \    q\in[1,\infty], \     \alpha,    \beta  \in \R \\[0.2ex] 
p=\infty, \ q\in[1,\infty],     \ \alpha  < - \frac1{q}, \ \beta \in \R \\[0.2ex]
 p=\infty,   \ q\in[1,\infty],     \ \alpha   = - \frac1{q}, \      \beta    < - \frac1{q}    \\[0.2ex]
 p=   q=\infty,  \  \alpha  =   \beta    = 0.  
\end{cases}   
\end{cases}  
\end{equation}
Here, $\overbar{\phi}_{\alpha, \beta}\colon [0, 1] \to [0, \infty)$ is a  quasiconcave function obeying 
\begin{equation}\label{E:OptRI1bis}
\overbar{\phi}_{\alpha, \beta}(t) \, \approx \, t^{1- \frac mn} \left( \log^{\alpha }\left(\frac{e}{t}\right)\log\log^{\beta} \left(\frac{e}{t}\right) \chi_{(0, \xi_{\alpha, \beta}]}(t) +   \chi_{(\xi_{\alpha, \beta},1)}(t) \right)\end{equation} for $t \in [0, 1]$, where $\xi_{\alpha, \beta}=e^{1- e^{-\left(\frac{\beta}{\alpha}+1\right)}}$.

Moreover, the target space is, in each case, the optimal (smallest)   rearrangement\hyp{}invariant space which renders \eqref{E:OptRI1} true. 
 
\vspace{1mm}
    
\noindent {\sc(B)}  Assume that $q\in[1,\infty]$, $\alpha, \beta \in \R$ satisfy one of the   alternatives~\eqref{c:L1}. 
If $m \in \N$, with $m < n$, then     
\begin{equation}\label{E:OptRI2} 
W^{m}L^{(1,q; \alpha, \beta)}(\Omega)  \, \hookrightarrow  \,
\begin{cases} 
L^{1^*\!, q; \alpha + 1, \beta}(\Omega)& \quad      \text{if} \  \,    q=1, \    \alpha > -1, \    \beta \geq - \alpha -1 
\\[0.2ex] 
\Lambda_{\overbar{\phi}_{\alpha+1, \beta}}(\Omega)
 &\quad    \text{if} \ \, q= 1,  \   \alpha   > -1, \ \beta < - \alpha -1
\\[0.2ex] 
L^{1^*\!, 1; 0, \beta+1}(\Omega) 
 &\quad    \text{if} \ \, q=1,  \   \alpha   =-1, \ \beta > -1
\\[0.2ex] 
L^{1^*\!, 1; 0,0, 1}(\Omega)  &\quad   \text{if} \ \,   q=1, \   \alpha   =   \beta   =-1 \\[0.2ex]
L^{1^*\!, q; \alpha + \frac1q, \beta}(\Omega) &\quad    \text{if} \ \,   q\in (1,\infty],  \   \alpha   > - \frac {1}{q}, \    \beta \in \R   \\[0.2ex] 
L^{1^*\!, q; 0, \beta+\frac 1q}(\Omega) &\quad    \text{if} \ \,   q\in (1,\infty],  \   \alpha   = - \frac {1}{q}, \    \beta > -\frac {1}{q} \\[0.2ex] 
L^{1^*\!, q; 0,0,\frac {1}{q}}(\Omega)  &\quad   \text{if} \ \,   q\in (1,\infty), \   \alpha   =   \beta   = -\frac {1}{q} 
\end{cases}     
\end{equation}
where $\overbar{\phi}_{\alpha+1, \beta}$ is the  quasiconcave function described in \eqref{E:OptRI1bis} with $\alpha$ replaced by $\alpha+1$. 

Moreover, the target space  in \eqref{E:OptRI2}\textsubscript{1-4}  is optimal in the class of rearrangement\hyp{}invariant spaces. The target space   in \eqref{E:OptRI2}\textsubscript{5-6} is optimal in the class of GLZ~spaces of the second exponent $q$ and the target space in \eqref{E:OptRI2}\textsubscript{7} is optimal in the class of five\hyp{}parameter GLZ~spaces of the second exponent $q$.

In addition, $(L^{(1,q; \alpha, \beta)})_{m, \text{opt}}(\Omega)$ is not a GLZ~space when $q\in (1,\infty]$ and either      $\alpha   > - \frac {1}{q}$, $\beta \in \R$ or $\alpha   = - \frac {1}{q}$, $ \beta> - \frac {1}{q}$, and it is not a five\hyp{}parameter GLZ~space when $q\in (1,\infty)$ and  $\alpha   =  \beta = - \frac {1}{q}$.
\end{thm}

Part~{\sc(A)} recovers and embraces, in particular, the  improvements of the classical embeddings \eqref{CS1}\textsubscript{1,2}  of \cites{Poho, TRU, Yudo} proved by   \cites{ONeil, Peetre} ($1 \leq p=q < \frac nm$, $\alpha= \beta=0$), \cites{BW, Hansson} ($p=q= \frac nm$, $\alpha= \beta=0$) 
as well as \cite{EKP}*{Section~7} ($p=q \in \left(1, \frac nm\right)$, $\alpha= \beta= 0$, and $p=q=\frac nm$, either $\alpha<0$, $ \beta= 0$, or $ \alpha=0$, $\beta=-\frac{1}{n'}$), 
\cite{CP_trans}*{Example~5.6} (either $p=q= 1$, $\alpha \geq 0$ or $p=q> 1$, $\alpha \in \R$, and $\beta=0$), and \cite{CM}*{Theorem~5.1} ($\beta=0$). 

The conclusions that are derived in Part~{\sc(B)} are new, as far as we know, even for smooth domains in the simplest case $\beta=0$. It is worth noting here that  the embeddings  \eqref{E:OptRI2}\textsubscript{1-3} follow from \eqref{E:OptRI1}\textsubscript{2}, \eqref{E:OptRI1}\textsubscript{4} and \eqref{E:OptRI1}\textsubscript{1}, respectively, via  \eqref{c:equivLZ2a}.

\begin{remark}\label{R:OptallGLZ}  \rm  \rm When $q>1$, the target space  in \eqref{E:OptRI2}\textsubscript{5-6} [resp.~in \eqref{E:OptRI2}\textsubscript{7}] may fail to be optimal in the class of  all [five\hyp{}parameter] GLZ~spaces on $\Omega$,  whatever the John domain $\Omega$   is. This
can be verified, for instance,  by choosing $m=1$, $q=2$,  $\alpha=\beta=0$  and  letting $\Omega$ be any John domain in $\mathbb R^n$, with $n \geq 2$.
Indeed, on  the one hand,  $L^{\frac n{n-1},2; \frac12, 0}(\Omega)$    is the optimal target space for $W^{1}L^{(1,2; 0, 0)}(\Omega)$  in the class of GLZ~spaces of the second exponent $2$ by  \eqref{E:OptRI2}\textsubscript{5}. On the other hand,  combining the right-hand side of \eqref{E:imm} with  \eqref{E:OptRI1}\textsubscript{1}  yields
$$W^{1}L^{(1,2; 0, 0)}(\Omega)  \, \hookrightarrow  \, W^{1,1}(\Omega) \, \hookrightarrow  \, L^{\frac n{n-1},1}(\Omega).$$  However,  $ L^{\frac n{n-1},2; \frac12, 0}(\Omega) \not \hookrightarrow L^{\frac n{n-1},1}(\Omega) $ by   \cite{OP}*{Theorem~4.5}. Hence, the target space  in \eqref{E:OptRI2}\textsubscript{5}   fails to be optimal  in the class of all GLZ spaces on $\Omega$.  Similar arguments may be employed for the target spaces in \eqref{E:OptRI2}\textsubscript{6-7}.
\end{remark}

 \begin{remark}\label{R:OptRI}  \rm The precise description of the optimal rearrangement\hyp{}invariant target  space for $W^{m}L^{(1,q; \alpha, \beta)}(\Omega)$ for $m<n$ and $q>1$ remains a   challenging open problem. A useful piece of information is, however, actually at hand thanks to some key steps of our proof of Theorem~\ref{T:OptRI}~{\sc(B)} below (cf.~Section~\ref{S:ORI}), namely, \begin{align}
&L^{1^*\!, q;  \alpha + 1,  \beta} (\Omega)\,  \subsetneqq \, (L^{(1,q; \alpha, \beta)})_{m, \text{opt}}(\Omega) \,  \subsetneqq \, L^{1^*\!, q;  \alpha + \frac 1q,  \beta} (\Omega) && \quad    \text{if} \ \,     \alpha   > - \frac {1}{q}, \    \beta \in \R \label{XE:optemb1}\\[0.2ex]
&L^{1^*\!, q;  \frac{1}{q'}, \beta + 1} (\Omega)  \,  \subsetneqq \,(L^{(1,q; -\frac 1q, \beta)})_{m, \text{opt}}(\Omega)  \,  \subsetneqq \,   L^{1^*\!, q; 0, \beta +   \frac 1q} (\Omega)  
&& \quad    \text{if} \ \,      \beta  > - \frac {1}{q} \label{XE:optemb2}\\[0.2ex]
 &L^{1^*\!, q;   \frac{1}{q'},  \frac{1}{q'}, 1} (\Omega) \,  \subsetneqq \, (L^{(1,q; -\frac 1q, -\frac 1q)})_{m, \text{opt}}(\Omega)  \,  \subsetneqq \,  L^{1^*\!, q; 0, 0,   \frac 1q} (\Omega)  && \quad    \text{if} \ \, q \in (1, \infty),\label{XE:optemb3}
\end{align} and each space of the  rightmost side of \eqref{XE:optemb1}-\eqref{XE:optemb3} is optimal in the relevant scale for the latter inclusion to hold.
\end{remark} 

\subsection{Embeddings into    spaces of continuous functions}\label{SS:Holder}

In this section, in contrast with the others, we will assume that the  domain $\Omega$ fits in the Jones class. The assumption that $\Omega$ is merely a John domain is not sufficient for Sobolev embeddings into spaces of continuous functions to hold, even classically (see e.g.~\cite{MaPobook}). With regard to this, consider, for example, $\Omega=\{(\rho,\theta)\colon \rho \in (1,2), \theta \in (0,2\pi)\}$ and take the function $u(\rho,\theta)=\theta$ for $(\rho,\theta) \in \Omega$. Then, $u\in W^{m,p} (\Omega)$ for all $m \in \N$ and $p \in [1, \infty]$, but $u \not \in C^{0}_b(\Omega)$. Here, $\Omega$ satisfies the cone property (hence it is a John domain), but it does not belong to the Jones class. 

As we pointed out in Section~\ref{Sec:Jones}, the geometry of bounded Jones domains $\Omega$ is \lq good' enough to assure a version of the Jones extension theorem for functions belonging to the Sobolev\hyp{}type space $W^{m}X(\Omega)$ for each $m \in \N $ and any rearrangement\hyp{}invariant space $X(\Omega)$. Hence, all the main results of \cite{Monia1}*{Section~3} that we are going to exploit in the present section, although stated for bounded Lipschitz domains, continue to hold for general bounded Jones domains $\Omega$.

First, we seek continuity properties of $m$\hyp{}order GLZ\hyp{}Sobolev functions. The only case of interest is when $m$   is strictly less than the dimension $n$ of the Euclidean space. Indeed,  both $W^{m}L^{p, q; \alpha, \beta}(\Omega)$ and $W^{m}L^{(1, q; \alpha, \beta)}(\Omega)$ embed  in $C^{0}_b(\Omega)$ for $m \geq n$, on account of   \eqref{CS2} and the second embedding in \eqref{E:imm}. When $m<n$, then
 \cite{Monia1}*{Theorem~3.1} tells us that the embedding $W^{m}X (\Omega)  \hookrightarrow  C^{0}_b(\Omega)$  holds for some rearrangement\hyp{}invariant space $X (\Omega)$ if and only if $X(\Omega)  \hookrightarrow L^{\frac nm,1}(\Omega)$. Henceforth, well\hyp{}known embeddings among GLZ~spaces \cite{OP}*{Theorems~4.5,~4.6 and~5.5} provide us with the following sharp complete answer to the problem in question.
 
\begin{thm}[Embeddings into $C^{0}_b(\Omega)$] \label{T:C0}  
Let $\Omega$ be a  bounded Jones domain in $\mathbb R^n$, with $n \geq 2$, and $m \in \N$, with $m <n$. 
       
\vspace{1mm}
    
\noindent {\sc(A)}    Assume that  $p, q\in[1,\infty]$, $\alpha, \beta \in \R$ fulfill one of the alternatives~\eqref{c:LZnorm}. Then,  
\begin{equation}\label{E:C0p}
W^{m}L^{p, q; \alpha, \beta}  (\Omega)  \, \hookrightarrow \, C^{0}_b(\Omega)  \quad \text{if and only if} \quad  \begin{cases}  
p=\frac nm, \   q=1, \    \alpha  =   \beta = 0    \\[0.2ex]       p=\frac nm, \  q\in[1,\infty], \  \alpha   = \frac{1}{q'}, \   \beta   > \frac{1}{q'}  \\[0.2ex] 
p=\frac nm,  \  q\in[1,\infty], \    \alpha   > \frac{1}{q'}, \ \beta \in \R  \\[0.2ex]    
p \in \left(\frac nm, \infty\right]  
\end{cases} 
\end{equation} 
   
\vspace{1mm}
      
\noindent {\sc(B)}   If $ q\in [1,\infty]$, $\alpha, \beta \in \R$ satisfy one of the alternatives~\eqref{c:L1}, then 
\begin{equation}\label{E:C0*}
W^{m}L^{(1, q; \alpha, \beta)} (\Omega)  \, \not \hookrightarrow \, C^{0}_b(\Omega). 
\end{equation} 
\end{thm}
 
Next, we provide the explicit form of the (existing) optimal modulus of continuity for Sobolev\hyp{}type embedding  \eqref{EmbHolder} in Section~\ref{Intro} within the GLZ~class. This optimal modulus  of continuity takes a distinct form according to whether $m=1$, $m \in \{2, \dots,  n-1\}$ and $m=n$.   Plainly, such a problem is relevant only for $m \leq n$, since functions in $ W^{m}L^{p, q; \alpha, \beta}(\Omega)$ and $ W^{m}L^{(1, q; \alpha, \beta)}(\Omega)$ are  Lipschitz for $m > n$, as a consequence of \eqref{CS3}\textsubscript{3} and the second embedding in \eqref{E:imm}. 

The explicit optimal $L^{p, q;\alpha, \beta}$\hyp{}Sobolev counterpart of \eqref{CS3} is the content of the following theorem.  
  
\begin{thm}[Optimal embeddings into H\"older  spaces - case $L^{p, q;\alpha, \beta}$] \label{T:HOp}  
Let  $\Omega$ be a bounded Jones domain in $\mathbb R^n$, with $n \geq 2$. Assume that $p,q\in[1,\infty]$, $\alpha, \beta \in \R$     fulfill one of the   alternatives~\eqref{c:LZnorm}.
       
\vspace{1mm}
     
\noindent {\sc(I)}   The embedding      
\begin{equation}\label{E:HOpA} 
W^{1}L^{p, q; \alpha, \beta}  (\Omega) \, \hookrightarrow \, C^{0, \widehat \sigma_1(\cdot)}(\Omega) 
\end{equation} 
holds, where  $\widehat \sigma_1\colon  (0, \infty) \to (0, \infty)$ obeys
\begin{align}\label{E:HOpAa} 
\widehat \sigma_1(r)  \approx 
\begin{cases}
\log\log^{  \frac {1}{q'} - \beta} \! \left(\frac{e}{r}\right) 
 &\quad  \text{if} \ \,p=n, \   q\in[1,\infty], \        \alpha  = \frac {1}{q'}, \   \beta  >  \frac {1}{q'} \\[0.2ex]  
 \log^{\frac {1}{q'}  -\alpha }\left(\frac{e}{r}\right)\log\log^{- \beta} \left(\frac{e}{r}\right) 
 &\quad  \text{if} \ \,  p=n,  \   q\in[1,\infty], \    \alpha   > \frac {1}{q'}, \ \beta \in \R  \\[0.2ex]
 r^{1-\frac np}\log^{ -\alpha} \!\left(\frac{e}{r}\right)\log\log^{- \beta} \!\left(\frac{e}{r}\right) 
 &\quad  \text{if} \ \,  p \in (n, \infty), \   q\in[1,\infty], \    \alpha,  \beta \in \R \\[0.2ex]  
  r\log^{-\frac 1q -\alpha} \!\left(\frac{e}{r}\right)\log\log^{- \beta} \! \left(\frac{e}{r}\right)
& \quad      \text{if} \ \, p=\infty,  \ q \in [1, \infty], \ \alpha   < -\frac 1q, \   \beta \in \R  \\[0.2ex] 
r \log\log^{- \frac 1q - \beta} \!\left( \frac{e}{r}\right) 
& \quad   \text{if} \  \ p=\infty,   \   q\in[1,\infty], \ \alpha   =- \frac{1}{q}, \     \beta < - \frac{1}{q} \\[0.2ex] 
 r 
& \quad  \text{if} \ \,p=q=\infty,   \   \alpha  =   \beta   = 0   
\end{cases} 
\end{align}
near $0$. 
Moreover, the function  $\widehat \sigma_1$  is (up to multiplicative constants) the optimal modulus of continuity in the embedding~\eqref{E:HOpA}, in the sense that, if there exists a modulus of continuity  $ \sigma$ such that   $ W^{1}L^{p, q; \alpha, \beta}  (\Omega) \hookrightarrow C^{0,  \sigma(\cdot)}(\Omega)$, then $C^{0,  \widehat \sigma_1(\cdot)}(\Omega) \hookrightarrow  C^{0, \sigma(\cdot)}(\Omega)$. 

\noindent In addition, no embedding of $W^{1}L^{p, q; \alpha, \beta}  (\Omega)$ into spaces of the form $C^{0, \sigma(\cdot)}(\Omega)$  holds for the values of the parameters $p, q, \alpha, \beta$ not included in the formula~\eqref{E:HOpAa}.

 \vspace{1mm}
 
\noindent {\sc(II)}   If   $m \in \{2, \dots,  n-1\}$, then  
\begin{equation}\label{E:HOpB} 
W^{m}L^{p, q; \alpha, \beta}  (\Omega) \, \hookrightarrow \, C^{0, \widehat \sigma_m(\cdot)}(\Omega),
\end{equation}
where $\widehat \sigma_m\colon    (0, \infty) \to (0, \infty)$ obeys
\begin{align}\label{E:HOpBb} 
\widehat \sigma_m(r)  \, \approx \, 
\begin{cases}
\log\log^{\frac {1}{q'} - \beta} \! \left(  \frac{e}{r}\right) 
&\quad  \text{if} \ \, p=\frac nm, \   q\in[1,\infty], \  \alpha  = \frac {1}{q'}, \    \beta  >  \frac {1}{q'}\\[0.2ex]
\log^{\frac {1}{q'}  -\alpha }\! \left(\frac{e}{r}\right)   \log\log^{- \beta}\!\left(\frac{e}{r}\right) 
 &\quad   \text{if} \ \, p=\frac nm,  \   q\in[1,\infty], \    \alpha   > \frac {1}{q'}, \ \beta \in \R   \\[0.2ex] 
r^{m-\frac np}\log^{ -\alpha} \! \left(\frac{e}{r}\right)   \log\log^{- \beta} \! \left( \frac{e}{r}\right) 
 &\quad   \text{if} \ \,  p \in \left(\frac nm, \frac n{m-1}\right)\! , \   q\in[1,\infty], \    \alpha,     \beta \in \R  \\[0.2ex] 
r \, \log^{\frac {1}{q'} -\alpha} \! \left(\frac{e}{r}\right)   \log\log^{- \beta} \! \left(  \frac{e}{r}\right) 
 &\quad   \text{if} \ \,  p = \frac n{m-1}, \   q\in[1,\infty], \    \alpha  <\frac {1}{q'}, \    \beta \in \R
 \\[0.2ex] 
  r \, \log  \log^{\frac {1}{q'}- \beta} \! \left(  \frac{e}{r}\right) &\quad   \text{if} \ \,  p = \frac n{m-1}, \   q\in[1,\infty], \    \alpha  =\frac {1}{q'}, \   \beta  < \frac {1}{q'} \\[0.2ex] 
  r  \, \log\log\log^{\frac {1}{q'}} \! \left( \frac{e}{r}\right) 
& \quad     \text{if} \ \,  p = \frac n{m-1}, \   q\in (1,\infty], \    \alpha  =   \beta  = \frac {1}{q'} \\[0.2ex] 
r  
& \quad     \text{if}  \,  \begin{cases}p = \frac n{m-1}, \   q=1, \    \alpha  =   \beta  = 0  \\[0.2ex]
p = \frac n{m-1}, \    q\in[1,\infty], \     \alpha =\frac {1}{q'}, \    \beta  >  \frac {1}{q'} \\[0.2ex]
p = \frac n{m-1}, \    q\in[1,\infty], \     \alpha  >\frac {1}{q'}, \    \beta  \in \R\\[0.2ex]
p \in \left(\frac n{m-1}, \infty\right) \!, \    q\in[1,\infty], \     \alpha,  \beta \in \R \\[0.2ex]
p=\infty, \ q\in[1,\infty], \ \alpha  < - \frac1{q}, \ \beta \in \R \\[0.2ex]   
p=\infty,   \ q\in[1,\infty], \ \alpha   = - \frac1{q}, \      \beta    < - \frac1{q}    \\[0.2ex]
p=   q=\infty,  \  \alpha  =   \beta    = 0.  
\end{cases}  
\end{cases} 
\end{align}
near $0$.  
Moreover, the function $\widehat \sigma_m$ is (up to multiplicative constants) the optimal modulus of continuity in the embedding~\eqref{E:HOpB}.

\noindent In addition, no embedding of $W^{m}L^{p, q; \alpha, \beta}(\Omega)$ into spaces of the form $C^{0,\sigma(\cdot)}(\Omega)$ holds for the values of the parameters $p, q, \alpha, \beta$ not included in the formula \eqref{E:HOpBb}.

 \vspace{1mm}
 
\noindent {\sc(III)}   The embedding     
\begin{equation}\label{E:HOpC} 
W^{n}L^{p, q; \alpha, \beta}  (\Omega) \, \hookrightarrow \, C^{0, \widehat \sigma_n(\cdot)}(\Omega) 
\end{equation} 
holds, where $\widehat \sigma_n\colon    (0, \infty) \to (0, \infty)$ obeys
\begin{equation}\label{E:HOpCc} 
\widehat \sigma_n(r) \,  \approx \, 
\begin{cases}            
r^{n-\frac np}   \log^{-\alpha} \! \left(\frac{e}{r}\right)   \log\log^{- \beta} \!  \left(\frac{e}{r}\right) 
&  \quad \   \text{if}     \, \begin{cases}  p= q=1, \    \alpha=0,\  \beta >0 \\  p= q=1, \       \alpha>0, \ \beta \in \R\\    p \in  \left(1, \frac n{n-1} \right) \!, \ q\in[1,\infty], \    \alpha,   \beta \in \R  
\end{cases}     \\[0.2ex] 
r \, \log^{\frac {1}{q'} -\alpha} \! \left(\frac{e}{r}\right)       \log\log^{- \beta} \!   \left(\frac{e}{r}\right)
& \quad \ \text{if} \ \,   p=\frac n{n-1}, \ q\in[1,\infty], \ \alpha   < \frac {1}{q'}, \  \beta \in \R      
 \\[0.2ex] 
r \,   \log\log^{\frac {1}{q'}   -\beta} \!   \left(\frac{e}{r}\right) 
&\quad \ \text{if} \ \,   p=\frac n{n-1}, \ q\in[1,\infty], \   \alpha   = \frac {1}{q'}, \      \beta  <    \frac {1}{q'}  \\[0.2ex] 
r \,   \log\log\log^{\frac {1}{q'}} \! \left(  \frac{e}{r}\right) &\quad \ \text{if} \ \, p=\frac n{n-1}, \  q \in (1, \infty], \  \alpha  =   \beta = \frac {1}{q'}    \\[0.2ex] 
 r  &\quad \ \text{if}    \, \begin{cases}p=\frac n{n-1}, \    q=1, \    \alpha =    \beta    =   0 \\ 
 p=\frac n{n-1}, \ q\in[1,\infty], \  \alpha  =\frac {1}{q'},   \   \beta  > \frac {1}{q'}   \\[0.2ex]   
 p=\frac n{n-1}, \ q\in[1,\infty], \    \alpha  >\frac {1}{q'}, \ \beta \in \R \\[0.2ex] 
 p \in \left(\frac n{n-1}, \infty\right) \!,  \ q\in[1,\infty], \    \alpha,    \beta \in \R  \\[0.2ex] 
 p=\infty, \ q\in[1,\infty],  \ \alpha  < - \frac1{q}, \ \beta \in \R \\[0.2ex]
 p=\infty,   \ q\in[1,\infty],  \ \alpha   = - \frac1{q}, \      \beta    < - \frac1{q}    \\[0.2ex]  
 p=   q=\infty,  \  \alpha  =   \beta    = 0  
\end{cases}  
\end{cases} 
\end{equation}
near $0$. 
Moreover, the function $\widehat \sigma_n$  is (up to multiplicative constants) the optimal modulus of continuity in the embedding~\eqref{E:HOpC}. 
\end{thm}

The results of Theorem~\ref{T:HOp} encompass, generalize and/or improve, in particular, the classical embeddings \eqref{CS3}, 
 \cite{EGO3}*{Theorem~4.9}, \cite{EGO4}*{Theorem~3.2}  ($m= 1$, $p\in (n, \infty)$ and $2 \leq m\leq n$, $ p \in \left(\frac{n}{m}, \frac{n}{m-1}\right)$, $q \in (1, \infty)$),  \cite{EGO3}*{Theorem~4.11}, \cite{EGO4}*{Theorem~3.3} ($m> 1$, $p= \frac{n}{m-1}$, $q \in (1, \infty)$, either $\alpha < \frac{1}{q'}$, $\beta = 0$ or  $\alpha = \frac{1}{q'}$ and $\beta < \frac{1}{q'}$), and
\cite{Monia1}*{Example~6.7} (either $p=q= 1$, $\alpha \geq 0$ or $p=q> 1$, $\alpha \in \R$, and $\beta=0$). We also refer  the interested reader  to~\cite{EGO5}*{Theorem~2}, \cite{GNO}*{Theorem~3}, and \cite{GNO2}*{Corollary~6.2}.

\begin{remark}\label{SiCnoH} \rm \noindent  Theorem~\ref{T:C0} tells us that, for any bounded Jones domain  $\Omega$ in $\mathbb R^n$, with $n \geq 2$, and any $m \in \N$, with $m \leq n$, functions in  $W^{m}L^{p,q; \alpha, \beta} (\Omega)$   (or pointwise a.e.~equivalent versions of them) are bounded continuous functions on $\Omega$ provided the parameters fulfill one of the alternatives on the right\hyp{}hand side of \eqref{E:C0p}. When such a condition is in force, Theorem~\ref{T:HOp} clarifies that no embedding of $W^{m}L^{\frac nm,1} (\Omega)$ into spaces of the form $C^{0,\sigma(\cdot)}(\Omega)$ holds whatever the modulus of continuity $\sigma$ is.
\end{remark}

In view of Theorem~\ref{T:C0}, it is apparent that the problem of $L^{(1, q; \alpha, \beta)}$\hyp{}Sobolev embeddings into  H\"older type spaces has to be faced  only when $m=n$. It is addressed in the next result, where the optimal modulus of continuity is explicitly identified as well. Let us stress in advance that only in the case $q=1$ conclusions \eqref{E:HO*B}\textsubscript{1,2} may be deduced from the parallel ones \eqref{E:HOpCc}\textsubscript{1,2} via the equation  \eqref{c:equivLZ2a}. The others are new even in the case of smooth domains.

 \begin{thm}[Optimal embeddings into H\"older  spaces - case $L^{(1,q; \alpha, \beta)}$] \label{T:HO*} 
Let  $\Omega$ be a bounded Jones domain in $\mathbb R^n$, with $n \geq 2$. Assume that $q\in[1,\infty]$, $\alpha, \beta \in \R$ satisfy one of the   alternatives~\eqref{c:L1}.
Then,  
\begin{equation}\label{E:HO*A}
W^{n}L^{(1,q; \alpha, \beta)} (\Omega) \, \hookrightarrow \, C^{0, \widehat \omega_n(\cdot)}(\Omega), 
\end{equation}  
where $\widehat \omega_n\colon    (0, \infty) \to (0, \infty)$ obeys
\begin{equation} \label{E:HO*B} 
\widehat \omega_n(r) \,  \approx \,  \,
\begin{cases} 
\log^{-\frac 1q - \alpha} \! \left(\frac{e}{r}\right)   \log\log^{- \beta} \!  \left(\frac{e}{r}\right)  &\quad   \text{if} \ \, q\in[1,\infty], \      \alpha > - \frac{1}{q}, \ \beta\in \R \\[0.2ex] 
\log\log^{-\frac 1q- \beta} \!  \left(\frac{e}{r}\right)  &\quad    \text{if} \ \,  q\in[1,\infty], \   \alpha   = - \frac{1}{q}, \    \beta  > - \frac{1}{q}  \\[0.2ex] 
\log\log\log^{-\frac 1q} \! \left( \frac{e}{r}\right)  &\quad    \text{if} \ \,   q\in[1,\infty), \ \alpha  =   \beta  = - \frac{1}{q}   
\end{cases}  
\end{equation} 
near $0$.  
Moreover, the function $\widehat \omega_n$  is (up to multiplicative constants) the optimal modulus of continuity in the embedding~\eqref{E:HO*A}. 
\end{thm}
 
 \subsection{Embeddings into Morrey and Campanato type spaces}\label{SS:MorreyCamp}

In the present section  the explicit forms of the (existing) optimal admissible functions in the embeddings \eqref{EmbMor} and \eqref{EmbCam} in Section~\ref{Intro} for GLZ~spaces are identified.

We deal first with $m$\hyp{}order GLZ\hyp{}Sobolev  embeddings into Morrey  spaces.  They are special instances of the general result  \cite{CCPS2}*{Theorem~2.2}, which we recall at the beginning of Section~\ref{P:MorreyCamp}.

\noindent We may (and will) assume that the order $m$ of the Sobolev\hyp{}type space is strictly less than the dimension of the underlying domain, without loss of generality. Indeed, as pointed out at the beginning of Section~\ref{SS:ri} and at the end of Section~\ref{sec2.5}, respectively, any $m$\hyp{}order GLZ\hyp{}Sobolev space on a John domain $\Omega$  embeds into $L^\infty(\Omega)$ when $m \geq n$, and $L^\infty(\Omega)=\mathcal M^{0}(\Omega)$.  

\begin{thm}[Optimal embeddings into Morrey   spaces] \label{T:Mor}   
Let $\Omega$ be a John domain in $\mathbb R^n$, with $n \geq 2$, and $m \in \N$, with $m < n$. 
       
\vspace{1mm}
    
\noindent {\sc(A)}   Assume that $p, q\in[1,\infty]$, $\alpha, \beta \in \R$ fulfill one of the alternatives~\eqref{c:LZnorm}.  Then,  
\begin{equation}\label{E:MOpA}
W^{m}L^{p, q; \alpha, \beta}  (\Omega)  \, \hookrightarrow \, \mathcal M^{\widetilde\varphi_m(\cdot)}(\Omega) ,
\end{equation} 
where $\widetilde\varphi_m\colon(0,\infty)\to(0,\infty)$ obeys
\begin{equation}\label{E:MOpAa}
\widetilde\varphi_m(r)  \approx 
\begin{cases}
r^{m-\frac{n}{p}} \log^{- \alpha} \! \left(\frac{e}{r}\right) \log\log^{- \beta} \! \left( \frac{e}{r}\right) 
&\quad    \text{if} \, \begin{cases} p= q= 1, \  \alpha=0, \  \beta \geq 0  \\
p= q=1,   \ \alpha>0, \  \beta \in \R 
\\[0.2ex] 
p \in \left(1, \frac nm\right) \!, \ q\in[1,\infty], \ \alpha, \beta \in \R
\end{cases}    \\[0.2ex] 
\log^{\frac {1}{q'} -\alpha}\! \left(\frac{e}{r}\right)   \log\log^{- \beta} \! \left( \frac{e}{r}\right) 
&\quad    \text{if} \ \, p=\frac nm,  \ q\in[1,\infty],  \   \alpha   < \frac {1}{q'}, \ \beta \in \R  \\[0.2ex] 
\log\log^{\frac {1}{q'} -\beta} \! \left(  \frac{e}{r}\right) &\quad    \text{if} \ \, p=\frac nm, \ q\in[1,\infty], \ \alpha   = \frac {1}{q'},   \  \beta < \frac {1}{q'}   \\[0.2ex] 
\log\log\log^{\frac {1}{q'}} \!   \left(\frac{e}{r}\right) 
&\quad    \text{if} \  \, p=\frac nm,   \ q \in (1, \infty], \ \alpha =   \beta = \frac {1}{q'}  \\[0.2ex] 
1 &\quad    \text{if} \, \begin{cases}p=\frac nm,  \ q=1, \  \alpha= \beta=0 \\ p=\frac nm,
\ q\in[1,\infty], \  \alpha =\frac {1}{q'}, \   \beta >  \frac {1}{q'}   \\[0.2ex]
p=\frac nm,    \ q\in[1,\infty],  \   \alpha   > \frac {1}{q'},\  \beta \in \R   \\[0.2ex]
p\in \left(\frac nm, \infty\right) \! , \ q\in[1,\infty], \ \alpha,   \beta \in \R  \\[0.2ex]
p=\infty,   \ q\in[1,\infty], \ \alpha    < - \frac1{q} , \ \beta \in \R     \\[0.2ex]
p=\infty,  \ q\in[1,\infty], \ \alpha   = - \frac1{q}, \     \beta  <   - \frac1{q}   \\[0.2ex] 
p=   q=\infty,  \  \alpha=  \beta=0 
\end{cases} 
\end{cases} 
\end{equation}
near $0$.
Moreover, the function $ \widetilde\varphi_m$  is (up to multiplicative constants) the optimal admissible function in the embedding~\eqref{E:MOpA},  in the sense that, if there exists an admissible function $\varphi$ such that $W^{m}L^{p, q; \alpha, \beta}  (\Omega) \hookrightarrow \mathcal M^{\varphi(\cdot)}(\Omega)$, then $\mathcal M^{ \widetilde\varphi_m(\cdot)}(\Omega) \hookrightarrow \mathcal M^{\varphi(\cdot)}(\Omega)$. 
  
\vspace{1mm}
      
\noindent {\sc(B)}    Assume that $q\in [1,\infty]$, $\alpha, \beta \in \R$ satisfy one of the alternatives~\eqref{c:L1}.
Then, 
\begin{equation}\label{E:MOpB}
W^{m}L^{(1, q; \alpha, \beta)}  (\Omega) \, \hookrightarrow \,  \mathcal M^{\widetilde\psi_m(\cdot)}(\Omega) ,
\end{equation}
where $\widetilde\psi_m\colon(0,\infty)\to(0,\infty)$ obeys
\begin{equation}\label{E:MOpBb}
\widetilde\psi_m(r) \,  \approx \, 
\begin{cases}
r^{m -n}  \log^{-\frac{1}{q}-\alpha} \! \left(\frac{e}{r}\right)   \log\log^{-\beta}\!\left(\frac{e}{r}\right) 
&\quad    \text{if} \ \, q\in[1,\infty], \ \alpha >- \frac1{q}, \ \beta \in \R  \\[0.2ex] 
r^{m -n}   \log\log^{-\frac{1}{q}-\beta} \!  \left( \frac{e}{r}\right) &\quad   \text{if} \ \, 
 q\in[1,\infty], \   \alpha   = - \frac1{q}, \    \beta >  - \frac1{q}    \\[0.2ex] 
r^{m -n}  \log\log  \log^{-\frac 1q} \! \left( \frac{e}{r}\right)  &\quad   \text{if} \ \, q \in [1, \infty), \   \alpha= \beta = - \frac{1}{q}
\end{cases} 
\end{equation}
near $0$.  
Moreover, the function $\widetilde\psi_m$  is (up to multiplicative constants) the optimal admissible function in the embedding~\eqref{E:MOpB}.    
\end{thm} 
 In the case when either $p \in (1, \infty)$, $\alpha \in \R$, $\beta = 0$ or $p=1$, $\alpha\geq 0$, $\beta = 0$, our conclusions in \eqref{E:MOpAa} are the content of \cite{CCPS2}*{Corollary~2.11}. The results of Part~{\sc(B)} are actually new even for more regular domains, and only for $q=1$ conclusions \eqref{E:MOpBb}\textsubscript{1,2} can be derived from the parallel ones \eqref{E:MOpAa}\textsubscript{1,2} via the equation  \eqref{c:equivLZ2a}.

\vspace{2mm}
Next, we focus on $m$\hyp{}order GLZ\hyp{}Sobolev embeddings into Campanato spaces, which are special instances of the general result  \cite{CCPS2}*{Theorem~2.6}  recalled in Section~\ref{P:MorreyCamp}, p.~\pageref{E:optCa}. 
Distinguishing between the cases when the spaces $L^{p, q; \alpha, \beta}(\Omega)$ or the spaces $L^{(1,q; \alpha, \beta)}(\Omega)$ are concerned, we provide the exact description of the (existing) optimal admissible function $ \widehat \varphi_{m}$ [resp.~$\widehat \psi_{m}$] in the embedding \eqref{EmbCam}. It depends -in analogy with the parallel embeddings into H\"older type spaces- on whether $m=1$, $m \in \{2, \dots,  n-1\}$ and $m=n$ [resp.~$m=1$ and $m \in \{2, \dots,  n\}$].

\begin{thm}[Optimal embeddings into  Campanato spaces -  case $L^{p, q;\alpha, \beta}$]\label{T:CampP}   Let $\Omega$ be a John domain in $\mathbb R^n$, with $n \geq 2$, and $m \in \N$, with $m \leq n$. 
Assume that $p, q\in[1,\infty]$, $\alpha, \beta \in \R$ fulfill one of the   alternatives~\eqref{c:LZnorm}.
       
\vspace{1mm}
     
\noindent {\sc(I)}   The embedding 
\begin{equation}\label{E:CaPA}
W^1 L^{p, q; \alpha, \beta}(\Omega) \, \hookrightarrow \,  \mathcal L^{\widehat \varphi_{1}(\cdot)}(\Omega) 
\end{equation}
holds, where $\widehat \varphi_{1} \colon(0,\infty)\to(0,\infty)$ obeys 
\begin{equation}\label{E:CaPAa}
\widehat \varphi_{1}(r) \, \approx \,
\begin{cases}
r^{1 -\frac{n}{p}} \log^{- \alpha}\!\left(\frac{e}{r}\right) \log\log^{- \beta}\! \left(\frac{e}{r}\right) &\quad    \text{if}   \,  \begin{cases}p =   q=1, \    \alpha =0, \  \beta \geq 0 \\[0.2ex]   p =   q=1, \    \alpha >0, \ \beta \in \R\\[0.2ex]  
p \in  (1, n],   \    q\in[1,\infty], \     \alpha,    \beta  \in \R
\end{cases}   \\[0.5ex] 
 \widehat \sigma_1(r)   &\quad  \text{if} \ \, p \in (n, \infty]  
\end{cases} 
\end{equation}

\noindent near $0$, and $\widehat \sigma_1$ is as in \eqref{E:HOpAa}. 
Moreover, the function $\widehat \varphi_{1}$ is (up to multiplicative constants) the optimal admissible function in the embedding~\eqref{E:CaPA}, in the sense that, if there exists an admissible function  $\varphi$ such that  
$W^{1}L^{p, q; \alpha, \beta}(\Omega) \hookrightarrow \mathcal L^{\varphi(\cdot)}(\Omega)$, then $\mathcal L^{\widehat \varphi_{1}(\cdot)}(\Omega) \hookrightarrow \mathcal L^{\varphi(\cdot)}(\Omega)$. 

\vspace{1mm}

\noindent {\sc(II)}    If  $m \in \{2, \dots,  n-1\}$, then       
\begin{equation}\label{E:CaPB}
W^m L^{p, q; \alpha, \beta}(\Omega) \, \hookrightarrow \,  \mathcal L^{\widehat \varphi_m(\cdot)}(\Omega), 
\end{equation}
where $\widehat \varphi_m\colon(0,\infty)\to(0,\infty)$ obeys
\begin{equation}\label{E:CaPBb}
\widehat \varphi_m(r)  \, \approx \,
\begin{cases}
r^{m -\frac{n}{p}} \log^{- \alpha} \! \left(\frac{e}{r}\right) \log\log^{- \beta} \!   \left(\frac{e}{r}\right) 
&\quad    \text{if} \, \begin{cases} p= q= 1, \  \alpha=0,  \ \beta\geq 0  \\[0.2ex] 
p= q=1,   \ \alpha>0, \ \beta \in \R 
\\[0.2ex] 
p \in \left(1, \frac{n}{m}\right] \!, \   q   \in[1,\infty], \ \alpha, \beta \in \R
\end{cases}  \\[0.5ex] 
\widehat \sigma_{m}(r) &\quad    \text{if} \ \, p\in \left(\frac{n}{m}, \infty\right]
\end{cases} 
\end{equation}

\noindent near $0$, and $\widehat \sigma_m$ is as in \eqref{E:HOpBb}.  
Moreover, the function $\widehat \varphi_{m}$  is (up to multiplicative constants) the optimal admissible function in the embedding~\eqref{E:CaPB}. 
 
\vspace{1mm}
 
\noindent {\sc(III)} The embedding 
\begin{equation}\label{E:CaPC}
W^{n}L^{p, q; \alpha, \beta}  (\Omega) \, \hookrightarrow \, \mathcal L^{\widehat \varphi_n(\cdot)}(\Omega)
\end{equation} 
holds,  where  $\widehat \varphi_n\colon    (0, \infty) \to (0, \infty)$ obeys 
\begin{equation}\label{E:CaPCc}
\widehat \varphi_n  (r) \, \approx \,   \begin{cases} 1 &\quad    \text{if} \ \,   p= q= 1, \  \alpha= \beta= 0  \\\widehat \sigma_n(r)  &\quad \text{otherwise}  
\end{cases} 
\end{equation} 
near $0$, with $\widehat \sigma_n$ being as in   \eqref{E:HOpCc}.
Moreover, the function $\widehat \varphi_{n}$  is (up to multiplicative constants) the optimal admissible function in the embedding~\eqref{E:CaPC}. 
\end{thm} 

In the case when either $p \in (1, \infty)$, $\alpha \in \R$, $\beta = 0$ or $p=1$, $\alpha\geq 0$, $\beta = 0$, the  conclusions above are the content of \cite{CCPS2}*{Corollary~2.14} for $k=0$.

We conclude this section with new explicit embeddings into Campanato type spaces concerning Sobolev spaces built upon members of the $L^{(1, q; \alpha, \beta)}$~class.

\begin{thm}[Optimal embeddings into   Campanato spaces - case   $L^{(1, q; \alpha, \beta)}$] \label{T:Ca*}    
Let $\Omega$ be a John do\-main in $\mathbb R^n$, with $n \geq 2$, and $m \in \N$, with $m \leq n$.  
Assume that $q\in[1,\infty]$, $\alpha, \beta \in \R$ satisfy one of the alternatives~\eqref{c:L1}. 
   
\vspace{1mm}
     
\noindent {\sc(I)}   The embedding
\begin{equation}\label{E:Ca*A}
W^1 L^{(1, q; \alpha, \beta)}(\Omega)  \, \hookrightarrow \, \mathcal L^{ \widehat \psi_{1}(\cdot)}(\Omega) 
\end{equation}
holds, where $\widehat\psi_{1} \colon(0,\infty)\to(0,\infty)$ obeys
\begin{equation}\label{E:Ca*Aa}
\widehat \psi_{1}(r)  \, \approx \,
\begin{cases}
r^{1 - n} \log^{- \frac{1}{q}- \alpha} \! \left(\frac{e}{r}\right) \log\log^{- \beta} \!  \!\left(\frac{e}{r}\right) &\quad   \text{if} \ \,  q\in[1,\infty], \ \alpha   > - \frac{1}{q}, \ \beta \in \R  \\[0.2ex]  
 r^{1-n}\log\log^{- \frac{1}{q}- \beta}\!\left( \frac{e}{r}\right) &\quad    \text{if} \ \,  q\in[1,\infty], \ 
 \alpha  = - \frac{1}{q}, \     \beta > - \frac{1}{q}  \\[0.2ex]  
r^{1-n}\log\log\log^{ - \frac{1}{q}} \!  \left(\frac{e}{r}\right) &\quad    \text{if} \ \,  q\in[1,\infty), \ \alpha= \beta   = - \frac{1}{q}  
\end{cases} 
\end{equation}

\noindent near $0$. 
Moreover, the function $\widehat \psi_{1}$  is (up to multiplicative constants) the optimal admissible function in the embedding~\eqref{E:Ca*A},  in the sense that, if there exists an admissible function  $\varphi$ such that $W^1 L^{(1, q; \alpha, \beta)}(\Omega) \hookrightarrow \mathcal L^{\varphi(\cdot)}(\Omega) $, then $\mathcal L^{\widehat \psi_{1}(\cdot)}(\Omega) \hookrightarrow \mathcal L^{\varphi(\cdot)}(\Omega)$. 

\vspace{1mm}
   
\noindent {\sc(II)}   If $m \in \{2, \dots,  n\}$, then   
\begin{equation}\label{E:Ca*B}
W^m L^{(1, q; \alpha, \beta)}(\Omega) \, \hookrightarrow \,  \mathcal L^{\widehat\psi_m(\cdot)}(\Omega),
\end{equation}
 where $\widehat\psi_m\colon(0,\infty)\to(0,\infty)$  obeys 
\begin{equation}\label{E:Ca*Bb}
\widehat \psi_m(r) \, \approx \,
\begin{cases}
r^{m  -n} \log^{-\frac{1}{q}-\alpha}   \!\left(\frac{e}{r}\right)   \log\log^{-\beta}\! \left(\frac{e}{r}\right) 
&\quad    \text{if} \ \,  q \in [1, \infty], \ \alpha >- \frac{1}{q}, \ \beta \in \R   \\[0.2ex]  
r^{m -n} \log\log^{- \frac{1}{q}-\beta}\!   \left(\frac{e}{r}\right) &\quad    \text{if} \ \, q \in [1, \infty], \   \alpha   = - \frac{1}{q},   \ \beta > - \frac{1}{q} \\[0.2ex] 
r^{m -n}   \log\log\log^{-\frac{1}{q}}\!   \left(\frac{e}{r}\right) &\quad    \text{if} \ \, q \in [1, \infty), \   \alpha  =  \beta  = - \frac{1}{q}  
\end{cases}
\end{equation}
near $0$. 
Moreover, the function $\widehat \psi_{m}$ is (up to multiplicative constants) the optimal admissible function in the embedding~\eqref{E:Ca*B}. 
\end{thm} 
 
\begin{remark}\label{R:camp*} \rm   If $m>n$, the optimal Campanato embeddings of the form \eqref{E:CaPB} and \eqref{E:Ca*B} trivially  hold,  with the functions $\widehat \varphi_m, \widehat \psi_m  \colon(0,\infty) \to (0,\infty)$   obeying  
$$\widehat \varphi_m(r) = \widehat\psi_m(r) \, \approx \, r \quad \ \text{near} \ 0.$$   
This is a straightforward consequence of \cite{CCPS2}*{Theorem~2.6}, via the basic properties \eqref{N2}, \eqref{N3} and \eqref{N5}  of any r.i.~ function norm.   
\end{remark} 

\section{Concluding remarks}\label{comparison}  Our results of Sections~\ref{SS:Holder}-\ref{SS:MorreyCamp} allow us to derive the following both well\hyp{}known and new results, in a spirit similar to the final comments in Section~\ref{sec2.6}. 

\begin{itemize} 
\item[ \bf (i)] For any bounded Jones domain $\Omega$ in $\R^n$, with $n \geq 2$, the Sobolev space $W^{m}L^{p, q; \alpha, \beta}(\Omega)$, with $m \leq n-1$, is embedded in some Hölder type space if and only if    
\begin{equation}\label{exH}
\begin{cases}     
p=\frac nm, \  q\in[1,\infty], \  \alpha   = \frac{1}{q'}, \   \beta   > \frac{1}{q'}  \\[0.2ex]  
p=\frac nm,  \  q\in[1,\infty], \    \alpha   > \frac{1}{q'}, \ \beta \in \R  \\[0.2ex]    
p \in \left(\frac nm, \infty\right] \!, \  q, \alpha, \beta \  \text{fulfill one of the alternatives} ~(\ref{c:LZnorm}). 
\end{cases} \end{equation}
Then, the optimal Hölder type space for $W^{m}L^{p, q; \alpha, \beta}(\Omega)$ agrees, up to equivalent norms, with the corresponding optimal Campanato type space if and only if the parameters $p,  q, \alpha, \beta$    fulfill one of the alternatives~\eqref{exH} under the additional restriction that $q=1$ when $p=\frac nm$.
\vspace{0.5mm}

\item[ \bf (ii)] For any John domain $\Omega$ in $\R^n$, with $n \geq 2$, the optimal Morrey type space and the optimal Campanato type space for $W^{m}L^{p, q; \alpha, \beta}(\Omega)$  coincide,  up to equivalent norms,  if and only if  
\begin{equation}\label{comp:MC-L}
\begin{cases}  
p=  q=1,  \    \alpha=0, \    \beta> 0   \\
p= q=1, \  \alpha>0, \ \beta\in \R   \\  
p\in \left(1,\frac nm\right), \     q\in[1,\infty], \ \alpha   = 0, \      \beta   < 0     \\ 
p=\frac nm, \    q=1, \   \alpha  < 0, \   \beta\in \R  \\ 
p=\infty,  \    q\in[1,\infty],  \   \alpha   = 0, \      \beta   \leq 0.  
\end{cases} 
\end{equation} 
\vspace{0.5mm}

\item[\bf (iii)] For any John domain $\Omega$ in $\R^n$, with $n \geq 2$, the optimal Morrey type space and the optimal Campanato type space for $W^{m}L^{(1, q; \alpha, \beta)}(\Omega)$ actually coincide, up to equivalent norms, for all parameters $q, \alpha, \beta$ fulfilling one of the alternatives~\eqref{c:L1}.
\end{itemize} 

\vspace{0.5mm}
 The assertion {\bf (i)} can be proved in the following three stages.
Firstly, note that
\begin{equation}\label{HinC}
C^{0,   \sigma(\cdot)}(\Omega) \, \hookrightarrow \,  \mathcal L^{\sigma(\cdot)}(\Omega) 
\end{equation}
for every modulus of continuity $\sigma \colon (0, \infty) \to (0, \infty)$. 

\noindent Secondly, as shown in \cite{Sp}, for any admissible function $\varphi \colon (0, \infty) \to (0, \infty)$ that decays so fast to $0$ as  $r \to 0^+$ so that the Dini condition  
\begin{equation}\label{Dini}
 \int_0 s^{-1} \varphi(s) \, \d s < \infty  
\end{equation} 
is satisfied, then  
\begin{equation}\label{Reverse}
\mathcal L^{\varphi(\cdot)}(\Omega) \, \hookrightarrow \,  C^{0,   \sigma_{\varphi}(\cdot)}(\Omega),
\end{equation} 
where $\sigma_{\varphi}\colon (0, \infty) \to (0, \infty)$ is the function defined as 
\begin{equation}\label{E: spanne}
\sigma_{\varphi}(r)=   \|  s^{-1} \varphi(s)\|_{L^{1}(0,r)}   \quad \ \text{near}  \  0.
\end{equation}
The relevant function $\sigma_{\varphi}$ is (up to multiplicative constants) the optimal modulus of continuity in the embedding \eqref{Reverse}.  
Notice that, when the function $s^{-1} \varphi(s)$ is non\hyp{}increasing, the condition \eqref{Dini} is necessary for the embedding of the Campanato space $\mathcal L^{\varphi(\cdot)}(\Omega)$ in $L^{\infty}(\Omega)$ to hold.

\noindent Lastly, thanks to the previous two steps \eqref{HinC} and \eqref{Reverse}, coupling Theorem~\ref{T:HOp} and Theorem~\ref{T:CampP} yields the conclusion {\bf (i)} by means of standard computations with the help of \eqref{E:H10} in Lemma~\ref{L:H1}.

\vspace{0.5mm}
The assertion {\bf (ii)} [resp.~{\bf (iii)}] follows from a parallel argument. 
Indeed, on one hand, 
\begin{equation}\label{MinC}
\mathcal M^{\varphi(\cdot)}(\Omega)  \, \hookrightarrow \,  \mathcal L^{\varphi(\cdot)}(\Omega) 
\end{equation}
for every admissible function $\varphi \colon (0, \infty) \to (0, \infty)$.
On the other hand, \cite{CPcamp}*{Proposition~1.4 and Comments after equation (1.5)} tell us that if  $\psi \colon (0, \infty) \to (0, \infty)$ is any admissible function, then  
\begin{equation}\label{ReverseLC}
\mathcal L^{\psi(\cdot)}(\Omega) \, \hookrightarrow \,   \mathcal M^{\varphi_{\psi}(\cdot)}(\Omega),
\end{equation} 
where $\varphi_{\psi}\colon (0, \infty) \to (0, \infty)$ is the admissible (continuous) function defined as 
\begin{equation}\label{E:CPcamp}
\varphi_{\psi}(r)=   \|  s^{-1} \psi(s)\|_{L^{1}(r,1)}   \quad \ \text{near}  \  0.
\end{equation}
The conclusion {\bf (ii)} [resp.~{\bf (iii)}]  may then be  derived from a combination of the properties \eqref{MinC}-\eqref{E:CPcamp} with Theorem~\ref{T:Mor}, Part~{\sc (A)}, and Theorem~\ref{T:CampP} [resp.~Theorem~\ref{T:Mor}, Part~{\sc (B)}, and Theorem~\ref{T:Ca*}]  via the properties \eqref{E:H11} and \eqref{UU} stated in the next section.

\vspace{1mm}

Plainly, thanks to their optimality, each of the results {\bf (i)}-{\bf (iii)} provides a link between embeddings in different classes of spaces. 
As an example, the case {\bf (iii)} tells us that under the same assumptions as in Theorem~\ref{T:Mor}, Part~{\sc(B)},  one has that the embedding   $W^{m}L^{(1, q; \alpha, \beta)}(\Omega) \hookrightarrow  \mathcal L^{\varphi(\cdot)}(\Omega)$  holds for some admissible function $\varphi$ if and only if  $W^{m}L^{(1, q; \alpha, \beta)}(\Omega)\hookrightarrow  \mathcal M^{\varphi(\cdot)}(\Omega)$.

\noindent Furthermore, on coupling the conclusions {\bf (i)} and {\bf (ii)} with Theorem~\ref{T:C0}, one exactly knows when and how the optimal target space changes. This is dictated, in particular, by the thresholds for $\alpha$ and $\beta$ identified in Sections~\ref{SS:Holder}-\ref{SS:MorreyCamp}, as the following example shows.

\begin{ex} Let $\Omega$ be a John domain in $\R^n$, with $n \geq 2$, and let $m\in \N$, with $m<n$. Then, Theorem~\ref{T:CampP} tells us that for any $\alpha, \beta \in \R$ the  embedding
\begin{equation}\label{E:CaPB10}
W^m L^{\frac{n}{m}, 1; \alpha, \beta}(\Omega) \, \hookrightarrow \,  \mathcal L^{\varphi(\cdot)}(\Omega)  
\end{equation}holds,
with  $$\varphi (r) \approx \log^{- \alpha} \! \left(\frac{e}{r}\right) \log\log^{- \beta} \!   \left(\frac{e}{r}\right)\quad \ \text{near} \ \, 0.$$  Moreover, the target space $\mathcal L^{\varphi(\cdot)}(\Omega)$ is optimal. Thus, in particular, the John\hyp{}Nirenberg space ${\rm BMO}(\Omega)$ is the optimal Campanato space in \eqref{E:CaPB10} in the case $\alpha=\beta=0$.

When $\alpha<0$ and $\beta\in \R$, the embedding \eqref{E:CaPB10} is equivalent to 
\begin{equation}\label{E:CaPB11} 
W^m L^{\frac{n}{m}, 1; \alpha, \beta}(\Omega) \, \hookrightarrow \, \mathcal M^{\varphi(\cdot)}(\Omega),   
\end{equation} as observed in item {\bf (ii)}. 

Under the additional assumption that $\Omega$ is a   Jones domain,~Theorem~\ref{T:C0} ensures that
 \begin{equation*} 
W^m L^{\frac{n}{m}, 1; 0, \beta}(\Omega) \, \not \hookrightarrow \,  C^{0}_b(\Omega)   
\end{equation*}
for ($\alpha=0$ and) $\beta<0$, and
(for $\alpha=\beta=0$)  
\begin{equation*} 
W^m L^{\frac nm, 1}(\Omega) \, \hookrightarrow \,  C^{0}_b(\Omega)  
\end{equation*}
No embedding of $W^{m}L^{\frac{n}{m}, 1; \alpha, \beta}  (\Omega)$ into spaces of the form $C^{0, \sigma(\cdot)}(\Omega)$ plainly holds when  $\alpha=0$ and $\beta<0$, but this drawback occurs also for   $\alpha=\beta=0$, as hinted in Remark~\ref{SiCnoH}. 

When $\alpha=0$ and $\beta>0$,  the embedding \eqref{E:CaPB10} is actually equivalent to 
\begin{equation}\label{E:CaPB12}
W^m L^{\frac{n}{m}, 1; 0, \beta}(\Omega) \, \hookrightarrow \,  C^{0, \varphi(\cdot)}(\Omega),   
\end{equation} as observed in item {\bf (i)}, and the target space $C^{0, \varphi(\cdot)}(\Omega)$ is optimal.
 \end{ex}

\vspace{1mm}

Finally, let us also observe that, as far as the specific exceptional case $m=n$, $p=q=1$, $\alpha  = \beta=0$ in item {\bf (ii)} is concerned, it is  well\hyp{}known that the optimal Campanato type space is the John\hyp{}Nirenberg space ${\rm BMO}(\Omega)$ whereas the optimal Morrey type space is $L^\infty(\Omega)$, and  $L^\infty(\Omega) \subsetneqq {\rm BMO}(\Omega)$. Now,  Theorem~\ref{T:CampP} tells us that, given any John domain $\Omega$, the space ${\rm BMO}(\Omega)$ is the optimal Campanato type space for an $m$\hyp{}th order Sobolev space, with $m \leq n$, built on GLZ~spaces $L^{p, q; \alpha, \beta}(\Omega)$ if and only if $p= \frac nm$, $q \in [1, \infty]$, and  $\alpha  = \beta=0$. 
This is consistent with the fact that the Lorentz space $L^{\frac{n}{m},\infty}(\Omega)$ is the optimal (largest possible) rearrangement\hyp{}invariant domain space in the embedding 
\begin{equation*} 
W^{m}L^{\frac nm,\infty}(\Omega) \, \hookrightarrow \, \BMO(\Omega),
\end{equation*}
as shown in \cite{CCPS2}*{Proposition~7.1}.

\section{Technical lemmas}\label{secTec} 
The present section is devoted to introducing and studying  a relevant class of continuous weights, which naturally comes into play in our GLZ~framework. 
The statements below shall allow a unified approach to the proofs of our main results,  which is more efficient than a succession of ad hoc arguments tailored to particular circumstances. An additional piece of information, of independent interest, is given by Remark~\ref{R:H1} below.

\begin{lemma}\label{L:H1}\   Let  $q\in[1,\infty]$, $\alpha, \beta \in \R$. Given $\lambda\in \R$, define the function $\Psi_{\lambda; \, q, \alpha, \beta}\colon (0,\infty) \to  (0,\infty)$  as
\begin{equation}\label{E:1pie}
\Psi_{\lambda; \, q, \alpha, \beta}(s)= s^{\lambda -\frac{1}{q'}} \ell^{-\alpha}(s) \ell\ell^{- \beta}(s)  \quad \ \text{for} \ s \in  (0,\infty).
\end{equation} 

Then,
\begin{align}\label{E:H10}  
\|  \Psi_{\lambda; \, q, \alpha, \beta}\|_{L^{q'}(0,r)}    \approx  
\begin{cases}\begin{aligned}
& \infty & \quad   &\text{if} \, \begin{cases}  \lambda <0, \ q \in [1, \infty], \ \alpha, \beta \in \R  \\[0.2ex] \lambda=0,  \ q \in [1, \infty], \   \alpha< \frac{1}{q'},  \ \beta \in \R           \\[0.2ex] 
\lambda=0,  \ q \in [1, \infty], \   \alpha=\frac{1}{q'},  \ \beta < \frac{1}{q'}         \\[0.2ex]
\lambda=0,  \ q \in (1, \infty], \   \alpha= \beta = \frac{1}{q'}      \\[0.2ex]
 \end{cases}  \\[1ex] 
& 1
& \quad   &   \text{if}  \   \lambda=0, \ q=1, \ \alpha= \beta= 0        \\[0.2ex]  
&\ell\ell^{\frac{1}{q'}- \beta}(r) 
&  \quad   &   \text{if} \     \lambda=0,    \ q \in [1, \infty], \ \alpha   =\tfrac{1}{q'}, \    \beta > \tfrac{1}{q'}     \\[0.2ex]  & \ell^{\frac{1}{q'} -\alpha}(r) \, \ell\ell^{- \beta}(r)
&  \quad   &   \text{if} \    \lambda=0,     \ q \in [1, \infty], \ \alpha   >\tfrac{1}{q'}, \ \beta \in \R      \\[0.2ex]   & r^{\lambda}\, \ell^{- \alpha}(r) \, \ell\ell^{- \beta}(r) 
& \quad    &   \text{if} \    \lambda >0,  \ q \in [1, \infty], \ \alpha, \beta \in \R    
\end{aligned}
\end{cases} 
\end{align} near $0$, and
\begin{align}\label{E:H11}    \|  \Psi_{\lambda; \, q, \alpha, \beta}\|_{L^{q'}(r,1-r)}  \approx   \begin{cases} \begin{aligned}
&r^{\lambda}\, \ell^{- \alpha}(r) \, \ell\ell^{- \beta}(r)  
& \ \  &       \text{if}   \ \lambda <0, \ q \in [1, \infty], \ \alpha,   \beta \in \R  \\[0.2ex]  
& \ell^{\frac{1}{q'} -\alpha}(r) \, \ell\ell^{- \beta}(r)
& &     \text{if} \ \lambda=0,  \ q \in [1, \infty], \ \alpha   <\tfrac{1}{q'}, \ \beta \in \R    \\[0.2ex] 
&\ell\ell^{\frac{1}{q'}- \beta}(r) 
& &   \text{if} \ \lambda=0,  \ q \in [1, \infty], \   \alpha=\tfrac{1}{q'},  \ \beta < \tfrac{1}{q'}       \\[0.2ex]  
&\ell\ell\ell^{\frac{1}{q'}}(r) 
& &     \text{if}   \  \lambda=0,  \ q \in (1, \infty], \  \alpha =  \beta  = \tfrac{1}{q'}    \\[0.2ex]  
& 1 &&\text{if} \, \begin{cases} \lambda=0, \ q=1, \ \alpha= \beta= 0         \\[0.2ex]  \lambda=0,  \ q \in [1, \infty], \    \alpha  =\frac{1}{q'}, \    \beta > \tfrac{1}{q'}        \\[0.2ex] 
 \lambda=0,  \ q \in [1, \infty], \ \alpha  >\frac{1}{q'}, \ \beta \in \R    \\[0.2ex]       \lambda>0,  \ q \in [1, \infty], \ \alpha, \beta \in \R       \end{cases}
\end{aligned}   
\end{cases} 
\end{align} near $0$.
\end{lemma}
\begin{proof} Let $q\in[1,\infty]$, $\alpha, \beta, \lambda \in \R$. We begin by proving the alternative behaviors \eqref{E:H10}.
Assume, for the time being, that $r\in (0,1)$. Throughout this proof the constants in the relation  \lq\lq $\approx$'' are independent of $r$.

 If $\lambda < 0$, then, 
$$\|  \Psi_{\lambda; \, q, \alpha, \beta}\|_{L^{q'}(0,r)} \,  \approx \, \sup_{s  \in (0,r)} s^{\lambda} \ell^{-\alpha}(s)  \ell\ell^{- \beta}(s)= \lim_{s  \to 0^+}  s^{\lambda} = \infty.$$ 
The equivalence follows from an application of \cite{ED}*{Proposition~3.4.33~(vi)} when $q >1$ and the fact that the function $ \Psi_{\lambda;\, 1, \alpha, \beta}$ is equivalent to a non\hyp{}increasing function on $\left(0, 1\right)$. This proves \eqref{E:H10}\textsubscript{1}.

  Next, let  $\lambda =0$. If $q>1$ and $\alpha  \neq \frac{1}{q'}$, thanks to a change of variables, and also an application of \cite{ED}*{Proposition~3.4.33,~(v) and~(vi)} when $\beta \neq 0$,  one has that
$$ \|  \Psi_{0; \, q, \alpha, \beta}\|_{L^{q'}(0,r)} = \|   t^{-\alpha}  \ell^{- \beta}(t)\|_{L^{q'}(\ell(r) ,\infty)}\,  \approx \,
\begin{cases} \infty &\quad   \text{if} \ \, q \in (1, \infty], \, \alpha < \frac{1}{q'}, \, \beta \in \R 
\\[0.2ex] 
\ell^{\frac{1}{q'}- \alpha}(r) \ell\ell^{ - \beta}(r) &\quad  \text{if} \ \,   q \in (1, \infty], \,   \alpha   >  \frac{1}{q'}, \, \beta \in \R.
\end{cases}$$
When   $q>1$ and $\alpha   =\frac{1}{q'}$, via a change of variables one infers that
$$ 
\|  \Psi_{0; \, q, \frac{1}{q'}, \beta}\|_{L^{q'}(0,r)} = \|t^{-\beta}\|_{L^{q'}(\ell\ell(r) ,\infty)}\,  \approx \, 
\begin{cases}    \infty &\quad   \text{if} \ \,    q \in (1, \infty],   \,  \beta  \leq \frac{1}{q'}\\[0.2ex] 
\ell\ell^{\frac{1}{q'}- \beta}(r) 
&\quad  \text{if} \ \,    q \in (1, \infty],   \,   \beta   >\frac{1}{q'} .
\end{cases} 
$$ 
For $q=1$, inasmuch as the function
\begin{equation}\label{new}    
\Psi_{0; \, 1, \alpha, \beta}  \ \  \text{is}   \    \begin{cases} \text{non-increasing on}   \ (0,1)    &     \text{if} \   \begin{cases}  \alpha<0, \, \beta \leq - \alpha\\[0.2ex]    \alpha = 0, \, \beta   <0 \end{cases} \\[0.5ex] 
\text{non-increasing on}   \ (0,\xi), \  \text{non-decreasing on} \ (\xi,1)   &     \text{if} \    \alpha< 0, \, \beta > - \alpha\\[0.5ex] 
1 &   \text{if} \      \alpha = \beta =0  \\[0.5ex] 
\text{non-decreasing on} \ (0,1)  &    \text{if} \ \begin{cases}  \alpha=0, \, \beta >0\\[0.2ex]    \alpha > 0, \, \beta   \geq -\alpha \end{cases} \\[0.5ex] 
\text{non-decreasing on} \    (0,\xi), \text{non-increasing on}    \ (\xi,1)  &    \text{if} \ \,  \alpha   > 0, \, \beta <  - \alpha
\end{cases}  \end{equation}   
with $\xi=e^{1- e^{-\left(\frac{\beta}{\alpha}+1\right)}}$,
then
$$ \|  \Psi_{0; \, 1, \alpha, \beta}\|_{L^{\infty}(0,r)} = \| \ell^{-\alpha}  \ell\ell^{- \beta}  \|_{L^{\infty}(0,r)} \, \approx \,    \begin{cases}   \infty &\quad   \text{if} \ \begin{cases}  \alpha<0, \, \beta \in \R\\[0.2ex]    \alpha = 0, \, \beta   <0 \end{cases} \\[0.5ex] 
1 &\quad  \text{if} \ \,    \alpha = \beta =0  \\[0.2ex]  
\ell\ell^{ - \beta}(r) &\quad   \text{if} \ \,    \alpha = 0, \, \beta >0  \\[0.2ex]  
\ell^{ -\alpha}(r)\ell\ell^{- \beta}(r) &\quad   \text{if} \ \,  \alpha   > 0, \, \beta \geq   -  \alpha \end{cases}  $$  for every $r\in (0,1)$,
and
\begin{equation}\label{new1}   \|  \Psi_{0; \, 1, \alpha, \beta}\|_{L^{\infty}(0,r)}   \, \approx \,       
\ell^{ -\alpha}(r)\ell\ell^{- \beta}(r) \quad  \  \text{if} \ \,  \alpha   > 0, \, \beta < -  \alpha  \end{equation}  for   $r\in (0,\xi)$.
Altogether, this ensures that the assertions \eqref{E:H10}\textsubscript{2-7} hold.

Now, consider  $\lambda >0$. When $q =1$, $$\| \Psi_{\lambda; \, 1, \alpha, \beta}\|_{L^{\infty}(0,r)} = r^{\lambda} \ell^{-\alpha}(r)  \ell\ell^{- \beta}(r) ,$$ since the function $\Psi_{\lambda;\, 1, \alpha, \beta}$ is equivalent to a non\hyp{}decreasing function on $\left(0, 1\right)$; if  $q >1$, via an  application of \cite{ED}*{Proposition~3.4.33~(v)}, one obtains that $$\|  \Psi_{\lambda; \, q, \alpha, \beta}\|_{L^{q'}(0,r)}\, \approx \,   \sup_{s  \in (0,r)} s^{\lambda} \ell^{-\alpha}(s)  \ell\ell^{- \beta}(s) \, \approx \,  r^{\lambda} \ell^{-\alpha}(r)  \ell\ell^{- \beta}(r).$$
The remaining assertion \eqref{E:H10}\textsubscript{8}  is thus established.    

We now turn the attention to the alternative behaviors \eqref{E:H11}, and we assume  that $r\in \left(0,\frac13\right]$. Here, and in what follows, the constants in the relation  \lq\lq $\approx$''   are independent of $r$, as above.

\noindent We preliminarily observe that, when either $\lambda < 0$ or $\lambda= 0$ and $q \in (1, + \infty]$, then, 
\begin{equation}\label{UU}
\| \Psi_{\lambda; \, q, \alpha, \beta}\|_{L^{q'}(r,1-r)}  \,   \approx \,  \left\|\Psi_{\lambda; \,  q, \alpha, \beta} \right\|_{L\sp{q'}(r,1)} \quad \ \text{for any} \ \alpha, \beta \in \R, 
\end{equation}
inasmuch as the function $\Psi_{\lambda; \, q, \alpha, \beta}$ is equivalent to a non\hyp{}increasing function on  $(0,1)$ and the length of the interval $(1-r,1)$ does not exceed that of $(r,1-r)$.

First, take $\lambda <0$.  When  $q=1$, one has that 
\begin{equation}\label{G1}
\left\|\Psi_{\lambda; \, 1, \alpha, \beta} \right\|_{L\sp{\infty}(r,1-r)}  \, \approx \, 
r^{\lambda}    \ell^{-\alpha}(r)\ell\ell^{-\beta}(r),
\end{equation}
being the function $\Psi_{\lambda; \, 1, \alpha, \beta}$ equivalent to a non\hyp{}increasing function on $\left(0, 1\right)$; for  $q > 1$, by the property \eqref{UU},  a change of variables and an application of \cite{ED}*{Proposition~3.4.33~(vi)}, one gets 
\begin{equation}\label{G2} 
\begin{aligned}    \| \Psi_{\lambda; \, q, \alpha, \beta}\|_{L^{q'}(r,1-r)}  \,  
&\approx \,  \left\|\Psi_{\lambda;\,  q, \alpha, \beta} \right\|_{L\sp{q'}(r,1)}    =  \left\|\Psi_{-\lambda; \, q, \alpha, \beta} \right\|_{L\sp{q'}\left(1,\frac1r\right)} \\   
&\approx \,   \sup_{t \in \left(1, \frac 1r\right)}  t^{ -\lambda} \ell^{-\alpha}(t)\ell\ell^{-\beta}(t)   \,  \approx \,  r^{\lambda} \ell^{-\alpha}(r) \ell\ell^{- \beta}(r) .
\end{aligned}
\end{equation} 
Here, the validity of the last equivalence is due to the fact that the function  $t\mapsto t^{ -\lambda} \ell^{-\alpha}(t)\ell\ell^{-\beta}(t)$ is equivalent to a non\hyp{}decreasing function on $\left(1,\infty\right)$.
Coupling \eqref{G1} and \eqref{G2} yields the conclusion \eqref{E:H11}\textsubscript{1}, that plainly holds for every $r\in \left(0,\frac13\right]$.
    
 Next, suppose that $\lambda =0$. If $q>1$ and $\alpha  \neq \frac{1}{q'}$, we use the property \eqref{UU}, a change of variables and applications of \cite{ED}*{Proposition~3.4.33,~(v) and~(vi)} to infer that
\begin{equation}\label{G3} 
  \|  \Psi_{0; \, q, \alpha, \beta}\|_{L^{q'}(r,1-r)}    \,  
 \approx \,  \|   t^{-\alpha}  \ell^{- \beta}(t)\|_{L^{q'}(1,\ell(r))}    \,    \approx  \,  \begin{cases} \ell^{\frac{1}{q'}- \alpha}(r) \ell\ell^{ - \beta}(r)
&\quad   \text{if} \ \,  \alpha   <  \frac{1}{q'}, \,  \beta \in \R   \\[0.2ex]  1 &\quad  \text{if} \ \,   \alpha   >  \frac{1}{q'}, \,     \beta   \neq 0  
\end{cases}    
\end{equation} for every $r\in \left(0,\frac13\right]$, and
\begin{equation*}\label{G3bis} 
  \|  \Psi_{0; \, q, \alpha, 0}\|_{L^{q'}(r,1-r)}    \,  
 \approx \,  \|   t^{-\alpha}  \|_{L^{q'}(1,\ell(r))}    \,    \approx  \,    
1 \qquad \text{if} \ \,     \alpha   > \tfrac{1}{q'}          
\end{equation*} near $0$.

\noindent When  $q>1$ and $\alpha =\frac{1}{q'}$, thanks to  the property \eqref{UU} and a change of variables  one obtains that
\begin{equation*}  \|  \Psi_{0; \, q, \frac{1}{q'}, \beta}\|_{L^{q'}(r,1)} = \|t^{-\beta}\|_{L^{q'}(1,\ell\ell(r))}\,  \approx \, \begin{cases}\ell\ell^{\frac{1}{q'}- \beta}(r) 
&\quad   \text{if} \ \,   \beta  <\frac{1}{q'}  \\[0.2ex]  \ell\ell\ell^{\frac{1}{q'}}(r) &\quad  \text{if} \ \,      \beta   = \frac{1}{q'}   \\[0.2ex] 
    \end{cases}   
\end{equation*}for every $r\in \left(0,\frac13\right]$, and
\begin{equation*}  \|  \Psi_{0; \, q, \frac{1}{q'}, \beta}\|_{L^{q'}(r,1)} = \| t^{-\beta}\|_{L^{q'}(1,\ell\ell(r))}\,  \approx \,   
1  \qquad   \text{if} \ \,   \beta  > \tfrac{1}{q'}      
\end{equation*} near $0$.

\noindent For $q=1$, in view of \eqref{new}, one plainly obtains that
\begin{equation}\label{new2}  
\begin{aligned}       
\left\|\Psi_{0;\,1, \alpha, \beta}  \right\|_{L\sp{\infty}(r,1-r)}   \,   \approx \,   \begin{cases}\ell^{-\alpha}(r)\ell\ell^{-\beta}(r) \quad   &\text{if} \   \alpha < 0, \, \beta \leq - \alpha   \\[0.2ex]  \ell\ell^{-\beta}(r) \quad   &\text{if}  \  \alpha = 0, \,  \beta< 0 \\[0.2ex] 
1 & \text{if} \     \alpha= \beta= 0  
\end{cases} 
\end{aligned}
\end{equation}
for every $r\in \left(0,\frac13\right]$, and
\begin{equation}\label{new3}        
\left\|\Psi_{0;\,1, \alpha, \beta}  \right\|_{L\sp{\infty}(r,1-r)}   \,   \approx \,  \begin{cases}\ell^{-\alpha}(r)\ell\ell^{-\beta}(r) \quad   &  \text{if} \  \alpha < 0, \, \beta > - \alpha \\[0.2ex] 
1 &\text{if} \, \begin{cases}\alpha = 0,  \, \beta >0    \\[0.2ex] \alpha > 0, \,  \beta \in \R\end{cases} 
\end{cases}   
\end{equation} near $0$. The conclusions \eqref{E:H11}\textsubscript{2-7} are thus established.
In order to prove the last assertion \eqref{E:H11}\textsubscript{8}, it suffices to note that, for all $\lambda>0$,   
\begin{equation*}      
\left\|\Psi_{\lambda;\, q, \alpha, \beta}  \right\|_{L\sp{q'}(r,1-r)}   \,   \approx \,   (1-r)^{\lambda} \ell^{-\alpha}(1-r)\ell\ell^{-\beta}(1-r)  \,   \approx \,  1   \end{equation*}near $0$.
\end{proof}

\begin{remark}\label{R:H1} \rm    Given $q\in[1,\infty]$, $\alpha$, $\beta$  and $\lambda \in \R$, the statement of Lemma~\ref{L:H1} is formulated in terms of the behaviors of the relevant norms $\| \Psi_{\lambda; \, q, \alpha, \beta}\|_{L^{q'}(0,r)}$ and $\| \Psi_{\lambda; \, q, \alpha, \beta}\|_{L^{q'}(r,1-r)}$ of the function $\Psi_{\lambda; \, q, \alpha, \beta}$ defined in \eqref{E:1pie} for $r$ just near $0$, for great clarity of exposition.

 As highlighted in the previous proof, the conclusions \eqref{E:H10}  hold for every $r \in (0,1)$, with the exception of \eqref{E:H10}\textsubscript{7} in the cases ($\lambda=0$,) $q=1$, $\alpha >0$ and $\beta < - \alpha$ (cf.~Eq.~\eqref{new1}). 
 
 \noindent On the other hand, a parallel issue for the conclusions \eqref{E:H11}  is inevitably more articulate, since they depend at once on the behaviors of the functions $s^{\lambda -\frac{1}{q'}}$, $\ell^{-\alpha}$ and $ \ell\ell^{- \beta}$ not only near $0$ but also near $1$. Specifically, the properties 
\eqref{E:H11}\textsubscript{1} and \eqref{E:H11}\textsubscript{3-5}   hold  for every $r\in \left(0,\frac13\right]$. This range for $r$ pertains to   the equivalence  \eqref{E:H11}\textsubscript{2} apart from the cases ($\lambda=0$,) $q=1$, $\alpha <0$ and $\beta > -\alpha$ (cf.~Eqs.~\eqref{G3}\textsubscript{1},~\eqref{new2}\textsubscript{1} and~\eqref{new3}\textsubscript{1}), as well as to the equivalence  \eqref{E:H11}\textsubscript{7} provided ($\lambda=0$,) $q>1$, $\alpha > \frac{1}{q'}$ and $\beta \neq 0$ (cf.~Eq.~\eqref{G3}\textsubscript{2}).
\end{remark} 

\begin{lemma}\label{L:H2}\   Let  $q\in[1,\infty)$ and let $\Psi_{\lambda; \, q, \alpha, \beta}$ be the function from  Lemma~\ref{L:H1}.  
Then
\begin{align}\label{E:H10bis}   
\| \Psi_{0; \, q, \frac{1}{q'}, \frac{1}{q'}} \ell\ell\ell^{-1}\|_{L^{q'}(0,r)}\,  \approx \,  \ell\ell\ell^{-\frac{1}{q}} (r)\quad \ \text{for}  \  r\in (0,1),
\end{align} 
and, given any $\lambda < 0$,
\begin{align}\label{E:H11bis}   
\|  \Psi_{\lambda; \, q, \frac{1}{q'}, \frac{1}{q'}} \ell\ell\ell^{-1}\|_{L^{q'}(r,1-r)}\,  \approx \,   \Psi_{\lambda+ \frac{1}{q'}; \, q, \frac{1}{q'}, \frac{1}{q'}}(r)  \ell\ell\ell^{-1}(r) \quad \ \text{for}  \    r\in \left(0,\tfrac13\right] .
\end{align} 
\end{lemma}
\begin{proof} The equivalence \eqref{E:H10bis} is trivial. Consider property \eqref{E:H11bis}. Fix $\lambda < 0$, and let $r\in \left(0,\tfrac13\right]$. If $q=1$, one has that
$$\|  \Psi_{\lambda; \, 1, 0, 0} \, \ell\ell\ell^{-1}\|_{L^{\infty}(r,1-r)} = \|  s^{\lambda} \, \ell\ell\ell^{-1}(s)\|_{L^{\infty}(r,1-r)} \,   \approx \,  r^{\lambda} \, \ell\ell\ell^{-1}(r) = \Psi_{\lambda; \, 1, 0, 0}(r)   \ell\ell\ell^{-1}(r).$$ When $q>1$, 
\begin{equation*}
\begin{aligned}
 \|  \Psi_{\lambda; \, q, \frac{1}{q'}, \frac{1}{q'}} \ell\ell\ell^{-1}\|_{L^{q'}(r,1-r)}\,  &\approx \,     \|  \Psi_{\lambda; \, q, \frac{1}{q'}, \frac{1}{q'}} \ell\ell\ell^{-1}\|_{L^{q'}(r,1)}\, \approx \,   \left\|\Psi_{-\lambda; \, q, \frac{1}{q'}, \frac{1}{q'}} \ell\ell\ell^{-1}\right\|_{L\sp{q'}\left(1,\frac1r\right)} \\
 &\approx \,  \sup_{t \in \left(1, \frac 1r\right)}  t^{ -\lambda}   \ell^{-\frac{1}{q'}}(t)  \ell\ell^{- \frac{1}{q'}}(t) \ell\ell\ell^{-1}(t) \, \approx \,  \Psi_{\lambda+ \frac{1}{q'}; \, q, \frac{1}{q'}, \frac{1}{q'}}(r)  \ell\ell\ell^{-1}(r).  
\end{aligned}\end{equation*} Here, the first equivalence holds thanks to the facts that the function $\Psi_{\lambda; \, q, \frac{1}{q'}, \frac{1}{q'}} \ell\ell\ell^{-1}$ is  equivalent to a non\hyp{}increasing function on  $(0,1)$ and   the length of the interval $(1-r,1)$ does not exceed that of $(r,1-r)$. The third equivalence follows from an application of \cite{ED}*{Proposition~3.4.33~(vi)} and  the last one is due to the non\hyp{}decreasing monotonicity (up to equivalence) of the function $t\mapsto t^{ -\lambda}  \ell^{-\frac{1}{q'}}(t)  \ell\ell^{- \frac{1}{q'}}(t) \ell\ell\ell^{-1}(t)$ on  $\left(1,\infty\right)$ and the very definition of the function $\Psi_{\lambda; \, q, \alpha, \beta}$.
\end{proof} 

\section{Proofs of main results}\label{proof}   As a preliminary, let us observe that in all the proofs we shall make use of the properties \eqref{c:equivLZspA} and \eqref{c:equivLZ2} as well as of the characterizations  \eqref{c:ASSnorm}-\eqref{c:ASSnorm2} of the associate spaces of GLZ~spaces many times. This will be done without further explicit reference to them.

\subsection{Proof of optimal embeddings into rearrangement\hyp{}invariant spaces}\label{S:ORI}  In this section, without further warning, we shall repeatedly employ the characterization  \eqref{E:EOPJ2} of the associate space of the optimal  rearrangement\hyp{}invariant target space in Sobolev\hyp{}type embeddings for $W^mX(\Omega)$, when $X(0,1)=L^{p,q; \alpha, \beta}(0,1)$ and  $X(0,1)=L^{(1,q;\alpha, \beta)}(0,1)$, respectively,    the non\hyp{}increasing monotonicity on $(0,1)$ for any $f \in   L^0_+(0,1)$ of the function $S_mf$ defined in \eqref{E:EOPJC}.

\vspace{2mm}
\noindent{\textsc{Proof of Theorem~\ref{T:OptRI}}.} {\sc(A)}  Let $p, q\in[1,\infty]$, $\alpha, \beta \in \R$ fulfill one of the alternatives~\eqref{c:LZnorm}. 
 Assume first that  $p=q =1$, and either $\alpha= 0$ and $\beta \geq 0$ or $\alpha >0$  and $\beta \geq -\alpha$. Then,
\begin{equation*}
\begin{aligned}
\left\|   f \right\|_{{\left(L^{1,1,\alpha, \beta}\right)_{m, \text{opt}}'(0,1)}}    \, 
&\approx \, \left\|S_mf\right\|_{L\sp{\infty, \infty; -\alpha, -\beta}(0,1)}    =  \sup_{t \in (0,1)} t\sp{\frac mn}  f\sp{**}(t) \sup_{s \in (0,t]}     \ell^{-\alpha}(s) \ell\ell^{-\beta}(s) \\  &=\sup_{t \in (0,1)} t\sp{\frac mn}  \ell^{-\alpha}(t) \ell\ell^{-\beta}(t)  f\sp{**}(t) \,  \approx \,   \left\|   f \right\|_{{\left(L^{1^*\!, 1; \alpha, \beta}\right)'(0,1)}} 
\end{aligned} 
\end{equation*} 
for every  $f \in  L^0_+(0,1)$. Indeed, the first equality is due to the interchange of the order of the suprema and the second one to the fact that   $\Psi_{0; \, 1, \alpha, \beta}=  \ell^{-\alpha}  \ell\ell^{-\beta}$ is equivalent to a non\hyp{}decreasing function on $(0,1)$ according to \eqref{new}\textsubscript{4-6}. The conclusions \eqref{E:OptRI1}\textsubscript{1,2} thus hold. 

If  $p=q =1$,   $\alpha >0$  and  $\beta < -\alpha$, then, combining \eqref{new}\textsubscript{7}, \eqref{ASS_q-conc} and \eqref{Endsp_ASS} tells us that
\begin{equation*} 
\begin{aligned}
\left\|   f \right\|_{{\left(L^{1,1,\alpha, \beta}\right)_{m, \text{opt}}'(0,1)}}    \, 
&\approx \, \left\|S_mf\right\|_{L\sp{\infty, \infty; -\alpha, -\beta}(0,1)}    =  \sup_{t \in (0,1)} \left(t\sp{\frac mn}  \sup_{s \in (0,t]}     \ell^{-\alpha}(s) \ell\ell^{-\beta}(s)\right)  f\sp{**}(t) \\  &=\left\|   f \right\|_{\mathsf{M}_{\phi_{\alpha, \beta}}(0,1)}  \,  \approx \,   \left\|   f \right\|_{\Lambda'_{\overbar{\phi}_{\alpha, \beta}}(0,1)} 
\end{aligned} 
\end{equation*} 
for every  $f \in  L^0_+(0,1)$. Here, ${\phi}_{\alpha, \beta} \colon [0, 1] \to [0, \infty)$ is the quasiconcave function defined by  
\begin{equation}\label{E:100}
{\phi}_{\alpha, \beta}(t) \, = \,\begin{cases} t^{\frac mn}   \ell^{-\alpha}(t)\ell\ell^{-\beta}(t)  \quad   &\text{if} \ t \leq e^{1- e^{-\left(\frac{\beta}{\alpha}+1\right)}}\\
 \left( - \frac{\alpha}{\beta} \right)^\beta   e^{\alpha+ \beta}\, t^{\frac mn} \quad   &\text{if} \ t >  e^{1- e^{-\left(\frac{\beta}{\alpha}+1\right)}},\end{cases}\end{equation}  and \begin{equation}\label{E:101}
{\overbar\phi}_{\alpha, \beta}(t) \, = \,\begin{cases} t^{1-\frac mn}   \ell^{\alpha}(t)\ell\ell^{\beta}(t) \quad   &\text{if} \ t \leq e^{1- e^{-\left(\frac{\beta}{\alpha}+1\right)}}\\
 \left( - \frac{\beta}{\alpha} \right)^\beta  e^{-(\alpha+ \beta)}\, t^{1-\frac mn} \quad   &\text{if} \ t >  e^{1- e^{-\left(\frac{\beta}{\alpha}+1\right)}},\end{cases}\end{equation} according to \eqref{Endsp_ASS}. Thus, \eqref{E:OptRI1}\textsubscript{4} is established. 

\noindent When $p \in \left(1, \frac nm\right)$, then,
\begin{align}
\left\|   f \right\|_{{\left(L^{p,q;\alpha, \beta}\right)_{m, \text{opt}}'(0,1)}}    \, 
&\approx \, \left\|S_mf\right\|_{L\sp{p', q'; -\alpha, -\beta}(0,1)}    \, \approx \,  \left\| s\sp{\frac{1}{p'}+\frac mn -  \frac{1}{q'}}       \ell^{-\alpha}(s) \ell\ell^{-\beta}(s) f\sp{**}(s)\right\|_{L\sp{q'}(0,1)}  \label{E:critS}\\  
&  \approx \,   \left\|   f \right\|_{{\left(L^{p^\ast\!,q;\alpha, \beta}\right)'(0,1)}}  \notag  
\end{align} 
for every  $f \in L^0_+(0,1)$.
Concerning the second equivalence in \eqref{E:critS}, let us stress that the inequality $\gtrsim$ trivially follows from the very definition of $S_mf$ and the property \eqref{N3} of r.i.~function norms. The reverse inequality does hold thanks to an application of \cite{GOP}*{Theorem~3.2} and of \cite{ED}*{Proposition~3.4.33~(v)}  when $q\in (1, \infty]$, and the fact that $s \mapsto s^{\frac{1}{p'}}\ell^{-\alpha}(s) \ell\ell^{-\beta}(s)$ is equivalent to a non\hyp{}decreasing function on $(0,1)$ when $q=1$. The conclusion \eqref{E:OptRI1}\textsubscript{3} is thus established.
  
\noindent For  $p= \frac nm $, the chain \eqref{E:critS} is still valid (on replacing $p$  by $\frac nm$) and one has that  
\begin{align*}
\left\|   f \right\|_{{\left(L^{\frac nm,q,\alpha, \beta}\right)_{m, \text{opt}}'(0,1)}}       
&\approx   \left\|f\right\|_{L\sp{(1, q'; -\alpha, -\beta)}(0,1)}     \\ 
&\approx     
\begin{cases} 
\left\|   f \right\|_{{\left(L^{\infty, q; \alpha-1 , \beta}\right)'(0,1)}}     & \      \text{if} \ \,  q\in[1,\infty],  \,   \alpha   < \frac {1}{q'}, \, \beta \in \R 
\\[0.2ex]  
\left\|   f \right\|_{{\big(L\sp{\infty, q; - \frac 1q , \beta-1}\big)'(0,1)}}
& \     \text{if} \ \,  q\in[1,\infty],  \,   \alpha   = \frac {1}{q'}, \,    \beta < \frac {1}{q'} 
\\[0.2ex] 
\left\|   f \right\|_{{\big(L^{\infty, q; - \frac 1q , -\frac 1q , -1}\big)'(0,1)}}
& \   \text{if} \ \,   q\in (1,\infty], \,   \alpha   =   \beta   =\frac {1}{q'} 
\\[0.2ex]
\left\|   f \right\|_{{(L^{\infty})'(0,1)}}
& \  \text{if} \,   
\begin{cases}
q=1, \,  \alpha   =   \beta   = 0  \\[0.2ex]   
q\in[1,\infty], \,   \alpha   = \frac {1}{q'},  \,  \beta > \frac {1}{q'}   \\[0.2ex]  
q\in[1,\infty],  \,   \alpha   > \frac {1}{q'}, \, \beta \in \R
\end{cases}   
\end{cases}
\end{align*} 
for every  $f \in L^0_+(0,1)$, proving the statements \eqref{E:OptRI1}\textsubscript{5-10}.

\noindent When $p > \frac{n}{m}$, it suffices to observe that any GLZ~space $L^{p,q;\alpha, \beta}(0,1)$ embeds in the Lorentz space $L^{\frac{n}{m},1}(0,1)$ and to make use of the  assertion \eqref{E:OptRI1}\textsubscript{8} just proved. This concludes the proof of \eqref{E:OptRI1}.

\vspace{1mm}
\noindent {\sc(B)} Let $q\in[1,\infty]$, $\alpha, \beta \in \R$ satisfy one of the conditions \eqref{c:L1}. Assume first that  $q =1$.\\
In the light of \eqref{c:equivLZ2a}, the embeddings \eqref{E:OptRI2}\textsubscript{1-3}  follow from \eqref{E:OptRI1}\textsubscript{2}, \eqref{E:OptRI1}\textsubscript{4} and \eqref{E:OptRI1}\textsubscript{1}, respectively. 
When $\alpha =  \beta = -1$, 
\begin{equation*}
\begin{aligned}
\left\|   f \right\|_{{\left(L^{(1,1;-1, -1)}\right)_{m, \text{opt}}'(0,1)}}   \, 
&\approx \, \left\|S_mf\right\|_{L\sp{\infty, \infty; 0, 0,-1}(0,1)}    =  \sup_{t \in (0,1)} t\sp{\frac mn}  f\sp{**}(t) \sup_{s \in (0,t]}     \ell \ell\ell^{-1}(s) \\  &=\sup_{t \in (0,1)} t\sp{\frac mn}    \ell \ell\ell^{-1}(t)  f\sp{**}(t) \, = \,   \left\|   f \right\|_{{L^{\frac nm, \infty; 0, 0, -1}(0,1)}}\,  \approx \,   \left\|   f \right\|_{{\left(L^{1^*\!, 1; 0, 0, 1}\right)'(0,1)}} 
\end{aligned} 
\end{equation*} 
The first equality is obtained by interchanging the order of the suprema and the last equivalence follows from \cite{PLK}*{Theorem~3.30~(i)}.  Hence the conclusion  \eqref{E:OptRI2}\textsubscript{4} holds as well. Furthermore, all target spaces in \eqref{E:OptRI2}\textsubscript{1-4}  are   optimal among all rearrangement\hyp{}invariant spaces.

Suppose next that $q\in (1,\infty]$. Let us first focus on the case $\alpha > -\frac{1}{q}$ and $\beta \in \R$. 

\noindent We claim that
\begin{align}\label{E:optemb1}
L^{1^*\!, q;  \alpha + 1,  \beta} (0,1)\,  \hookrightarrow \, (L^{(1,q; \alpha, \beta)})_{m, \text{opt}}(0,1) \,  \hookrightarrow \,   L^{1^*\!, q;  \alpha + \frac 1q,  \beta} (0,1),
\end{align}
\begin{align}
  &(L^{(1,q; \alpha, \beta)})_{m, \text{opt}}(0,1) \,  \not \hookrightarrow \,   L^{1^*\!, q;  \alpha + \epsilon,  \gamma} (0,1) && \text{for every}  \    \epsilon\in  \left(\tfrac{1}{q},1\right], \gamma \in \mathbb{R}, \label{E:optemb1bis}\\[-1ex]
   \intertext{and}
 &(L^{(1,q; \alpha, \beta)})_{m, \text{opt}}(0,1) \,  \not \hookrightarrow \,   L^{1^*\!, q;  \alpha + \frac 1q,  \gamma} (0,1)  &&\text{for every}  \  \gamma > \beta. \label{E:optemb1bis:2}
\end{align}
To verify \eqref{E:optemb1}, notice that the very definition of the function $S_mf$  and the property \eqref{N3} of r.i.~function norms imply that
\begin{equation}\label{P:njpAAx}\begin{aligned}
\left\|   f \right\|_{{\left(L^{(1,q; \alpha, \beta)}\right)_{m, \text{opt}}'(0,1)}}   \, 
&\approx \, \left\|S_m f \right\|_{{L^{\infty,q';-\alpha-1, -\beta}(0,1)}} 
\, \gtrsim \,     \left\lVert     f  \right\rVert_{L^{\frac nm, q'; -\alpha -1, - \beta}(0,1)} \\ 
&\approx  \left\lVert     f  \right\rVert_{(L^{1^*\!, q; \alpha +1,  \beta})'(0,1)} 
\end{aligned}\end{equation}
for every  $f \in  L^0_+(0,1)$. Hence, the first embedding in \eqref{E:optemb1} follows making use of \eqref{E:embequ2b}.
On the other hand,
\begin{equation}\label{P:njpAA}
 \left\|S_m f \right\|_{{L^{\infty,q';-\alpha-1, -\beta}(0,1)}}  \, 
 \lesssim \,     \left\lVert     f  \right\rVert_{L^{\frac nm, q'; -\alpha -\frac 1q, - \beta}(0,1)}   \approx  \left\lVert     f  \right\rVert_{\left(L^{1^*\!, q; \alpha +\frac 1q,  \beta}\right)'(0,1)} 
\end{equation}
for every  $f \in L^0_+(0,1)$. The inequality in \eqref{P:njpAA} holds by \cite{GOP}*{Theorem~3.2~(i)}, applied with  
\begin{equation}\label{E:GOPa1}\begin{aligned}&\varphi= f^{**}, \qquad p=q \ \ \text{but with $q$ replaced by $q'$},   \qquad u(t)=t^{\frac mn}\chi_{(0,1)}(t), \\ & \omega(t)= t^{-1} \ell^{- (\alpha+1)q'}(t)\ell\ell^{- \beta q'}(t)\chi_{(0,1)}(t),\qquad   v(t)= t^{{\frac mn q'}-1} \ell^{1- \left(\alpha+ 1\right)q'}(t)\ell\ell^{- \beta q'}(t)\chi_{(0,1)}(t)  \end{aligned}\end{equation}
for $t \in (0,\infty)$. Notice also that calculation shows that 
\begin{equation}\label{E:GOPa2}  \left(\int_0^{x} \left(\sup_{\tau \in [t,x]} u(\tau) \right)^{q'}  \omega(t) \, \d  t  \right)^{\frac 1{q'}} \, \approx \,  \left(\int_0^{x}    v(t) \, \d  t  \right)^{\frac 1{q'}} \quad \ \text{for every}  \    x\in  (0,\infty) .\end{equation} 
Let us stress that  the inequality $\lesssim$ holds for every $x\in  (0,\infty)$ in \eqref{E:GOPa2} and this fact characterizes the validity of the inequality in \eqref{P:njpAA}. 
Due to  \eqref{E:embequ2b}, coupling \eqref{P:njpAAx} and \eqref{P:njpAA}  proves \eqref{E:optemb1}. The claims \eqref{E:optemb1bis} and \eqref{E:optemb1bis:2} are a direct consequence of \cite{GOP}*{Theorem~3.2~(i)}, in view of \eqref{E:GOPa2}. 

\noindent In addition, combining the first embedding in \eqref{E:optemb1}, the properties \eqref{E:embequ2}, \eqref{E:embequ2b}  and \eqref{E:optemb1bis} with $\epsilon=1, \gamma = \beta$ yields that 
\begin{equation}\label{XE:optemb1BB}L^{1^*\!, q;  \alpha + 1,  \beta} (0,1)\,  \subsetneqq \, (L^{(1,q; \alpha, \beta)})_{m, \text{opt}}(0,1).\end{equation} 
Our aim is now to prove that
\begin{equation}\label{XE:optemb1H}
  (L^{(1,q; \alpha, \beta)})_{m, \text{opt}}(0,1) \,  \subsetneqq \, L^{1^*\!, q;  \alpha + \frac 1q,  \beta} (0,1)
 .\end{equation} 
 To this end, let us observe that if $(L^{(1,q; \alpha, \beta)})_{m, \text{opt}}(0,1) =   L^{1^*\!, q;  \alpha + \frac 1q,  \beta} (0,1) $, then
the fractional operator $T_{\frac{m}{n}}$  defined on functions $f \in L^{0}_{+}(0,1)$ as
\begin{equation}\label{D:frac}
T_{\frac{m}{n}}f(s) = s^{- \frac{m}{n}} \sup\limits_{t \in [s, 1]} t^{\frac{m}{n}}f^{\ast}(t) \quad \  \text{for} \ s \in [0,1]
\end{equation}
 had to be bounded on $(L^{1^*\!, q; \alpha+\frac 1q, \beta})'(0,1)  =  L^{\frac{n}{m}, q'; -\alpha-\frac 1q, -\beta} (0,1)$ by \cite{ZM2}*{Theorem~4.6}. Then, by \cite{GOP}*{Theorem~3.2~(i)}, such a boundedness had to be in turn equivalent to the  
  existence of some positive constant $C$  such that 
\begin{equation}\label{E:GOPaX} x^{\frac{m}{n}} \left(\int_0^{x}   \omega(t) \ell(t) \, \d  t  \right)^{\frac 1{q'}} \, \leq  \, C  \,\left(\int_0^{x}    v(t) \, \d  t  \right)^{\frac 1{q'}} \quad \ \text{for every}  \    x\in  (0,\infty) ,\end{equation} with $\omega, v$ being the weights described in \eqref{E:GOPa1}. But, since
  \begin{equation*}
   x^{\frac{m}{n}} \left(\int_0^{x}  \omega(t) \ell(t) \, \d  t  \right)^{\frac 1{q'}} \approx \begin{cases}\infty &   \  \text{if} \, \begin{cases}
    \alpha + 1 \in \left( \frac{1}{q'}, \frac{2}{q'}\right), \, \beta   \in \R\\[0.2ex] 
      \alpha + 1 = \frac{2}{q'}, \, \beta \leq \frac{1}{q'}  
 \end{cases}\\[0.2ex]
 x^\frac mn  \ell\ell^{\frac1{q'}-\beta}(x)\chi_{(0,1)}(x) & \ \text{if}  \  \,   \alpha + 1 = \frac{2}{q'}, \, \beta > \frac{1}{q'} \\[0.2ex]
  x^\frac mn  \ell^{\frac2{q'}-(\alpha+1)}(x)\ell\ell^{-\beta}(x)\chi_{(0,1)}(x) & \ \text{if}  \  \,   \alpha + 1 > \frac{2}{q'}, \, \beta \in \R 
       \end{cases} 
\end{equation*}
and
\begin{equation*}
    \left(\int_0^{x}    v(t) \, \d  t  \right)^{\frac 1{q'}} \, \approx \, x^\frac mn \ell^{\frac 1 {q'} - (\alpha +1)}(x) \ell\ell^{-\beta}(x) 
\end{equation*}  for every $ x\in  (0,1)$, 
condition   \eqref{E:GOPaX} plainly fails for any positive constant $C$. Therefore, one necessarily has that $(L^{(1,q; \alpha, \beta)})_{m, \text{opt}}(0,1) \neq   L^{1^*\!, q;  \alpha + \frac 1q,  \beta} (0,1) $. The assertion \eqref{XE:optemb1H}  hence  follows from coupling this piece of information and the second embedding in \eqref{E:optemb1}.\\ 
Notice that, in the light of \eqref{E:optemb1}-\eqref{E:optemb1bis:2},  \eqref{XE:optemb1BB}, \eqref{XE:optemb1H} and \cite{OP}*{Theorem~4.5},  
 $L^{1^*\!, q;  \alpha + \frac 1q,  \beta} (\Omega)$ is the optimal GLZ~space of the first exponent $1^\ast$ and the second exponent $q$ in the embedding \eqref{E:OptRI2}\textsubscript{5} for any John domain $\Omega$ of $\R^n$. 
 
We have to show that   $L^{1^*\!, q;  \alpha + \frac 1q,  \beta} (\Omega)$ is actually the optimal GLZ~space of the second exponent $q$ in the embedding just mentioned. To this purpose, let $L^{\tilde{p}, q; \tilde{\alpha}, \tilde{\beta}}(0,1)$ be any GLZ~space  such that $\tilde{p} \in [1, \infty]$, with $\tilde{p} \neq 1^\ast$,    $\tilde{\alpha}$,
$\tilde{\beta} \in \R$ satisfy  one of the alternatives in \eqref{c:LZnorm}, where $p$, $\alpha$ and $\beta$ are replaced with $\tilde{p}$, $\tilde{\alpha}$ and $\tilde{\beta}$.\\
We observe that, thanks to \eqref{E:embequ2a}, if $W^{m}L^{(1,q; \alpha, \beta)}(\Omega)  \hookrightarrow L^{\tilde{p}, q; \tilde{\alpha}, \tilde{\beta}}(\Omega)$, then the first embedding in \eqref{E:optemb1}  ensures that $L^{1^*\!, q;  \alpha + 1,  \beta} (0,1) \hookrightarrow L^{\tilde{p}, q; \tilde{\alpha}, \tilde{\beta}}(0,1)$. Thereby, from \cite{OP}*{Theorem~4.6}, we  infer  that $\tilde{p} < 1^\ast$. With this restriction on $\tilde{p}$,  it immediately follows from \cite{OP}*{Theorem~4.6}  that
$
 L^{1^\ast \!, q; \alpha + \frac 1q, \beta}(0,1)  \hookrightarrow   L^{\tilde{p}, q; \tilde{\alpha}, \tilde{\beta}}(0,1)  
$. The optimality of the target space $L^{1^*\!, q;  \alpha + \frac 1q,  \beta} (\Omega)$ in \eqref{E:OptRI2}\textsubscript{5} among all GLZ~spaces of the second exponent $q$ is thus established.

In order to conclude the proof of the case $q\in (1,\infty]$, $\alpha > -\frac{1}{q}$ and $\beta \in \R$, it remains to show that    $(L^{(1,q; \alpha, \beta)})_{m, \text{opt}}(0,1)$ is not a GLZ~space, namely \begin{equation}\label{OCTAA}
 (L^{(1,q; \alpha, \beta)})_{m, \text{opt}}(0,1) \,  \neq  \,   L^{\tilde{p}, \tilde{q};  \tilde{\alpha}, \tilde{\beta}} (0,1) 
\end{equation}  for all parameters $\tilde{p}$, $\tilde{q}$,   $\tilde{\alpha}$ and $\tilde{\beta}$ satisfying one of the    alternatives \eqref{c:LZnorm} in place of $p$, $q$,   $\alpha$ and $\beta$. Let us assume, for the sake of contradiction, that the equality \begin{equation}\label{OCT2}
 (L^{(1,q; \alpha, \beta)})_{m, \text{opt}}(0,1) \,  =  \,   L^{\tilde{p}, \tilde{q};  \tilde{\alpha}, \tilde{\beta}} (0,1), 
\end{equation}  holds, up to equivalent norms,  for some parameters $\tilde{p}$, $\tilde{q}$,   $\tilde{\alpha}$ and $\tilde{\beta}$.\\
We first note  that if $\tilde{q}=q$, then, by making use of the second embedding in \eqref{E:optemb1} combined with \eqref{XE:optemb1H}, we obtain that
\begin{equation*}
   L^{\tilde{p}, q; \tilde{\alpha}, \tilde{\beta}}(0,1) \, \hookrightarrow \,  L^{1^\ast\!, q; \alpha + \frac 1q, \beta}(0,1). 
\end{equation*} Thus, on taking into account   the optimality of the target space $L^{1^*\!, q;  \alpha + \frac 1q,  \beta} (\Omega)$ in \eqref{E:OptRI2}\textsubscript{5} among all GLZ~spaces of the second exponent $q$ and  \eqref{E:embequ2a}, we have that \begin{equation*} 
   L^{\tilde{p}, q; \tilde{\alpha}, \tilde{\beta}}(0,1) \, = \, L^{1^\ast\!, q; \alpha + \frac 1q, \beta}(0,1). 
\end{equation*} This, in view of   \eqref{XE:optemb1H}, contradicts the equation \eqref{OCT2} for any space $L^{\tilde{p}, q; \tilde{\alpha}, \tilde{\beta}}(0,1)$ of the second exponent     $q$.\\
Thus, owing to \eqref{E:optemb1},  \cite{OP}*{Theorems~4.5 and~4.6}   and   the fact that $q>1$, we shall necessarily   treat  only the cases when  $\tilde{p}=1^{\ast}$, and the parameters $\tilde{q}(\neq q)$,   $\tilde{\alpha}$ and $\tilde{\beta}$   fulfill one of the following alternatives (which depends on whether $q \in (1,2)$ or $q=2$ or $q \in (2, \infty)$ in the case when $\tilde{q} \in (q, \infty]$):
\begin{equation}\label{prec:desc:1}
\begin{cases} q \in (1,2), \ \tilde{q} \in (q, \frac q{2-q}), \ \tilde{\alpha} \in (\alpha + \frac 2q - \frac 1{\tilde{q}}, \alpha +1), \    \tilde{\beta}\in \mathbb{R}   \\[0.2ex]
q \in (1,2), \ \tilde{q} \in (q, \frac q{2-q}), \ \tilde{\alpha} = \alpha + \frac 2q - \frac 1{\tilde{q}}, \    \tilde{\beta}\in (\beta + \frac 1q - \frac 1{\tilde{q}}, \infty)   \\[0.2ex]
q \in (1,2), \ \tilde{q} \in (q, \frac q{2-q}), \ \tilde{\alpha} = \alpha + 1, \    \tilde{\beta}\in (-\infty, \beta]\\[0.3ex]
q = 2, \ \tilde{q} \in (2, \infty), \ \tilde{\alpha} \in (\alpha + 1 - \frac 1{\tilde{q}}, \alpha +1), \    \tilde{\beta}\in \mathbb{R}   \\[0.2ex]
q = 2, \ \tilde{q} \in (2, \infty), \ \tilde{\alpha} = \alpha + 1 - \frac 1{\tilde{q}}, \    \tilde{\beta}\in (\beta + \frac 12 - \frac 1{\tilde{q}}, \infty)   \\[0.2ex]
q = 2, \ \tilde{q} \in (2, \infty), \ \tilde{\alpha} = \alpha + 1, \    \tilde{\beta}\in (-\infty, \beta]\\[0.3ex] 
  q \in (2, \infty), \ \tilde{q} \in (q, \infty], \ \tilde{\alpha} \in (\alpha + \frac 2q - \frac 1{\tilde{q}}, \alpha +1), \    \tilde{\beta}\in \mathbb{R}   \\[0.2ex]
q \in (2, \infty), \ \tilde{q} \in (q, \infty], \ \tilde{\alpha} = \alpha + \frac 2q - \frac 1{\tilde{q}}, \    \tilde{\beta}\in (\beta + \frac 1q - \frac 1{\tilde{q}}, \infty)   \\[0.2ex]
q \in (2, \infty), \ \tilde{q} \in (q, \infty], \ \tilde{\alpha} = \alpha + 1, \    \tilde{\beta}\in (-\infty, \beta]; 
\end{cases}
\end{equation}
and
\begin{equation}\label{prec:desc:4}
\begin{cases} \tilde{q} \in (1, q), \ \tilde{\alpha} \in (\alpha + \frac 1q, \alpha +1 + \frac 1q - \frac 1{\tilde{q}}), \    \tilde{\beta}\in \mathbb{R}   \\[0.2ex]
\tilde{q} \in (1, q), \ \tilde{\alpha} = \alpha + \frac 1q, \    \tilde{\beta}\in (\beta, \infty)   \\[0.2ex]
\tilde{q} \in (1, q), \ \tilde{\alpha} = \alpha +1 + \frac 1q - \frac 1{\tilde{q}}, \    \tilde{\beta}\in (-\infty, \beta + \frac 1q - \frac 1{\tilde{q}})  \\[0.2ex]
\tilde{q} \in (1, q), \ \tilde{\alpha} = \alpha + \frac 1q, \    \tilde{\beta} = \beta.
\end{cases}
\end{equation}
When \eqref{prec:desc:1}  and \eqref{prec:desc:4}\textsubscript{1-3} are in force,  combining the properties  \eqref{E:optemb1bis}, \eqref{E:optemb1bis:2} and \cite{OP}*{Theorem~4.5} ensures the existence of some space $L^{1^\ast\!, q; \epsilon, \gamma}(0,1)$ such that
\begin{align*}
  (L^{(1,q; \alpha, \beta)})_{m, \text{opt}}(0,1) \,  \not \hookrightarrow \, L^{1^\ast\!, q; \epsilon, \gamma}(0,1) \qquad  \text{and} \qquad       L^{1^*\!, \tilde{q};  \tilde{\alpha}, \tilde{\beta}} (0,1) \, \hookrightarrow \, L^{1^\ast\!, q; \epsilon, \gamma}(0,1),  
\end{align*} 
a contradiction to equation \eqref{OCT2}.\\
It only remains to prove that     \eqref{OCT2} fails in the  case   \eqref{prec:desc:4}\textsubscript{4}. It suffices to show that
\begin{equation*}
    L^{1^\ast\!, \tilde{q}; \alpha + \frac 1q, \beta}(0,1) \not \hookrightarrow (L^{(1,q; \alpha, \beta)})_{m, \text{opt}}(0,1),
\end{equation*} or equivalently, via \eqref{E:embequ2b},   \begin{equation}\label{notemb:1}
   (L^{(1,q; \alpha, \beta)})'_{m, \text{opt}}(0,1)  \,  \not \hookrightarrow \,  (L^{1^\ast\!, \tilde{q}; \alpha + \frac 1q, \beta})'(0,1).  
\end{equation} 
Take $\gamma > 0$, and define
\begin{equation*}
    f_{\gamma}(t) = t^{-\frac mn}\ell^{ \gamma}(t)   \quad \  \text{for} \ t \in (0,1).
\end{equation*} This function    is equivalent to a non-increasing function, so    $f^*_{\gamma} \approx f_{\gamma}$  a.e.~on $(0,1)$. Moreover,   $f_{\gamma}^{\ast \ast}  \approx f_{\gamma}$ by  \cite{ED}*{Proposition~3.4.33~(v)}, and $S_m f_{\gamma}  \approx \ell^{ \gamma}   $ in $(0,1)$.\\
On account of \eqref{E:EOPJ2},   \eqref{c:ASSnorm2}\textsubscript{1} and \eqref{c:ASSnorm}\textsubscript{3}, a direct computation shows that 
 \begin{align*} 
    \lVert f_{\gamma} \rVert_{(L^{(1,q; \alpha, \beta)})'_{m, \text{opt}}(0,1)} \approx \lVert \ell^{ \gamma}  \rVert_{L^{\infty, q'; -\alpha -1, -\beta }(0,1)} < \infty  
\end{align*}
  and  
\begin{align*} 
    \lVert f_{\gamma} \rVert_{(L^{1^\ast\!, \tilde{q}; \alpha+\frac 1q, \beta})'(0,1) } \approx \lVert \ell^{ \gamma} \rVert_{L^{ \infty ,\tilde{q}'; -\alpha - \frac 1q, -\beta}(0,1) }   = \infty 
\end{align*}
  for every positive
$\gamma \in \left(\alpha + \frac 1q - \frac 1{\tilde{q}'}, \alpha + \frac 1q\right)$.
This proves \eqref{notemb:1}, and the property \eqref{OCTAA} is therefore established.
The proof of the case $q\in (1,\infty]$, $\alpha > -\frac{1}{q}$ and $\beta \in \R$ is thus complete.
 
 The remaining cases, namely if either $q\in (1,\infty]$, $\alpha = - \frac{1}{q}$ and  $\beta > - \frac{1}{q}$ or $q \in (1,\infty)$ and $\alpha = \beta = - \frac{1}{q}$, follow from similar arguments. We thus limit ourselves to  outlining the necessary modifications. 
 
 When $q\in (1,\infty]$, $\alpha = - \frac{1}{q}$ and  $\beta > - \frac{1}{q}$, one has that 
 \begin{align}\label{E:optemb2Z}
( L^{1^*\!, q; 0, \beta +   \frac 1q})' (0,1) \,  \hookrightarrow \, (L^{(1,q; -\frac 1q, \beta)})'_{m, \text{opt}}(0,1) \,  \hookrightarrow \, (L^{1^*\!, q;   \frac{1}{q'}, \beta + 1})' (0,1).    
\end{align} Indeed,
\begin{equation*}\label{E:njp30} 
\begin{aligned}
\left\|   f \right\|_{\left(L^{(1,q; - \frac{1}{q}, \beta)}\right)_{m, \text{opt}}'(0,1)}   \, 
&\approx \, \left\|S_m f \right\|_{{L^{\infty,q';- \frac{1}{q'}, -\beta-1}(0,1)}} \, \gtrsim  \,   \left\lVert     f  \right\rVert_{L^{\frac nm, q'; - \frac{1}{q'}, - \beta-1}(0,1)}\\
& \approx \,   \left\lVert     f  \right\rVert_{\left(L^{1^*\!, q; \frac{1}{q'},  \beta+1}\right)'(0,1)} 
\end{aligned}
\end{equation*}   and 
\begin{equation}\label{P:njp30}
 \left\|S_m f \right\|_{{L^{\infty,q';-\frac{1}{q'}, -\beta-1}(0,1)}}  \, 
 \lesssim \,     \left\lVert     f  \right\rVert_{L^{\frac nm, q'; 0, - \beta - \frac 1q}(0,1)}   \approx  \left\lVert     f  \right\rVert_{\left(L^{1^*\!, q; 0,  \beta + \frac 1q}\right)'(0,1)} 
\end{equation}
for every $f \in  L^0_+(0,1)$. The inequality in \eqref{P:njp30} holds  by \cite{GOP}*{Theorem~3.2~(i)}, applied with  the weights $\omega$ and $v$ in \eqref{E:GOPa1} replaced by
\begin{equation*}\label{E:GOPa30}\omega(t)= t^{-1} \ell^{-1}(t) \ell\ell^{- (\beta+1) q'}(t)\chi_{(0,1)}(t),\qquad   v(t)=t^{{\frac mn q'}-1}\ell\ell^{1-(\beta+1) q'}(t) \chi_{(0,1)}(t). \end{equation*}
 This choice guarantees that \eqref{E:GOPa2} holds. Moreover, on replacing $v$ by any weight of the form
\begin{align*}   &v_{\epsilon, \gamma}(t)=t^{{\frac mn q'}-1}\ell^{- \epsilon q'}(t)\ell\ell^{-\gamma q'}(t) \chi_{(0,1)}(t)  &&  \text{for}  \    \epsilon\in  \left(0,\tfrac{1}{q'}\right],\gamma\in  \mathbb{R} \\[-1ex]
\intertext{or}
& v_{\gamma}(t)=t^{{\frac mn q'}-1}\ell\ell^{1-(\gamma+1) q'}(t) \chi_{(0,1)}(t)  &&  \text{for}  \      \gamma > \beta,\end{align*}
the inequality $\lesssim$ in \eqref{E:GOPa2} fails. Hence,  \cite{GOP}*{Theorem~3.2~(i)} tells us that  
\begin{align}
  (L^{(1,q; - \frac{1}{q}, \beta)})_{m, \text{opt}}(0,1) \,  &\not \hookrightarrow \,  L^{1^{\ast}, q; \epsilon,  \gamma} (0,1)  && \text{for every}  \    \epsilon\in  \left(0,\tfrac{1}{q'}\right],   \gamma\in  \mathbb{R}, \label{E:optemb1ter}\\[-1ex]
 \intertext{and}  
 (L^{(1,q; - \frac{1}{q}, \beta)})_{m, \text{opt}}(0,1) \,  &\not \hookrightarrow \,  L^{1^{\ast}, q; 0,  \gamma+\frac 1q} (0,1)  &&\text{for every}  \      \gamma > \beta. \label{E:optemb1ter:2}
\end{align}
By virtue of \eqref{E:embequ2b} and \eqref{E:optemb2Z} one thus obtains that \begin{align}\label{E:optemb5}
L^{1^*\!, q;  \frac{1}{q'},  \beta+1} (0,1)\,  \hookrightarrow \, (L^{(1,q; - \frac{1}{q}, \beta)})_{m, \text{opt}}(0,1) \,  \hookrightarrow \,   L^{1^*\!, q;    0,  \beta + \frac 1q} (0,1) . 
\end{align}
In addition, with the first embedding in \eqref{E:optemb5}, the properties \eqref{E:embequ2} and \eqref{E:optemb1ter} with $\epsilon=\frac{1}{q'}$ and $\gamma=\beta +1$ at hand, one has that   
\begin{equation}\label{XE:optemb1DD} L^{1^*\!, q;  \frac{1}{q'},  \beta+1} (0,1)\,  \subsetneqq \, (L^{(1,q; -\frac 1q, \beta)})_{m, \text{opt}}(0,1) .\end{equation} 
Our task is now to prove that
\begin{equation}\label{XE:optemb1DDa}
  (L^{(1,q; -\frac 1q, \beta)})_{m, \text{opt}}(0,1) \,  \subsetneqq \, L^{1^*\!, q;   0,  \beta+ \frac 1q} (0,1)
 .\end{equation} 
 Assume, by contradiction, that $(L^{(1,q; -\tfrac{1}{q}, \beta)})_{m, \text{opt}}(0,1) =   L^{1^*\!, q;  0,  \beta +\frac 1q} (0,1) $. Then, the fractional operator $T_{\frac{m}{n}}$  described in \eqref{D:frac} is  bounded on $(L^{1^*\!, q;  0, \beta +  \frac 1q})'(0,1)  =  L^{\frac{n}{m}, q';  0, -\beta-\frac 1q} (0,1)$ by \cite{ZM2}*{Theorem~4.6}. By \cite{GOP}*{Theorem~3.2~(i)}, this is equivalent to the existence of some positive constant $C$ such that  
\begin{equation*}\label{E:GOPaX2} x^{\frac{m}{n}q'}  \int_0^{x}   t^{-1} \ell\ell^{1-(\beta+1)q'}(t)\, \d  t   \, \leq  \, C  \, \int_0^{x}    t^{\frac{m}{n}q'-1} \ell\ell^{1-(\beta+1)q'}(t)\, \d  t \,   \quad \ \text{for every}  \    x\in  (0,1) , \end{equation*} 
which is impossible. Hence,  taking into account the second embedding in \eqref{E:optemb2Z},  \eqref{XE:optemb1DDa} is established. 
 
\noindent Putting together \eqref{E:optemb1ter}-\eqref{XE:optemb1DDa}, \cite{OP}*{Theorem~4.5}  and \eqref{E:embequ2a},  we get that $L^{1^*\!, q;  0 ,  \beta + \frac 1q} (\Omega)$ is the optimal GLZ~space of the first exponent $1^\ast$ and the second exponent $q$ in the embedding  \eqref{E:OptRI2}\textsubscript{6} for any John domain $\Omega$ of $\R^n$, with $n \geq 2$. 

The space  $L^{1^*\!, q;  0 ,  \beta + \frac 1q} (\Omega)$ is actually the optimal target in the embedding  \eqref{E:OptRI2}\textsubscript{6} among all GLZ~spaces of the second exponent $q$. Indeed, if  $L^{\tilde{p}, q; \tilde{\alpha}, \tilde{\beta}}(\Omega)$ with $\tilde{p}, \tilde{\alpha}, \tilde{\beta}$ satisfying one of the alternatives in \eqref{c:LZnorm},  where $p$, $\alpha$ and $\beta$ are replaced with $\tilde{p}$, $\tilde{\alpha}$ and $\tilde{\beta}$, and $\tilde{p} \neq 1^\ast$, is a target space in the embedding \eqref{E:OptRI2}\textsubscript{6}, then -via \eqref{E:embequ2a}-\cite{OP}*{Theorem~4.6}  and the first embedding in  \eqref{E:optemb5}   yield that
\begin{equation*}
 L^{1^*\!, q;  0 ,  \beta + \frac 1q} (0,1)    \, \hookrightarrow \, L^{\tilde{p}, q; \tilde{\alpha}, \tilde{\beta}}(0,1). 
\end{equation*}

In addition, a slight modification of the argument in the proof of the previous case shows that the (existing) optimal rearrangement\hyp{}invariant target space $(L^{(1,q;  -\frac{1}{q}, \beta)})_{m, \text{opt}}(\Omega)$ is not a GLZ~space when $q\in (1,\infty]$, and $\beta > -\frac{1}{q}$.
This concludes the proof of this case.

When $q \in (1,\infty)$, $\alpha = \beta=- \frac{1}{q}$, one infers that 
\begin{align}\label{E:optemb3}
L^{1^*\!, q;  \frac{1}{q'}, \frac{1}{q'}, 1} (0,1)\,  \hookrightarrow \, (L^{(1,q; -\frac 1q, -\frac 1q)})_{m, \text{opt}}(0,1) \,  \hookrightarrow \,   L^{1^*\!, q; 0, 0,   \frac 1q} (0,1).  
\end{align}  Indeed,
\begin{equation*}\label{E:njp31}    
\begin{aligned}
\left\|   f \right\|_{\left(L^{(1,q; - \frac{1}{q}, - \frac{1}{q})}\right)_{m, \text{opt}}'(0,1)}   \, 
&\approx \, \left\|S_m f \right\|_{{L^{\infty,q';- \frac{1}{q'}, - \frac{1}{q'},-1}(0,1)}} \, \, \gtrsim \,   \,   \left\lVert     f  \right\rVert_{L^{\frac nm, q'; - \frac{1}{q'}, - \frac{1}{q'},  -1}(0,1)} 
\\
&\approx  \left\lVert     f  \right\rVert_{\left(L^{1^*\!, q; \frac{1}{q'},  \frac{1}{q'}, 1}\right)'(0,1)} 
\end{aligned}
\end{equation*}
and
\begin{equation}\label{P:njp31}
 \left\|S_m f \right\|_{{L^{\infty,q';-\frac{1}{q'}, -\frac{1}{q'}, -1}(0,1)}}  \, 
 \lesssim \,     \left\lVert     f  \right\rVert_{L^{\frac nm, q'; 0, 0, - \frac 1q}(0,1)}   \approx  \left\lVert     f  \right\rVert_{\left(L^{1^*\!, q; 0,  0, \frac 1q}\right)'(0,1)} 
\end{equation} for every  $f \in  L^0_+(0,1)$.
The inequality   \eqref{P:njp31} holds by \cite{GOP}*{Theorem~3.2~(i)}, applied with the weights $\omega$ and $v$ in \eqref{E:GOPa1} replaced by
 \begin{equation*}\label{E:GOPa31}\omega(t)= t^{-1} \ell^{-1}(t) \ell\ell^{-1}(t) \ell\ell\ell^{- q'}(t)\chi_{(0,1)}(t),\quad   v(t)=t^{{\frac mn q'}-1}\ell\ell\ell^{1- q'}(t) \chi_{(0,1)}(t).  \end{equation*}
With these choices, we have  \eqref{E:GOPa2} again.
Furthermore, on replacing $v$ by any weight of the form
\begin{align*}    &v_{\epsilon, \gamma , \theta}(t)=t^{{\frac mn q'}-1}\ell^{- \epsilon q'}(t)\ell\ell^{-\gamma q'}(t) \ell\ell\ell^{-\theta q'}(t)\chi_{(0,1)}(t)  &&  \text{for}  \     \epsilon\in  \left(0,\tfrac{1}{q'}\right],  \gamma, \theta\in  \mathbb{R}, \\[-1ex]
\intertext{or}
&  v_{\gamma , \theta}(t)=t^{{\frac mn q'}-1}\ell\ell^{-\gamma q'}(t) \ell\ell\ell^{-\theta q'}(t)\chi_{(0,1)}(t)   &&  \text{for}  \    \gamma\in  \left(0,\tfrac{1}{q'}\right],   \theta\in  \mathbb{R},\\[-1ex]
\intertext{or} 
& v_{\theta}(t)=t^{{\frac mn q'}-1} \ell\ell\ell^{1-(\theta+1) q'}(t)\chi_{(0,1)}(t)   &&  \text{for} \    \theta > 0,  \end{align*}
the inequality $\lesssim$   in \eqref{E:GOPa2} fails. 
So, an application of \cite{GOP}*{Theorem~3.2~(i)} yields that  
\begin{align} 
  (L^{(1,q; -\frac 1q, -\frac 1q)})_{m, \text{opt}}(0,1) \,  &\not \hookrightarrow \,   L^{1^\ast, q;  \epsilon, \gamma, \theta} (0,1)  && \text{for every}  \    \epsilon\in  \left(0,\tfrac{1}{q'}\right],  \gamma, \theta\in  \mathbb{R}, \label{E:optemb1q}\\
  (L^{(1,q; -\frac 1q, -\frac 1q)})_{m, \text{opt}}(0,1) \,  &\not \hookrightarrow \,   L^{1^\ast, q;  0, \gamma, \theta} (0,1)  &&  \text{for every}  \    \gamma\in  \left(0,\tfrac{1}{q'}\right],   \theta\in  \mathbb{R}\label{E:optemb1q:2}\\[-1ex] 
 \intertext{and} 
 (L^{(1,q; -\frac 1q, -\frac 1q)})_{m, \text{opt}}(0,1) \,  &\not \hookrightarrow \,   L^{1^\ast, q;  0, 0, \theta+\frac 1q} (0,1) &&\text{for every}  \     \theta > 0. \label{E:optemb1q:3}
\end{align}
Combining \eqref{E:optemb1q} in the case $\epsilon=\frac{1}{q'}$, $\gamma=\frac{1}{q'}$ and $\theta = 1$ with the first embedding in \eqref{E:optemb3}, one gets -via \eqref{E:embequ2}- that
\begin{equation}\label{XE:optemb1DD:1} L^{1^*\!, q;  \frac{1}{q'}, \frac{1}{q'}, 1} (0,1)\,  \subsetneqq \, (L^{(1,q; -\frac 1q, -\frac 1q)})_{m, \text{opt}}(0,1) .\end{equation} 
As in the proofs of the previous cases, our aim is now to show that
\begin{equation}\label{XE:optemb1DDa:1}
  (L^{(1,q; -\frac 1q, -\frac 1q)})_{m, \text{opt}}(0,1) \,  \subsetneqq \, L^{1^*\!, q;   0,  0, \frac 1q} (0,1)
.\end{equation} 
\noindent  If $(L^{(1,q; -\frac{1}{q}, -\frac 1q)})_{m, \text{opt}}(0,1) =   L^{1^*\!, q;  0, 0, \frac 1q} (0,1)$, then
the fractional operator $T_{\frac{m}{n}}$  described in \eqref{D:frac} would be bounded on   $(L^{1^*\!, q;  0, 0,   \frac 1q})'(0,1)  =  L^{\frac{n}{m}, q';  0,0,-\frac 1q} (0,1)$ by \cite{ZM2}*{Theorem~4.6}, or equivalently -via \cite{GOP}*{Theorem~3.2~(i)}-  a positive constant $C$ would exist  such that  
\begin{equation*}\label{E:GOPaX2:1} x^{\frac{m}{n}q'}  \int_0^{x}   t^{-1}\ell\ell\ell^{1-q'}(t) \, \d  t   \, \leq  \, C  \, \int_0^{x}    t^{\frac{m}{n}q'-1}\ell\ell\ell^{1-q'}(t) \, \d  t \,   \quad \ \text{for every}  \    x\in  (0,1) , \end{equation*} which is impossible. Hence, in view of the second embedding in \eqref{E:optemb3},  \eqref{XE:optemb1DDa:1} is established.

Observe that  combining \eqref{E:optemb3}, \eqref{E:optemb1q}-\eqref{XE:optemb1DDa:1} and \cite{PLK}*{Theorem~3.17}  yields that   $L^{1^*\!, q;  0,  0, \frac 1q} (\Omega)$ is the optimal  five\hyp{}parameter GLZ~space of the first exponent $1^\ast$ and the second exponent $q$ in the embedding  \eqref{E:OptRI2}\textsubscript{7} for any John domain $\Omega$ of $\R^n$, with $n \geq 2$.

The space  $L^{1^*\!, q;  0,  0, \frac 1q} (\Omega)$  is actually the optimal target in the embedding  \eqref{E:OptRI2}\textsubscript{7} among all five\hyp{}parameter GLZ~spaces of the second exponent $q$. To see this, take any
 space $L^{\tilde{p}, q; \tilde{\alpha}, \tilde{\beta}, \tilde{\gamma}}(\Omega)$  
with $\tilde{p}, \tilde{\alpha}, \tilde{\beta}, \tilde{\gamma}$ satisfying one of the alternatives  \eqref{c:L1five},  where $p$, $\alpha$,  $\beta$ and $\gamma$ are replaced with $\tilde{p}$, $\tilde{\alpha}$, $\tilde{\beta}$ and $\tilde{\gamma}$, and $\tilde{p} \neq 1^\ast$. Note that if 
 $W^{m}L^{(1,q; \alpha, \beta)}(\Omega)  \hookrightarrow L^{\tilde{p}, q; \tilde{\alpha}, \tilde{\beta}, \tilde{\gamma}}(\Omega)$, then    $L^{1^*\!, q;  \frac{1}{q'}, \frac{1}{q'}, 1} (0,1) \hookrightarrow L^{\tilde{p}, q; \tilde{\alpha}, \tilde{\beta}, \tilde{\gamma}}(0,1)$ by \eqref{E:embequ2a} and the first embedding in \eqref{E:optemb3}. So \cite{PLK}*{Theorem~3.17} gives   that $\tilde{p} < 1^\ast$, and a further application of the latter theorem  implies that
 $L^{1^*\!, q;  0,  0, \frac 1q} (0,1) \, \hookrightarrow \, L^{\tilde{p}, q; \tilde{\alpha}, \tilde{\beta}, \tilde{\gamma}}(0,1)$.

It thus remains to prove that $(L^{(1,q;  -\frac{1}{q},  -\frac{1}{q})})_{m, \text{opt}}(0,1)$ is not a five\hyp{}parameter GLZ~space when $q\in (1,\infty)$. The outline of the proof is similar to that of the previous cases. The novel aspect is due to the presence of the fifth parameter. For the reader's convenience, we insert the proof.

\noindent Firstly, we notice that the validity -up to equivalent norms- of the equality
\begin{equation}\label{OCT20}(L^{(1,q;  -\frac{1}{q},  -\frac{1}{q})})_{m, \text{opt}}(0,1) \, = \, L^{\tilde{p}, \tilde{q}; \tilde{\alpha}, \tilde{\beta}, \tilde{\gamma}}(0,1),\end{equation} 
for some  $\tilde{q} \in [1, \infty] $ and $\tilde{p} \in (1, \infty]$,   $\tilde{\alpha}, \tilde{\beta}, \tilde{\gamma} \in \mathbb{R}$ satisfying   one of the alternatives  \eqref{c:L1five}, forces $\tilde{q} \neq q$.  Indeed, the fact that the space $L^{1^\ast\!, q; 0, 0, \frac 1q}(\Omega)$ is the optimal five\hyp{}parameter GLZ~space of the second exponent $q$ in the embedding \eqref{E:OptRI2}\textsubscript{7} entails that
\begin{equation}\label{opt:eqv:1}
 L^{1^\ast\!, q; 0, 0, \frac 1q}(0,1) \, \hookrightarrow \,   L^{\tilde{p}, q; \tilde{\alpha}, \tilde{\beta}, \tilde{\gamma}}(0,1) 
\end{equation} whenever  $W^{m}L^{(1,q; \alpha, \beta)}(\Omega)  \hookrightarrow L^{\tilde{p}, q; \tilde{\alpha}, \tilde{\beta}, \tilde{\gamma}}(\Omega)$. This contradicts assumption \eqref{OCT20} for five\hyp{}parameter GLZ~spaces of the form $L^{\tilde{p}, q; \tilde{\alpha}, \tilde{\beta}, \tilde{\gamma}}(0,1)$ owing to
\eqref{XE:optemb1DDa:1}.
 
\noindent We thus focus on \eqref{OCT20} in the case when $\tilde{q} \neq q$. On applying   \eqref{E:optemb3} and  \cite{PLK}*{Theorem~3.17}, we get that assumption \eqref{OCT20}   forces $\tilde{p}$ to be $1^\ast$ and   $\tilde{q}( \neq q)$,   $\tilde{\alpha}, \tilde{\beta}, \tilde{\gamma} \in \mathbb{R}$ to fulfill
one of the following alternatives  (which depends on whether $q \in (1,2)$ or $q=2$ or $q \in (2, \infty)$ in the case when $\tilde{q} \in (q, \infty]$): 
\begin{equation}\label{condemb:3:1}
\begin{cases} q \in (1,2), \  \tilde{q} \in (q, \frac q{2-q}), \   \tilde{\alpha}  \in (\frac 1q - \frac 1{\tilde{q}}, \frac 1{q'}), \    \tilde{\beta}, \tilde{\gamma} \in \mathbb{R}   \\[0.2ex]  
 q \in (1,2), \ \tilde{q} \in (q, \frac q{2-q}), \   \tilde{\alpha}  =  \frac 1q - \frac 1{\tilde{q}}, \    \tilde{\beta} \in (\frac 1q - \frac 1{\tilde{q}}, \infty), \ \tilde{\gamma} \in \mathbb{R}   \\[0.2ex]
 q \in (1,2), \ \tilde{q} \in (q, \frac q{2-q}), \   \tilde{\alpha}  = \tilde{\beta} =   \frac 1q - \frac 1{\tilde{q}}, \  \tilde{\gamma} \in (\frac 2q - \frac 1{\tilde{q}}, \infty)   \\[0.2ex]
 q \in (1,2), \ \tilde{q} \in (q, \frac q{2-q}), \   \tilde{\alpha}  = \frac 1{q'}, \ \tilde{\beta} \in (-\infty, \frac 1{q'}), \  \tilde{\gamma} \in \R   \\[0.2ex]
 q \in (1,2), \ \tilde{q} \in (q, \frac q{2-q}), \   \tilde{\alpha}  = \tilde{\beta} =  \frac 1{q'}, \   \tilde{\gamma} \in (-\infty, 1]\\[0.3ex]
q = 2, \  \tilde{q} \in (2, \infty), \   \tilde{\alpha}  \in (\frac 12 - \frac 1{\tilde{q}}, \frac 12), \    \tilde{\beta}, \tilde{\gamma} \in \mathbb{R}   \\[0.2ex] q = 2, \ \tilde{q} \in (2, \infty), \   \tilde{\alpha}  =  \frac 12 - \frac 1{\tilde{q}}, \    \tilde{\beta} \in (\frac 12 - \frac 1{\tilde{q}}, \infty), \ \tilde{\gamma} \in \mathbb{R}   \\[0.2ex]
q = 2, \ \tilde{q} \in (2, \infty), \   \tilde{\alpha}  = \tilde{\beta} =   \frac 12 - \frac 1{\tilde{q}}, \  \tilde{\gamma} \in (1 - \frac 1{\tilde{q}}, \infty)   \\[0.2ex]
q = 2, \ \tilde{q} \in (2, \infty), \   \tilde{\alpha}  = \frac 12, \ \tilde{\beta} \in (-\infty, \frac 12), \  \tilde{\gamma} \in \R   \\[0.2ex]
q = 2, \ \tilde{q} \in (2, \infty), \   \tilde{\alpha}  = \tilde{\beta} =  \frac 12, \   \tilde{\gamma} \in (-\infty, 1]  \\[0.3ex]
 q \in (2, \infty), \  \tilde{q} \in (q, \infty], \   \tilde{\alpha}  \in (\frac 1q - \frac 1{\tilde{q}}, \frac 1{q'}), \    \tilde{\beta}, \tilde{\gamma} \in \mathbb{R}   \\[0.2ex]  
q \in (2, \infty), \ \tilde{q} \in (q, \infty], \   \tilde{\alpha}  =  \frac 1q - \frac 1{\tilde{q}}, \    \tilde{\beta} \in (\frac 1q - \frac 1{\tilde{q}}, \infty), \ \tilde{\gamma} \in \mathbb{R}   \\[0.2ex]
q \in (2, \infty), \ \tilde{q} \in (q, \infty], \   \tilde{\alpha}  = \tilde{\beta} =   \frac 1q - \frac 1{\tilde{q}}, \  \tilde{\gamma} \in (\frac 2q - \frac 1{\tilde{q}}, \infty)   \\[0.2ex]
q \in (2, \infty), \ \tilde{q} \in (q, \infty], \   \tilde{\alpha}  = \frac 1{q'}, \ \tilde{\beta} \in (-\infty, \frac 1{q'}), \  \tilde{\gamma} \in \R   \\[0.2ex]
q \in (2, \infty), \ \tilde{q} \in (q, \infty], \   \tilde{\alpha}  = \tilde{\beta} =  \frac 1{q'}, \   \tilde{\gamma} \in (-\infty, 1] 
\end{cases}
\end{equation} 
and  
\begin{equation}\label{condemb:3:4}
\begin{cases}   \tilde{q} \in (1,q), \   \tilde{\alpha}  \in (0, \frac 1{\tilde{q}'}), \    \tilde{\beta}, \tilde{\gamma} \in \mathbb{R}   \\[0.2ex]  
\tilde{q} \in (1, q), \   \tilde{\alpha}  =  0, \    \tilde{\beta} \in (0, \infty), \ \tilde{\gamma} \in \mathbb{R}   \\[0.2ex]
\tilde{q} \in (1,q), \   \tilde{\alpha}  = \tilde{\beta} =  0, \  \tilde{\gamma} \in (\frac 1q, \infty)   \\[0.2ex]
\tilde{q} \in (1, q), \   \tilde{\alpha}  = \frac 1{\tilde{q}'}, \ \tilde{\beta} \in (-\infty, \frac 1{\tilde{q}'}), \  \tilde{\gamma} \in \R   \\[0.2ex]
\tilde{q} \in (1,q), \   \tilde{\alpha}  = \tilde{\beta} =  \frac 1{\tilde{q}'}, \   \tilde{\gamma} \in (-\infty, \frac 1q + \frac 1{\tilde{q}'})\\[0.2ex]
\tilde{q} \in (1,q), \   \tilde{\alpha}  = \tilde{\beta} =  0, \  \tilde{\gamma} = \frac 1q .
\end{cases}
\end{equation} 
 If  one of the   alternatives \eqref{condemb:3:1}  and \eqref{condemb:3:4}\textsubscript{1-5} holds, then
\begin{align*}
 (L^{(1,q;  -\frac{1}{q},  -\frac{1}{q})})_{m, \text{opt}}(0,1) \,  \not \hookrightarrow \, L^{1^\ast\!, q; \epsilon, \gamma, \theta}(0,1) \qquad \text{and} \qquad  L^{1^*\!, \tilde{q};  \tilde{\alpha}, \tilde{\beta}, \tilde{\gamma}} (0,1) \hookrightarrow L^{1^\ast\!, q; \epsilon, \gamma, \theta}(0,1) 
\end{align*} for some space $L^{1^\ast\!, q; \epsilon, \gamma, \theta}(0,1)$   by virtue of \eqref{E:optemb1q}-\eqref{E:optemb1q:3} and \cite{PLK}*{Theorem~3.17}.
Thus,
\begin{equation*} 
 ( L^{(1,q;  -\frac{1}{q},  -\frac{1}{q})})_{m, \text{opt}}(0,1) \,  \not \hookrightarrow \,   L^{1^*\!, \tilde{q};  \tilde{\alpha}, \tilde{\beta}, \tilde{\gamma}} (0,1),
\end{equation*}
a contradiction to \eqref{OCT20}.

In the remaining case \eqref{condemb:3:4}\textsubscript{6}, take $\gamma<0$, and  let  
\begin{equation*}
    g_{\gamma}(t) = t^{-\frac mn}\ell^{\gamma}(t) \quad \  \text{for} \ t \in (0, 1).
\end{equation*} 
This function is equivalent to a non-increasing function on $(0,1)$, so $g_{\gamma}^{\ast}\approx g_{\gamma}$ a.e.~in $(0,1)$. Moreover,     $g_{\gamma}^{\ast \ast}\approx g_{\gamma} $ by \cite{ED}*{Proposition~3.4.33~(v)}, and  $S_mg_{\gamma} \approx 1$ (being $\gamma < 0$) in $(0,1)$. \\
Then, owing to \eqref{E:EOPJ2}, \eqref{c:ASSnorm2}\textsubscript{3} and \eqref{N5}, one clearly has that
\begin{equation}\label{fin:1}
\lVert g_{\gamma} \rVert_{( L^{(1,q;  -\frac{1}{q},  -\frac{1}{q})})'_{m, \text{opt}}(0,1)} \approx \lVert S_m g_{\gamma} \rVert_{L^{\infty, q'; -\frac1{q'}, -\frac 1{q'}, -1}(0,1)} < \infty.
\end{equation} 
On the other hand, under the additional assumption that $\gamma \in (-\frac 1{\tilde{q}'}, 0)$,    \cite{PLK}*{Theorem~3.28}  yields that 
\begin{equation}\label{inf:1}
\lVert g_{\gamma} \rVert_{(L^{1^\ast\!, \tilde{q}; 0,0,\frac 1q})'(0,1)} \approx\lVert g_{\gamma} \rVert_{L^{\frac nm, \tilde{q}'; 0,0,- \frac 1q}(0,1)} \approx \left( \int_0^1 t^{-1} \ell^{\gamma \tilde{q}'}(t) \ell\ell\ell^{- \frac{\tilde{q}'}{q}}(t) \, \d t  \right)^{\frac 1{\tilde{q}'}} = \infty.
\end{equation}
Consequently,
\begin{equation*} 
 (L^{(1,q;  -\frac{1}{q},  -\frac{1}{q})})'_{m, \text{opt}}(0,1) \, \not \hookrightarrow \,(L^{1^\ast\!, \tilde{q}; 0,0,\frac 1q})'(0,1) 
\end{equation*} or, equivalently,
\begin{equation}\label{thr:emb}
 L^{1^\ast\!, \tilde{q}; 0,0,\frac 1q}(0,1) \, \not \hookrightarrow \, ( L^{(1,q;  -\frac{1}{q},  -\frac{1}{q})})_{m, \text{opt}}(0,1). 
\end{equation} The conclusion \eqref{thr:emb}    contradicts \eqref{OCT20} in the case when $\tilde{q} \in (1,q)$$, $\   $\tilde{\alpha}  = \tilde{\beta} =  0$ and $\tilde{\gamma} = \frac 1q$. 
We have thus established that    $(L^{(1,q;  -\frac{1}{q},  -\frac{1}{q})})_{m, \text{opt}}(0,1)$ is not a five\hyp{}parameter GLZ~space when $q\in (1,\infty)$.

Hence the proof of the case $q\in (1,\infty)$, $\alpha = \beta=  -\frac{1}{q}$ is complete, and Part~{\sc(B)} is fully proved.
 \qed

\subsection{Proofs of embeddings into H\"older type spaces}\label{HOL} 
 
\vspace{1mm}
For the reader’s convenience, we recall here the content of \cite{Monia1}*{Theorem~3.4} that will be used in  the proofs of Theorems~\ref{T:HOp} and~\ref{T:HO*} with $X(0,1)=L^{p, q; \alpha, \beta}(0,1)$  and $X(0,1)=L^{(1, q; \alpha, \beta)}(0,1)$, respectively. As pointed out at the end of Section~\ref{Sec:Jones}, although \cite{Monia1}*{Theorem~3.4} assumes that  $\Omega$ is a bounded Lipschitz domain in $\R^n$, the same conclusion holds even when $\Omega$ fits in the Jones class.   

Given any $m \in \N$, with $m \leq n$, and any r.i.~function norm $\|\cdot\|_{X(0,1)}$, let $\vartheta_{m,X}\colon (0, \infty) \to (0, \infty]$, with $1\leq m < n$, be the function defined as 
\begin{equation}\label{theta}
\vartheta_{m,X} (r) =  \left\|s^{-1+\frac{m}{n}}\chi_{(0, r^n)}(s)\right\|_{X'(0,1)} \quad \ \text{for}  \   r\in \left(0,\tfrac14\right] ,
\end{equation}
and  $ \vartheta_{m,X} (r) =  \vartheta_{m,X}\left(\frac 14\right)$ if $r\in \left(\frac14,\infty\right)$,  and  let  $\varrho_{m,X}\colon (0, \infty) \to (0, \infty]$, with $2\leq m \leq n$, be    defined   as
\begin{equation}\label{rho}
\varrho_{m,X} (r) = r \left\|s^{-1+\frac{m-1}{n}}\chi_{(r^n,1)}(s)\right\|_{X'(0,1)} \quad \ \text{for} \  r\in \left(0,\tfrac14\right],
\end{equation}
and   $\varrho_{m,X} (r) =  \varrho_{m,X}\left(\frac 14\right)$ if $r\in \left(\frac14,\infty\right)$. 
Set 
\begin{align}\label{feb31}
\widehat \sigma_{m,X}   
= \begin{cases}
\vartheta_{1,X}   &\quad \text{if} \ m=1
\\  \vartheta_{m,X}  + \varrho_{m,X}   &\quad \text{if} \   m \in \{2, \dots, n-1\} 
\\ \varrho_{n,X}   &\quad \text{if} \ m=n.
\end{cases} 
\end{align}
Whenever   
\begin{equation}\label{feb30}
\lim_{r\to 0^+} \widehat \sigma_{m,X} (r) = 0,
\end{equation} 
the function $\widehat \sigma_{m,X}$ is, up to multiplicative constants, the optimal modulus of continuity for the H\"older type embedding   
\begin{equation}\label{cont}    W^mX(\Omega)\, \hookrightarrow \, C^{0, \widehat \sigma_{m,X}(\cdot)}(\Omega) 
\end{equation}
to hold.
This means that if there exists a modulus of continuity  $\sigma$ such that $W^{m}L^{p, q; \alpha, \beta}(\Omega) \hookrightarrow C^{0, \sigma(\cdot)}(\Omega)$, then $C^{0,\widehat \sigma_{m,X}(\cdot)}(\Omega) \hookrightarrow  C^{0,\sigma(\cdot)}(\Omega)$.  

In the proofs offered below, we shall make use of the following facts many times. 
Fix any  $r\in \left(0,\frac14\right]$. Then 
\begin{equation}\label{E:optH1}
\left((\cdot)^{-1+\frac {m}n}\chi_{(0,r)}(\cdot)\right)^{**}(t) \, \approx \,  t^{-1+\frac {m}n}\chi_{(0,r)}(t)   + r^{\frac mn} t^{-1} \chi_{(r,1)}(t)  \quad \  \text{for} \ t\in (0,1) 
\end{equation} 
when $m \in [1,  n-1]$, whereas if $m \in [2,  n]$, then
\begin{equation}\label{E:optHma}
\left((\cdot)^{-1+\frac {m-1}n}\chi_{(r,1)}(\cdot)\right)^*(s)=(s+r)^{-1+ \frac {m-1}n}\chi_{(0,1-r)}(s) \quad \  \text{for} \ s \in (0,1),
\end{equation}
and, since $r<1-r$,
\begin{equation}\label{E:optHmb}
\left((\cdot)^{-1+\frac {m-1}n}\chi_{(r,1)}(\cdot)\right)^{**}(t) \,  \approx \,  r^{-1+\frac {m-1}n}\chi_{(0,r)}(t)   + t^{-1+\frac {m-1}n}  \chi_{(r,1-r)}(t) + t^{-1}  \chi_{(1-r,1)}(t)           
\end{equation}  
for $t\in (0,1)$. Here, and throughout this section, the relation \lq\lq $\approx$'' holds up to constants independent of $r$.
  
\vspace{2mm}

\noindent{\textsc{Proof of Theorem~\ref{T:HOp}}.}  Let $p, q\in[1,\infty]$, $\alpha, \beta \in \R$ fulfill one of the alternatives~\eqref{c:LZnorm}. 
The statement follows from \cite{Monia1}*{Theorem~3.4}  recalled above, once has been detailed the behavior of the functions \eqref{theta}-\eqref{feb31} as $r \to 0^+$ in the case when $X(0,1)=L^{p, q; \alpha, \beta}(0,1)$.
Throughout this proof, for ease of notation, we set
 $$ \vartheta_m =\ \vartheta_{m, L^{p, q; \alpha, \beta}},\qquad  \varrho_m =\varrho_{m, L^{p, q; \alpha, \beta}}, \qquad  \widehat \sigma_m = \widehat \sigma_{m, L^{p, q; \alpha, \beta}}.$$    

We first claim that, for each $m \in \N$ with $1\leq m \leq n-1$,  
\begin{align}\label{E:HOz}
\vartheta_m(r) \  \approx  \ 
\begin{cases}
\infty  
 &\quad  \text{if}   \, \begin{cases} p= q=1, \ \alpha=0, \  \beta \geq 0     
 \\ 
 p= q=1, \ \alpha>0, \ \beta \in \R    
 \\[0.2ex]  
 p \in \big(1, \tfrac nm\big) \!, \  q \in [1, \infty], \  \alpha,   \beta  \in \R  
 \\[0.2ex]
p = \tfrac nm, \  q \in [1, \infty], \  \alpha < \frac{1}{q'},  \  \beta  \in \R 
\\[0.2ex]
p = \tfrac nm, \  q \in [1, \infty], \  \alpha = \frac{1}{q'},  \  \beta  < \frac{1}{q'}  
\\[0.2ex]
p = \tfrac nm, \  q \in (1, \infty], \  \alpha = \beta = \frac{1}{q'}
\end{cases}
\\[0.5ex]    1  
&\quad  \text{if} \ \,p=\frac nm, \   q=1, \    \alpha  =   \beta = 0
\\[0.2ex]        
\ell\ell^{\frac{1}{q'} - \beta}(r)   
&\quad   \text{if} \ \, p=\frac nm, \  q\in[1,\infty], \   \alpha   = \frac{1}{q'}, \   \beta   > \frac{1}{q'}   
\\[0.2ex] 
\ell^{\frac{1}{q'}  -\alpha }(r) \ell\ell^{- \beta}(r)
 &\quad  \text{if} \ \, p=\frac nm,  \  q\in[1,\infty], \    \alpha   > \frac{1}{q'}, \ \beta \in \R  
\\[0.2ex]   
r^{m-\frac np}\ell^{   -\alpha }(r) \ell\ell^{- \beta}(r)
&\quad   \text{if} \ \,  p \in \left(\frac nm, \infty\right) \!, \ q\in[1,\infty], \    \alpha,    \beta \in \R      
\\[0.2ex] 
r^m \ell^{  -\frac{1}{q}  -\alpha }(r) \ell\ell^{- \beta}(r)
& \quad   \text{if} \ \, p=\infty,   \  q\in[1,\infty), \    \alpha   < -\frac{1}{q},  \ \beta \in \R   \\[0.2ex]  
r^m \ell\ell^{-\frac 1q -\beta} (r)
& \quad   \text{if} \  \, p=\infty,    \  q\in[1,\infty), \        \alpha   = -\frac{1}{q} , \   \beta    < -\frac{1}{q}  \\[0.2ex] 
r^m \ell^{ -\alpha }(r) \ell\ell^{- \beta}(r)
& \quad     \text{if} \, \begin{cases} p=q=\infty,    \    \alpha    < 0,  \ \beta \in \R  \\   p=q=\infty,    \        \alpha  = 0 , \   \beta   < 0    
\end{cases}   
\\[0.5ex] 
r^m  & \quad   \text{if} \ \,  p=  q=\infty,  \   \alpha=  \beta=0 
\end{cases} 
\end{align}
near $0$. 
  
\noindent Indeed, if    $p=q =1$, and   either $\alpha= 0$ and $\beta \geq 0$  or  $\alpha >0$ and $\beta \in \R$, or  $p \in (1, \infty)$, thanks to   \eqref{E:1pie}, one has  
 that   
 \begin{equation*} 
\vartheta_{m} \left(r^{\frac1n}\right)    
= \left\|s^{-1+\frac mn}\chi_{(0,r)}(s)\right\|_{L^{p', q'; -\alpha, -\beta}(0,1)}  
= \left\| \Psi_{\frac mn - \frac{1}{p}; \, q, \alpha, \beta}\right\|_{L^{q'}(0,r)}       \quad \      \text{near} \ 0,  
\end{equation*}
whence the alternative assertions \eqref{E:HOz}\textsubscript{1-10} plainly follow from \eqref{E:H10}.
 
\noindent   Now, assume that $p=\infty$ and $q\in [1,\infty)$. If  $\alpha< -\frac{1}{q}$,  in view of \eqref{E:optH1} and \eqref{E:1pie}, one infers that  
 \begin{equation*}
 \begin{aligned}
 \vartheta_{m}\left(r^{\frac1n}\right)   
 & = \left\|s^{-1+\frac{m}{n}}\chi_{(0,r)}(s)\right\|_{L\sp{(1, q'; -\alpha -1, -\beta)}(0,1)} \,   \approx \,   \left\|\Psi_{\frac{m}{n}; \, q, \alpha+1, \beta}\right\|_{L^{q'}(0,r)}  +  r^{\frac mn}\left\|\Psi_{0; \, q, \alpha+1, \beta}\right\|_{L^{q'}(r,1)} \\  & \approx \,  r^{\frac{m}{n}} \,  \ell^{-1-\alpha}(r)    \ell\ell^{-\beta}(r)  +  r^{\frac mn} \ell^{- \frac{1}{q}-\alpha}   \ell\ell^{-\beta}(r)  \ \approx \,  r^{\frac{m}{n}} \, \ell^{- \frac{1}{q}-\alpha}  (r)   \ell\ell^{-\beta}(r) \quad \  \text{near} \ 0. 
\end{aligned} 
\end{equation*}    
Notice that in order to obtain the second equivalence we have made use of  \eqref{E:H10}\textsubscript{8} -with $\alpha$ replaced by $\alpha+1$- for the first  addend. As far as the second addend is concerned, one has to treat the cases $q\in (1,\infty)$  and $q=1$ separately. In the former case,  take conclusion  \eqref{E:H11}\textsubscript{2} -with $\alpha$ replaced by $\alpha+1$- according to property \eqref{UU}. In the latter case $q=1$, it suffices to observe that 
\begin{equation}\label{new4}  
\left\|\Psi_{0;\,1, \alpha+1, \beta}  \right\|_{L\sp{\infty}(r,1)}   \,   \approx \,    \ell^{-\alpha-1}(r)\ell\ell^{-\beta}(r)  \quad \  \text{near} \ 0, 
\end{equation} in view of \eqref{new}\textsubscript{1} and \eqref{new}\textsubscript{3}. Thereby, the conclusion \eqref{E:HOz}\textsubscript{11} holds.

\noindent A similar argument also works in the case when $\alpha =- \frac{1}{q}$ and  $\beta <- \frac{1}{q}$, via natural modifications due to the different parameters involved. In fact,
\begin{equation*}
 \begin{aligned}
\vartheta_{m}\left(r^{\frac1n}\right)   
& = \left\|s^{-1+\frac{m}{n}}\chi_{(0,r)}(s)\right\|_{L\sp{(1, q'; -\frac 1 {q'}, -\beta-1)}(0,1)}   \,   \approx \,   \left\|\Psi_{\frac{m}{n}; \, q,  \frac 1 {q'}, \beta+1}\right\|_{L^{q'}(0,r)}   +   r^{\frac mn}\left\|\Psi_{0; \, q, \frac 1{q'},  \beta+1}\right\|_{L^{q'}(r,1)}\\
& \approx \, r^{\frac mn} \, \ell^{\frac 1q - 1}(r)  \ell\ell^{-\beta-1} (r) +  r^{\frac mn} \ell\ell^{-\frac 1q-\beta}(r)  \,   \approx \,  r^{\frac mn} \ell\ell^{-\frac 1q-\beta}(r) \quad \  \text{near} \ 0.
\end{aligned} 
\end{equation*}
In order to get the second equivalence one has to apply 
 \eqref{E:H10}\textsubscript{8} for the first addend and to observe that 
\begin{equation}\label{corr3}
\left\|\Psi_{0;\,q, \frac 1{q'}, \beta+1}  \right\|_{L\sp{q'}(r,1)}   \,   \approx \,    \ell\ell^{-\frac 1q-\beta}(r)  \quad \  \text{near} \ 0. 
\end{equation}
The property \eqref{corr3} is a consequence of   \eqref{E:H11}\textsubscript{3} -with $\beta$ replaced by $\beta +1$- and \eqref{UU} in the case when $q \in (1, \infty)$, and of \eqref{E:H11}\textsubscript{3} -with $\beta$ replaced by $\beta +1$- and \eqref{new}\textsubscript{2} when $q = 1$. The assertion \eqref{E:HOz}\textsubscript{12} is thus proved.
 
\noindent    
Finally, if $p = q= \infty$ and either  $\alpha<0$ or $\alpha=0$ and $\beta\leq 0$, coupling \eqref{E:1pie} and \eqref{E:H10}\textsubscript{8} ensures that
\begin{equation*}  
 \vartheta_{m} \left(r^{\frac1n}\right) \   
 = \  \left\|s^{-1+\frac{m}{n}}\chi_{(0,r)}(s)\right\|_{L\sp{1, 1; -\alpha, -\beta}(0,1)} 
 =  \left\|\Psi_{\frac mn; \, \infty,  \alpha, \beta}\right\|_{L^{1}(0,r)} \, \approx \, r^{\frac{m}{n}}  \,  \ell^{-\alpha}(r)  \ell\ell^{-\beta}(r) \quad \  \text{near} \ 0,  
\end{equation*}
namely, the properties  \eqref{E:HOz}\textsubscript{13-15} hold. Our claim \eqref{E:HOz} is thus verified.

\vspace{1mm}

Let us emphasize that
\begin{equation}\label{REMARK2} 
\lim_{r \to 0^+}  \vartheta_{m} (r)    =   0   \ \quad \text{only for those values} \ p, q, \alpha, \beta    \  \text{included in the formulas} \  (\ref{E:HOz})_{8-15}. 
\end{equation} 

\vspace{1mm}

Next, we claim that for each $m \in \N$ with $2\leq m \leq n$,  
 \begin{align}\label{E:Brho}
\varrho_m(r)  \, \approx \, 
\begin{cases}
\begin{aligned}    
\, & r^{m-n}  &\quad &   \text{if} \ \,   p=q=1, \ \alpha=\beta=0   &     
\\[0.2ex] 
\, & r^{m-\frac np} \ell^{-\alpha}(r) \ell\ell^{- \beta}(r)& &   \text{if} \, \begin{cases} p= q=1, \ \alpha=0, \  \beta > 0   
\\[0.2ex] 
 p= q=1, \ \alpha>0, \ \beta \in \R    
 \\[0.2ex]  
 p \in \big(1, \tfrac n{m-1}\big) \!, \  q \in [1, \infty], \   \alpha,   \beta  \in \R 
\end{cases}    
\\[0.5ex] 
&  r \, \ell^{   \frac{1}{q'}   -\alpha }(r) \ell\ell^{- \beta}(r) & & \text{if}   \  \,    p=\tfrac n{m-1}, \ q \in [1, \infty], \      \alpha   <  \tfrac{1}{q'}, \ \beta \in \R    
\\[0.2ex] 
&   r \, \ell\ell^{  \frac{1}{q'}   -\beta }(r)  & & \text{if}  \  \,   p=\tfrac n{m-1},  \ q \in [1, \infty], \   \alpha   =  \tfrac{1}{q'}, \    \beta  <  \tfrac{1}{q'}
\\[0.2ex] 
&    r \, \ell\ell\ell^{   \frac{1}{q'} }(r)  & & \text{if}     \  \,  p=\tfrac n{m-1}, \  q \in (1, \infty], \       \alpha  =   \beta = \tfrac{1}{q'}  
\\[0.2ex] 
 &    r  & &   \text{if} \, \begin{cases} p=\frac n{m-1}, \    q=1, \    \alpha  =   \beta = 0    
\\[0.2ex] 
p=\frac n{m-1}, \ q \in [1, \infty], \   \alpha= \frac{1}{q'}, \     \beta > \frac{1}{q'}  
\\[0.2ex]
p=\frac n{m-1}, \ q \in [1, \infty], \  \alpha> \frac{1}{q'}, \ \beta \in \R  
\\[0.2ex]
p \in \big(\frac n{m-1}, \infty\big)\!, \  q \in [1, \infty], \ \alpha,    \beta  \in \R   
\\[0.2ex]
p=\infty, \ q \in [1, \infty], \ \alpha    < -\frac1{q}, \ \beta \in \R   
\\[0.2ex]
p=\infty, \ q \in [1, \infty], \ \alpha    = -\frac1{q}, \     \beta< -\frac1{q} 
\\[0.2ex] 
p=q=\infty, \ \alpha=\beta=0
\end{cases}          
\end{aligned}
\end{cases} 
\end{align} near $0$.    

\noindent  In order to verify this claim, take first  $p=q =1$, and either $\alpha= 0$ and $\beta \geq 0$ or  $\alpha >0$ (and $\beta \in \R$). 
Via  \eqref{E:optHma} and  \eqref{E:1pie}, one has that 
\begin{align*}
\varrho_m\left(r^{\frac1n}\right)   
& = r^{\frac 1n}   \left\|s^{-1+\frac{m-1}{n}}\chi_{(r,1)}(s)\right\|_{L\sp{\infty, \infty; -\alpha, -\beta}(0,1)}   \\  
& \approx \    \max\left\{ r^{-1+\frac{m}{n}}\left\| \Psi_{0; \, 1, \alpha, \beta} \right\|_{L\sp{\infty}(0,r)}, \,  r^{\frac 1n} \left\| \Psi_{\frac {m-1}{n} -1; \,   1, \alpha, \beta}\right\|_{L\sp{\infty}(r,1-r)}\right\} \quad \  \text{near} \ 0.
\end{align*}  
Thus the assertions  \eqref{E:Brho}\textsubscript{1-3} hold  thanks to \eqref{E:H10}\textsubscript{5-7} and \eqref{E:H11}\textsubscript{1}.
     
\noindent When $p \in (1, \infty)$, then, according to \eqref{E:optHma} and \eqref{E:1pie}, 
\begin{align*} 
\varrho_m \left(r^{\frac1n}\right)  
& = r^{\frac 1n} \|s^{-1+\frac{m-1}{n}}\chi_{(r,1)}(s)\|_{L\sp{p', q'; -\alpha, -\beta}(0,1)} \\  
&\approx  r^{-1+\frac{m}{n}}\left\| \Psi_{\frac{1}{p'}; \, q, \alpha, \beta}\right\|_{L\sp{q'}(0,r)} +  r^{\frac 1n} \left\| \Psi_{\frac{m-1}{n} - \frac 1p; \, q, \alpha, \beta} \right\|_{L\sp{q'}(r,1-r)}   \quad \  \text{near} \ 0.
\end{align*}   
The alternative conclusions \eqref{E:Brho}\textsubscript{4-11} are therefore straightforward consequences of coupling \eqref{E:H10}\textsubscript{8}   with \eqref{E:H11}.

 \noindent  Next, let $p =\infty$ and $q \in [1,\infty)$.
We initially assume that $\alpha  <- \frac 1q$. Then the properties \eqref{E:optHmb} and \eqref{E:1pie} entail that     
\begin{equation*}
\begin{aligned}
\varrho_m\left(r^{\frac1n}\right)   
& = r^{\frac1n} \left\|s^{-1+\frac{m-1}{n}}\chi_{(r,1)}(s)\right\|_{L\sp{(1, q'; -\alpha-1, -\beta)}(0,1)}\\ 
& \approx \,  r^{-1+\frac{m}{n}}\left\| \Psi_{1; \, q, \alpha+1, \beta}\right\|_{L\sp{q'}(0,r)} +  r^{\frac1n} \left\|\Psi_{\frac{m-1}{n}; \, q, \alpha+1, \beta}\right\|_{L\sp{q'}(r,1-r)} +  r^{\frac1n} \left\|\Psi_{0; \, q, \alpha+1, \beta} \right\|_{L\sp{q'}(1-r,1)} 
\end{aligned} 
\end{equation*} 
near $0$. 
Since, being  $\alpha  <- \frac 1q$, \begin{equation}\label{E:HOz1}\left\| \Psi_{0; \, q, \alpha+1, \beta}\right\|_{L\sp{q'}(1-r,1)} \  \approx \  1 \quad \  \text{near} \ 0\end{equation} for every  $\beta  \in \R$ (cf.~\eqref{new}\textsubscript{1} and \eqref{new}\textsubscript{3} in the case when $q=1$),  the assertion \eqref{E:Brho}\textsubscript{12} with $q \neq \infty$ can be deduced from \eqref{E:H10}\textsubscript{8}    and    \eqref{E:H11}\textsubscript{8}.  

\noindent When $\alpha = -\frac1q$ and $\beta <-\frac1q$,  via  \eqref{E:optHmb}  and \eqref{E:1pie},  one has that
\begin{equation*}
\begin{aligned}
\varrho_m\left(r^{\frac1n}\right) 
& =   r^{\frac1n} \left\|s^{-1+\frac{m-1}{n}}\chi_{(r,1)}(s)\right\|_{L\sp{(1, q'; -\frac{1}{q'}, -\beta-1)}(0,1)} \\ 
& \approx    \, r^{-1+\frac{m}{n}}\left\| \Psi_{1; \, q, \frac{1}{q'}, \beta+1}\right\|_{L\sp{q'}(0,r)} +  r^{\frac1n} \left\|\Psi_{\frac{m-1}{n}; \, q, \frac{1}{q'}, \beta+1}\right\|_{L\sp{q'}(r,1-r)} +  r^{\frac1n}\left\|\Psi_{0; \, q, \frac{1}{q'}, \beta+1} \right\|_{L\sp{q'}(1-r,1)}  
\end{aligned} 
\end{equation*}    
near $0$. Note that \begin{equation}\label{E:HOz2}\left\| \Psi_{0; \, q, \frac{1}{q'}, \beta+1}\right\|_{L\sp{q'}(1-r,1)} \  \approx \  1 \quad \  \text{near} \ 0\end{equation} since $\beta  <- \frac 1q$ (cf.~\eqref{new}\textsubscript{2}   in the case when $q=1$).
Thus the conclusion  \eqref{E:Brho}\textsubscript{13}  with $q \neq \infty$ follows from combining  \eqref{E:H10}\textsubscript{8},   \eqref{E:H11}\textsubscript{8}  and \eqref{E:HOz2}.

\noindent  Finally, suppose that $p =q =\infty$, and either $\alpha <0$ or $\alpha= 0$ and $\beta \leq 0$. Then, for these choices, 
\begin{equation}\label{E:optMc7H}
\begin{aligned}
\varrho_m\left(r^{\frac1n}\right)    
& = r^{\frac 1n}    \left\|s^{-1+\frac{m-1}{n}}\chi_{(r,1)}(s)\right\|_{L\sp{1, 1; -\alpha, -\beta}(0,1)} \\ 
& \approx \,  r^{-1+\frac{m}{n}}\left\| \Psi_{1; \, \infty, \alpha, \beta}\right\|_{L\sp{1}(0,r)} +    r^{\frac 1n} \left\|\Psi_{\frac {m-1}n;\,  \infty, \alpha, \beta}\right\|_{L\sp{1}(r,1-r)} 
 \ \approx \  r^{\frac{1}{n}}     \quad \ \text{near} \ 0. \end{aligned} 
\end{equation}  
The latter equivalence holds owing to \eqref{E:H10}\textsubscript{8} and  \eqref{E:H11}\textsubscript{8}. Thus, from \eqref{E:optMc7H}  we  get  the missing case  $q = \infty$ in \eqref{E:Brho}\textsubscript{12-13} as well as \eqref{E:Brho}\textsubscript{14}, ending the proof of the claim~\eqref{E:Brho}. 

\vspace{1mm}
             
We are now in position to conclude the proof. 
Part~\textsc{(I)} follows from the conclusions \eqref{E:HOz} and \eqref{REMARK2} with the choice $m=1$.
Part~\textsc{(II)} can be derived via the properties \eqref{E:HOz}, \eqref{REMARK2}  and  \eqref{E:Brho} under the assumption that $m \in \{2, \dots,  n-1\}$, through elementary computation involving possibly the following fact: for any $\varepsilon \in \R$, with $\varepsilon\neq0$,  
\begin{equation*} 
\lim_{r \to 0^+}   r^{\varepsilon}   \ell^{\alpha}(r)\ell\ell^{\beta}(r)\ell\ell\ell^{\gamma}(r)  =  \lim_{r \to 0^+}   r^{\varepsilon}   
\end{equation*} 
holds for every $\alpha, \beta, \gamma \in \R$. One has just to observe that,  under the current assumption on $m$,  
\begin{equation*} 
  \lim_{r \to 0^+}\vartheta_{m} (r) + \varrho_m (r)   =  0    \ \quad \text{only for those values} \ p, q, \alpha, \beta    \  \text{included in the formulas} \  (\ref{E:HOz})_{8-15},  
\end{equation*}namely,
\begin{equation*} 
\lim_{r \to 0^+}\widehat \sigma_m (r)  =  0    \ \quad \text{only for  those values of the parameters} \ p, q, \alpha, \beta    \  \text{included in the formula} \  (\ref{E:HOpBb}).  
\end{equation*} 
Finally,  Part~\textsc{(III)} is a consequence of the claim \eqref{E:Brho} in the case when $m=n$, taking into account that  
\begin{equation*} 
\lim_{r \to 0^+}  \varrho_{n} (r)    \neq  0     \ \quad \text{only for} \ p= q=1, \alpha= \beta=0. 
\end{equation*}

\vspace{-7mm}

\qed

\begin{remark}\label{R:forCamP} \rm  It is clear from the above proof that, for any $m \in \{2,\dots, n-1\}$, the behavior near $0$ of the relevant function $\widehat \sigma_{m}= \widehat \sigma_{m, L^{p, q; \alpha, \beta}}$  -described in \eqref{feb31} corresponding to a general rearrangement\hyp{}invariant space $X(0,1)$- is dictated  by the function $\varrho_{m}=\varrho_{m,L^{p, q; \alpha, \beta}}$ -accordingly defined in \eqref{rho}- when the parameters $p, q , \alpha, \beta$ fulfill one of the alternatives~\eqref{c:LZnorm} with $p >\frac nm$. The same phenomenon trivially occurs in the case when $m=n$, with no restriction on the range of the parameter $p$.
\end{remark}
    
\vspace{2mm}

\noindent{\textsc{Proof of Theorem~\ref{T:HO*}}.}   Assume that $q\in[1,\infty]$, $\alpha, \beta \in \R$ fulfill one of the  alternatives~\eqref{c:L1}. In the light of \cite{Monia1}*{Theorem~3.4} recalled at the beginning of this subsection, we need to show the precise behavior   as $r \to 0^+$ of the function, given by \eqref{feb31} with the choices that $m=n$ and $X(0,1)=L^{(1, q; \alpha, \beta)}(0,1)$, as well as that the necessary condition \eqref{feb30} is actually fulfilled.    
We simply denote it as $\widehat \omega_n$, namely, we set 
\begin{equation}\label{H=C}    
\widehat \omega_n =  \widehat \sigma_{n, L^{(1, q; \alpha, \beta)}} =  \varrho_{n, L^{(1, q; \alpha, \beta)}}.
\end{equation}   
For the time being, let $r\in \left(0,\frac14\right]$. 

If $q \in [1, \infty]$ and $\alpha    > -\frac{1}{q}$,  via   \eqref{E:optHma}  and \eqref{E:1pie}, one has that  
\begin{align*} 
\widehat \omega_n\left(r^{\frac1n}\right) \,   
&=   \,  r^{\frac 1n}\|s^{- \frac{1}{n}}\chi_{(r,1)}(s)\|_{L\sp{\infty, q'; -\alpha-1, -\beta}(0,1)}   \, \\ 
& \approx \, \left\|\Psi_{0; \, q, \alpha+1, \beta}\right\|_{L\sp{q'}(0,r)} +    r^{\frac 1n} \left\|\Psi_{-\frac 1n; \, q, \alpha+1, \beta}\right\|_{L\sp{q'}(r,1-r)}   \quad \ \text{near} \ 0.
\end{align*}
Thus, the conclusion \eqref{E:HO*B}\textsubscript{1} follows from an application of \eqref{E:H10}\textsubscript{7}   -with $\alpha$ replaced by $\alpha+1$- and \eqref{E:H11}\textsubscript{1}.   
 
\noindent For $q \in [1, \infty]$, $\alpha  = -\frac{1}{q}$ and $\beta   > -\frac{1}{q}$, the equations \eqref{E:optHma} and \eqref{E:1pie} assure that
\begin{align*} 
\widehat \omega_n\left(r^{\frac1n}\right) \,  
& =   \,  r^{\frac 1n}\|s^{- \frac{1}{n}}\chi_{(r,1)}(s)\|_{L\sp{\infty, q'; -\frac{1}{q'}, -\beta -1}(0,1)}  \, \\
&\approx \, \left\|\Psi_{0; \, q, \frac{1}{q'}, \beta+1}\right\|_{L\sp{q'}(0,r)} +   r^{\frac 1n}  \left\|\Psi_{- \frac{1}{n}; \, q, \frac{1}{q'}, \beta+1}\right\|_{L\sp{q'}(r,1-r)}    \quad \ \text{near} \ 0,  
\end{align*} 
whence the conclusion \eqref{E:HO*B}\textsubscript{2} is obtained by applying \eqref{E:H10}\textsubscript{6}   -with $\beta$ replaced by $\beta+1$- and \eqref{E:H11}\textsubscript{1}.
 
\noindent Finally, the conclusion \eqref{E:HO*B}\textsubscript{3} can be derived from Lemma~\ref{L:H2} upon observing that, for $q \in [1, \infty)$ and $\alpha = \beta= - \frac{1}{q}$, 
\begin{equation*} 
\begin{aligned}
\widetilde  \omega_{n}\left(r^{\frac1n}\right)     
&=   r^{\frac 1n}\|s^{- \frac{1}{n}}\chi_{(r,1)}(s)\|_{L\sp{\infty, q'; -\frac{1}{q'}, -\frac{1}{q'}, -1}(0,1)} \\   
& \approx  \|  \Psi_{0; \, q, \frac{1}{q'}, \frac{1}{q'}} \ell\ell\ell^{-1}\|_{L^{q'}(0,r)}    +   r^{\frac 1n}   \left\|\Psi_{- \frac{1}{n}; \, q, \frac{1}{q'}, \frac{1}{q'}} \ell\ell\ell^{-1}\right\|_{L^{q'}(r,1-r)}   \quad \ \text{near} \ 0, 
\end{aligned}
\end{equation*}  
thanks to \eqref{E:optHma}  and \eqref{E:1pie}. It remains only to note that, in each of the previous cases, the necessary condition \eqref{feb30}  clearly holds.   
\qed

\subsection{Proof of Embeddings into  Morrey and Campanato type spaces}\label{P:MorreyCamp}   At the outset, to keep the paper as self\hyp{}contained as possible, we recall the content of \cite{CCPS2}*{Theorem~2.2} here applied. Given any $m \in \N$, with $m < n$, and  any r.i.~function norm $\|\cdot\|_{X(0,1)}$,   the   function $\widetilde\varphi_{m,X}\colon(0,\infty)\to(0,\infty)$, defined as  
\begin{equation}\label{E:optMo}
\widetilde\varphi_{m,X}\left(r\right) = \|s^{-1+\frac{m}{n}}\chi_{(r^n,1)}(s)\|_{X'(0,1)} \quad  \text{if} \ \, r\in \left(0,\tfrac14\right] 
\end{equation}
and  $\widetilde\varphi_{m,X}(r) = \widetilde\varphi_{m,X}\left(\frac 14\right)$ if $r\in \left(\tfrac14,\infty\right)$, is an admissible function. Moreover,  the space $\mathcal M^{\widetilde\varphi_{m,X}(\cdot)}(\Omega)$ is the optimal (smallest possible) target Morrey space in the embedding 
\begin{equation*}
W^{m}X(\Omega) \to \mathcal M^{\widetilde\varphi_{m,X}(\cdot)}(\Omega) 
\end{equation*} for all John domains $\Omega$ of $\R^n$. 
This means that if there exists an admissible function $\varphi$ such that $W^{m}X (\Omega) \hookrightarrow \mathcal M^{\varphi(\cdot)}(\Omega)$, then $\mathcal M^{\widetilde\varphi_{m,X}(\cdot)}(\Omega) \hookrightarrow \mathcal M^{\varphi(\cdot)}(\Omega)$. 
  
For later reference, we further note that for each $r\in \left(0,\frac14\right]$ one has that   
\begin{equation}\label{E:optRa}
\left((\cdot)^{-1+\frac mn}\chi_{(r,1)}(\cdot)\right)^*(s)=(s+r)^{-1+\frac mn}\chi_{(0,1-r)}(s) \quad \  \text{for} \ s \in (0,1),
\end{equation}
and, since $r<1-r$,
\begin{equation}\label{E:optMoA}
\left((\cdot)^{-1+\frac mn}\chi_{(r,1)}(\cdot)\right)^{**}(t) \, \approx \,  r^{-1+\frac mn}\chi_{(0,r)}(t)   + t^{-1+\frac mn}  \chi_{(r,1-r)}(t) + t^{-1}  \chi_{(1-r,1)}(t)   \quad   \text{for} \    t\in (0,1).     
\end{equation} 

\vspace{2mm}

\noindent{\textsc{Proof of Theorem~\ref{T:Mor}}.}   
 {\sc(A)} The statement will follow from \cite{CCPS2}*{Theorem~2.2}, once we determine the exact behavior near $0$ of the function  \eqref{E:optMo} in the case when $X(0,1)=L^{p, q; \alpha, \beta}(0,1)$, for  $p, q\in[1,\infty]$, $\alpha, \beta \in \R$ satisfying one of the alternatives~\eqref{c:LZnorm}. Indeed, as remarked in Section~{2.6}, the behavior near $0$ is the only piece of information about an admissible function   $\widetilde\varphi_{m,L^{p, q; \alpha, \beta}}$ which is needed for the definition of the optimal space  $\mathcal M^{\widetilde\varphi_{m,L^{p, q; \alpha, \beta}}(\cdot)}(\Omega)$. 
 
For simplicity of notation, throughout this proof we denote $\widetilde\varphi_{m,L^{p, q; \alpha, \beta}}$ by $\widetilde\varphi_m$.  
We also suppose, for the time being, that  $r\in \left(0,\frac14\right]$. 

Consider first the case when  $p=q =1$, and  either  $\alpha= 0$ and $\beta \geq 0$ or  $\alpha >0$. Owing to \eqref{E:optRa} and \eqref{E:1pie}, one has that 
\begin{align*} 
\widetilde\varphi_m\left(r^{\frac1n}\right)    
& =  \|s^{-1+\frac{m}{n}}\chi_{(r,1)}(s)\|_{L\sp{\infty, \infty; -\alpha, -\beta}(0,1)}  \,  \\    
&  \approx \,  \max\left\{ r^{-1+\frac{m}{n}}\left\| \Psi_{0; \, 1, \alpha, \beta} \right\|_{L\sp{\infty}(0,r)},  \ \left\| \Psi_{\frac mn-1; \,   1, \alpha, \beta}\right\|_{L\sp{\infty}(r,1-r)}\right\}  \quad \ \text{near} \ 0.
\end{align*}   
Thus, taking into account that $m<n$,  the conclusions \eqref{E:MOpAa}\textsubscript{1,2} follow from the assertions \eqref{E:H10}\textsubscript{5-7}  and \eqref{E:H11}\textsubscript{1}.
                   
\noindent When $p \in (1, \infty)$,  according to \eqref{E:optRa}  and \eqref{E:1pie}, we get 
 \begin{align*}
\widetilde\varphi_m\left(r^{\frac1n}\right)  
& =  \|s^{-1+\frac{m}{n}}\chi_{(r,1)}(s)\|_{L\sp{p', q'; -\alpha, -\beta}(0,1)} \, \\ 
& \approx \, r^{-1+\frac{m}{n}}\left\| \Psi_{\frac{1}{p'}; \, q, \alpha, \beta}\right\|_{L\sp{q'}(0,r)} +  \left\| \Psi_{\frac{m}{n} - \frac 1p; \, q, \alpha, \beta} \right\|_{L\sp{q'}(r,1-r)}   \quad \ \text{near} \ 0.  
\end{align*}   
Thus, applying \eqref{E:H10}\textsubscript{8} and \eqref{E:H11} yields the conclusions \eqref{E:MOpAa}\textsubscript{3-10}. 

\noindent Take next $p =\infty$ and $q \in [1,\infty)$. 
Via  \eqref{E:optMoA}  and \eqref{E:1pie}, one has that       
\begin{equation*}
\begin{aligned}
\widetilde\varphi_m\left(r^{\frac1n}\right) \, 
& =\,  \|s^{-1+\frac{m}{n}}\chi_{(r,1)}(s)\|_{L\sp{(1,q';-\alpha-1, -\beta)}(0,1)} \\ 
& \approx \,  r^{-1+\frac{m}{n}}\left\| \Psi_{1; \, q, \alpha+1, \beta}\right\|_{L\sp{q'}(0,r)} +   \left\|\Psi_{\frac mn; \, q, \alpha+1, \beta}\right\|_{L\sp{q'}(r,1-r)} +  \left\| \Psi_{0; \, q, \alpha+1, \beta}\right\|_{L\sp{q'}(1-r,1)}   \quad \ \text{near} \ 0   
\end{aligned} 
\end{equation*}  for $\alpha    < - \frac 1q$, and 
\begin{equation*}
\begin{aligned}
\widetilde\varphi_m\left(r^{\frac1n}\right) \, 
& = \,  \|s^{-1+\frac{m}{n}}\chi_{(r,1)}(s)\|_{L\sp{(1, q'; -\frac{1}{q'}, -\beta-1)}(0,1)}\\ 
& \approx \,  r^{-1+\frac{m}{n}}\left\| \Psi_{1; \, q, \frac{1}{q'}, \beta+1}\right\|_{L\sp{q'}(0,r)} +   \left\|\Psi_{\frac mn; \, q, \frac{1}{q'}, \beta+1}\right\|_{L\sp{q'}(r,1-r)} +  \left\| \Psi_{0; \, q, \frac{1}{q'}, \beta+1}\right\|_{L\sp{q'}(1-r,1)}   \ \ \text{near} \ 0   
\end{aligned} 
\end{equation*} for $\alpha =- \frac 1q$ and $\beta< - \frac 1q$. 
Thus the behavior \eqref{E:MOpAa}\textsubscript{11} with $q \in [1,\infty)$ can be deduced from combining the properties \eqref{E:H10}\textsubscript{8},  \eqref{E:H11}\textsubscript{8} and \eqref{E:HOz1}, whereas \eqref{E:MOpAa}\textsubscript{12} with $q \in [1,\infty)$ follows from  \eqref{E:H10}\textsubscript{8},  \eqref{E:H11}\textsubscript{8} and \eqref{E:HOz2}.

\noindent It remains to prove conclusions \eqref{E:MOpAa}\textsubscript{11,12} under the condition that $q=\infty$, and \eqref{E:MOpAa}\textsubscript{13}. For this, it suffices to observe that
if $p =q =\infty$, and   either  $\alpha <0$   or  $\alpha= 0$ and $\beta \leq 0$, then  via   \eqref{E:optRa}  and \eqref{E:1pie} one infers that
\begin{align*} 
\widetilde\varphi_m\left(r^{\frac1n}\right) \,  
& =   \,  \|s^{-1+\frac{m}{n}}\chi_{(r,1)}(s)\|_{L\sp{1, 1; -\alpha, -\beta}(0,1)}  \,  \\ 
& \approx \,  r^{-1+\frac{m}{n}}\left\| \Psi_{1; \, \infty, \alpha, \beta}\right\|_{L\sp{1}(0,r)} +    \ \left\|\Psi_{\frac mn;\,  \infty,  \alpha, \beta}\right\|_{L\sp{1}(r,1-r)} \quad \ \text{near} \ 0, 
\end{align*} and apply the equations \eqref{E:H10}\textsubscript{8}    and     \eqref{E:H11}\textsubscript{8}.     
The statement {\sc(A)} is thus fully proved. 
 
\vspace{1mm}

\noindent{\sc(B)} We   now have to determine the explicit behavior near $0$ of the function \eqref{E:optMo} in the case when $X(0,1)=L^{(1, q; \alpha, \beta)}(0,1)$, for $q\in[1,\infty]$, $\alpha, \beta \in \R$ fulfilling one of the alternatives~\eqref{c:L1}. For ease of notation, throughout this proof we denote $\widetilde\varphi_{m,L^{(1, q; \alpha, \beta)}}$ by $\widetilde\psi_m$, and we  
assume, for the time being, that $r\in \left(0,\frac14\right]$.  
 
If $q \in [1, \infty]$ and $\alpha  > -\frac{1}{q}$,  via  \eqref{E:optRa} and \eqref{E:1pie}, one obtains that 
\begin{align*} 
\widetilde\psi_m\left(r^{\frac1n}\right) \,   
& =   \,  \|s^{-1+\frac{m}{n}}\chi_{(r,1)}(s)\|_{L\sp{\infty, q'; -\alpha-1, -\beta}(0,1)}   \,  \\ 
& \approx \,  r^{-1+\frac{m}{n}}\left\|\Psi_{0; \, q, \alpha+1, \beta}\right\|_{L\sp{q'}(0,r)} +     \left\|\Psi_{\frac mn - 1; \, q, \alpha+1, \beta}\right\|_{L\sp{q'}(r,1-r)}  \quad \ \text{near} \ 0.
\end{align*}   
Thus, an application of  \eqref{E:H10}\textsubscript{7}   -with $\alpha$ replaced by $\alpha+1$- and \eqref{E:H11}\textsubscript{1} provides the conclusion \eqref{E:MOpBb}\textsubscript{1}.   
 
\noindent For $q \in [1, \infty]$, $\alpha  = -\frac{1}{q}$ and $\beta   > -\frac{1}{q}$, the equations \eqref{E:optRa} and \eqref{E:1pie} assure that
\begin{align*} 
\widetilde\psi_m \left(r^{\frac{1}{n}} \right)   
& =   \|s^{-1+\frac{m}{n}}\chi_{(r,1)}(s)\|_{L\sp{\infty, q'; -\frac{1}{q'}, -\beta -1}(0,1)}  \, \\ 
& \approx \,   r^{-1+\frac{m}{n}}\left\|\Psi_{0; \, q, \frac{1}{q'}, \beta+1}\right\|_{L\sp{q'}(0,r)} +     \left\|\Psi_{\frac mn - 1; \, q, \frac{1}{q'}, \beta+1}\right\|_{L\sp{q'}(r,1-r)}   \quad \ \text{near} \ 0,  
\end{align*}   
whence the conclusion \eqref{E:MOpBb}\textsubscript{2} is obtained by applying \eqref{E:H10}\textsubscript{6}   -with $\beta$ replaced by $\beta+1$- and \eqref{E:H11}\textsubscript{1}.
 
\noindent Finally, the conclusion \eqref{E:MOpBb}\textsubscript{3} follows from Lemma~\ref{L:H2} once observed that, for $q \in [1, \infty)$ and $\alpha = \beta= - \frac{1}{q}$, via \eqref{E:optRa}  and \eqref{E:1pie},   
\begin{align*} 
\widetilde\psi_m \left(r^{\frac{1}{n}} \right)     
& =   \|s^{-1+\frac{m}{n}}\chi_{(r,1)}(s)\|_{L\sp{\infty, q'; -\frac{1}{q'}, -\frac{1}{q'}, -1}(0,1)} \\   
&\approx  r^{-1+\frac{m}{n}}\|  \Psi_{0; \, q, \frac{1}{q'}, \frac{1}{q'}} \ell\ell\ell^{-1}\|_{L^{q'}(0,r)}    +   \|  \Psi_{\frac{m}{n}-1; \, q, \frac{1}{q'}, \frac{1}{q'}} \ell\ell\ell^{-1}\|_{L^{q'} (r,1-r)}  \quad \ \text{near} \ 0. 
\end{align*} 
This completes the proof of Part~{\sc(B)}.
\qed

\vspace{2mm}
 In preparation for the proofs of Theorems~\ref{T:CampP} and~\ref{T:Ca*}, for convenience of the reader, we recall that  \cite{CCPS2}*{Theorem~2.6} tells us the following. Given $m \in \N$, with $m \leq n$, and an r.i.~function $\|\cdot\|_{X(0,1)}$, the function $\widehat \varphi_{m, X}\colon(0,\infty)\to(0,\infty)$, defined for $r \in \left(0,\frac14\right]$ as  
\begin{equation}\label{E:optCa}
\widehat \varphi_{m, X}\left(r\right) \  = \ 
\begin{cases}
r^{-n+1}\|\chi_{(0,r^n)}\|_{X'(0,1)} & \text{if} \ \, m =  1  \\[0.2ex]   \varrho_{m,X} (r) &     \text{if} \ \, m \in \{2, \dots, n\} 
\end{cases}
\end{equation}
and continued by $\widehat \varphi_{m, X}(r) = \widehat \varphi_{m, X}\left(\frac 14\right)$ if $r\in \left(\frac14,\infty\right)$, is an admissible function. Here, the function $\varrho_{m,X}$ is the one defined in  \eqref{rho}.
Moreover,   the space $\mathcal L^{\widehat \varphi_{m, X}(\cdot)}(\Omega)$ is the optimal (smallest possible) target Campanato space in the embedding 
\begin{equation*}
W^{m}X(\Omega) \to \mathcal L^{\widehat \varphi_{m, X}(\cdot)}(\Omega) 
\end{equation*} 
for all John domains $\Omega$ of $\R^n$.
This means that if there exists an admissible function  $\varphi$ such that $W^{m}X (\Omega) \hookrightarrow \mathcal L^{\varphi(\cdot)}(\Omega)$, then $\mathcal L^{\widehat \varphi_{m, X}(\cdot)}(\Omega) \hookrightarrow \mathcal L^{\varphi(\cdot)}(\Omega)$. 
        
\vspace{2mm}

\noindent{\textsc{Proof of Theorem~\ref{T:CampP}}.} Let $m \in \N$, with $m \leq n$, and assume that  $p, q\in[1,\infty]$, $\alpha, \beta \in \R$ satisfy one of the alternatives~\eqref{c:LZnorm}. In the light of  \cite{CCPS2}*{Theorem~2.6} just recalled, we have to detect explicitly the behavior near $0$ of  the optimal admissible function \eqref{E:optCa} in the case when $X(0,1)=L^{p, q; \alpha, \beta}(0,1)$. Indeed, as remarked in Section~{2.6}, the behavior near $0$ is the only piece of information about an admissible function  $\widehat\varphi_{m,L^{p, q; \alpha, \beta}}$ which is needed for the definition of the optimal space  $\mathcal L^{\widehat\varphi_{m,L^{p, q; \alpha, \beta}}(\cdot)}(\Omega)$.
For simplicity of notation, throughout this proof we denote $\widehat \varphi_{m, L^{p, q; \alpha, \beta}}$ by $\widehat \varphi_{m}$.
 
Note that 
\begin{equation}\label{E:new0}
\widehat \varphi_{m} \ = \   \varrho_m    \ \quad \text{for} \ m   \in  \{2, \dots,n \},
\end{equation}  
and the exact behavior near $0$ of the function $\varrho_m$ -corresponding to the different  alternatives~\eqref{c:LZnorm}- has already been  displayed in \eqref{E:Brho}.  Hence Parts~\textsc{(II)}-\textsc{(III)} follow directly from coupling  \eqref{E:new0}  with  \eqref{E:Brho}, via Remark~\ref{R:forCamP}. 

We therefore need only to consider Part~\textsc{(I)}. Assume, for the time being, that $r\in \left(0,\frac14\right]$.
The  conclusions \eqref{E:CaPAa}\textsubscript{1-3} -for $p=q =1$, and either $\alpha= 0$ and $\beta \geq 0$ or $\alpha >0$ and $\beta \in \R$, or $p \in (1, n]$- plainly follow from  \eqref{E:1pie} and \eqref{E:H10}\textsubscript{8}. Indeed,   
\begin{equation}\label{E:new1}
\begin{aligned}
\widehat\varphi_{1}\left(r^{\frac1n}\right)   
& = \,  r^{-1+\frac{1}{n}}\left\|\chi_{(0, r)}\right\|_{L\sp{p', q'; -\alpha, -\beta}(0,1)} \, = \,  r^{-1 +\frac{1}{n}}\left\|\Psi_{\frac{1}{p'}; q, \alpha, \beta } \right\|_{L\sp{q'}(0,r)} \\ 
& \,  \approx \,  r^{\frac1n-\frac{1}{p}}\ell^{-\alpha}(r) \ell\ell^{-\beta}(r)  \quad \ \text{near} \ 0. 
\end{aligned}      
\end{equation}  
 
The chain \eqref{E:new1} still holds when $p \in (n, \infty)$, $q \in [1, \infty]$, $\alpha, \beta \in \R$, and $p=q=\infty$ and either $\alpha<0$, $\beta \in \R$ or $\alpha=0$, $\beta \leq 0$. As a consequence, in the cases corresponding to these choices of the parameters $p, q ,\alpha$ and $\beta$, one infers that $\widehat\varphi_{1} \approx \widehat \sigma_{1}$, where $\widehat \sigma_{1}$ is described in \eqref{E:HOpAa}.

\noindent Finally, suppose that $p=\infty$ and $q \in [1, \infty)$. As
\begin{equation}\label{E:optCaB}
\chi_{(0,r)}  ^{**}(s) \  = \  \chi_{(0,r)}(s)  + \frac{r}{s} \, \chi_{(r,1)}(s)  \quad \  \text{for} \ s\in (0,1),
\end{equation}  
thanks to \eqref{E:1pie}, one obtains that       
\begin{equation*}
\begin{aligned}
\widehat\varphi_1\left(r^{\frac1n}\right)     
&= \,  r^{-1+\frac{1}{n}}\left\|\chi_{(0, r)}\right\|_{ L\sp{(1, q'; -\alpha-1, -\beta)} (0,1)} \\
&\approx \,  r^{-1 +\frac{1}{n}}\left\|\Psi_{1; \, q, \alpha+1, \beta}\right\|_{L\sp{q'}(0,r)}  \, + \,   r^{\frac{1}{n}}\left\|\Psi_{0; \, q, \alpha+1, \beta}\right\|_{L\sp{q'}(r,1-r)} \, + \,       r^{\frac{1}{n}}\left\|\Psi_{0; \, q, \alpha+1, \beta}\right\|_{L\sp{q'}(1-r,1)}  
\end{aligned} 
\end{equation*}   near $0$ for $\alpha  < - \frac{1}{q}$, and 
\begin{equation*}\begin{aligned}
\widehat\varphi_1\left(r^{\frac1n}\right)     
&= \,  r^{-1+\frac{1}{n}} \| \chi_{(0,r)}\|_{L\sp{(1, q'; -\frac{1}{q'}, -\beta-1)}(0,1)} \\
& \approx \,  r^{-1+\frac{1}{n}}\left\| \Psi_{1; \, q, \frac{1}{q'}, \beta+1}\right\|_{L\sp{q'}(0,r)} +  \,  r^{\frac{1}{n}}\left\|\Psi_{0; \, q, \frac{1}{q'}, \beta+1}\right\|_{L\sp{q'}(r,1-r)} + r^{\frac{1}{n}} \left\| \Psi_{0; \, q, \frac{1}{q'}, \beta+1}\right\|_{L\sp{q'}(1-r,1)}       
\end{aligned} 
\end{equation*}  
near $0$ for $\alpha  =- \frac 1q$ and $\beta< - \frac 1q$. 
Altogether, from \eqref{E:H10}\textsubscript{8}, \eqref{E:H11}\textsubscript{2}  -with $\alpha$ replaced by $\alpha+1$-  and \eqref{E:HOz1} for   $\alpha < - \frac 1q$, $\beta \in \R$ and from  \eqref{E:H10}\textsubscript{8}, \eqref{E:H11}\textsubscript{3}  -with $\beta$ replaced by $\beta+1$-  and \eqref{E:HOz2} when $\alpha =- \frac 1q$, $\beta< - \frac 1q$, it follows that the equivalence $\widehat\varphi_{1} \approx \widehat \sigma_{1}$ also holds in these cases, which completes  the proof.
\qed
 
\vspace{2mm}

\noindent{\textsc{Proof of Theorem~\ref{T:Ca*}}.}   Let $m \in \N$, with $m \leq n$, and assume that $q\in[1,\infty]$, $\alpha, \beta \in \R$ fulfill one of the alternatives~\eqref{c:L1}.  We will disclose the precise behavior near $0$ of the function  $\widehat\varphi_{m, X}$ defined in \eqref{E:optCa} in the case  $X(0,1)=L^{(1, q; \alpha, \beta)}(0,1)$. The statement will follow from \cite{CCPS2}*{Theorem~2.6}. For simplicity of notation, in what follows, we denote $\widehat \varphi_{m, L^{(1, q; \alpha, \beta)}}$ by $\widehat\psi_m$ and we assume that $r\in \left(0,\frac14\right]$.

\vspace{1mm}
\noindent {\sc(I)} First, let $q \in [1, \infty]$ and $\alpha > - \frac{1}{q}$. Then,  according to  \eqref{E:1pie} and  \eqref{E:H10}\textsubscript{7} -with $\alpha$ replaced by $\alpha+1$-, one has that 
\begin{align*} 
\widehat\psi_1\left(r^{\frac1n}\right) \,  
& = \  r^{-1+\frac{1}{n}}\left\|\chi_{(0, r)}\right\|_{L\sp{\infty, q';-\alpha-1, -\beta} (0,1)} \\ 
& = \,  r^{-1 +\frac{1}{n}}\left\|\Psi_{0; \, q, \alpha+1, \beta } \right\|_{L\sp{q'}(0,r)}  \approx \,  r^{-1 +\frac{1}{n}}\ell^{-\frac{1}{q}-\alpha}(r) \ell\ell^{-\beta}(r)   \quad \ \text{near} \ 0, 
\end{align*}     
whence  \eqref{E:Ca*Aa}\textsubscript{1} follows.

\noindent Next, take $q \in [1, \infty]$, $\alpha  = -\frac{1}{q}$ and $\beta   > -\frac{1}{q}$. Then, an application of \eqref{E:H10}\textsubscript{6}   -with $\beta$ replaced by $\beta+1$- tells us that
\begin{align*} 
\widehat\psi_1\left(r^{\frac1n}\right) \, 
& =     \,  r^{-1+\frac{1}{n}}  \|\chi_{(0, r)}\|_{L\sp{\infty, q'; -\frac{1}{q'}, -\beta -1}(0,1)}  \,  \\
& =     \,  r^{-1+\frac{1}{n}}\left\|\Psi_{0; \, q, \frac{1}{q'}, \beta+1}\right\|_{L\sp{q'}(0,r)}  \, \approx \,  r^{-1+\frac{1}{n}} \ell\ell^{-\frac{1}{q}-\beta}(r) \quad \ \text{near} \ 0,
\end{align*}   
thereby \eqref{E:Ca*Aa}\textsubscript{2} holds.
Finally, pick $q \in [1, \infty)$ and $\alpha = \beta= - \frac{1}{q}$.  Then,
\begin{equation*} 
 \widehat\psi_1\left(r^{\frac1n}\right) \,  =     \,    r^{-1+\frac{1}{n}}       \|\chi_{(0, r)}\|_{L\sp{\infty, q'; -\frac{1}{q'}, -\frac{1}{q'}, -1}(0,1)}    \,  =     \,
r^{-1+\frac{1}{n}}    \|  \Psi_{0; \, q, \frac{1}{q'}, \frac{1}{q'}} \ell\ell\ell^{-1}\|_{L^{q'}(0,r)}    \quad \ \text{near} \ 0,
 \end{equation*}
so Lemma~\ref{L:H2} yields that
 \begin{equation*}  
\widehat\psi_1(r) \,  \approx \, r^{1-n}\ell\ell\ell^{-\frac{1}{q}}(r)  \quad \ \text{near} \ 0,
\end{equation*}
namely, \eqref{E:Ca*Aa}\textsubscript{3} is true.

\vspace{1mm}

\noindent  (II) We shall focus only on the case when $m \in [2, n-1]$. Indeed, for $m=n$, in the light of the very definition \eqref{E:optCa}\textsubscript{2} and the convention  \eqref{H=C}, the function $\widehat\psi_n$ agrees with the optimal modulus of continuity $\widehat \omega_n$ in the embedding \eqref{E:HO*A}  whose exact form has been settled  in \eqref{E:HO*B}.

\noindent  Assume first that $q \in [1, \infty]$ and $\alpha > - \frac{1}{q}$. Then,  via \eqref{E:optHma}  and \eqref{E:1pie}
 one has that  
\begin{equation*}
\begin{aligned}
\widehat\psi_m\left(r^{\frac1n}\right) \
& =   \  r^{\frac{1}{n}} \left\|s^{-1+\frac{m- 1}{n}}\chi_{(r,1)}(s)\right\|_{L\sp{\infty, q'; -\alpha-1, -\beta}(0,1)}   \\
& \approx \, r^{-1+\frac{m}{n}}\left\|\Psi_{0; \, q, \alpha+1, \beta}\right\|_{L\sp{q'}(0,r)} +    \ r^{\frac{1}{n}}\left\|\Psi_{\frac{m-1}{n} - 1; \, q, \alpha+1, \beta}\right\|_{L\sp{q'}(r,1-r)} \quad \ \text{near} \ 0.
\end{aligned} 
\end{equation*}  
Hence, the conclusion  \eqref{E:Ca*Bb}\textsubscript{1} follows from the equivalences \eqref{E:H10}\textsubscript{7} and either    \eqref{E:H11}\textsubscript{1} when $m<n-1$ or \eqref{E:H11}\textsubscript{7} when $m=n-1$, all applied with $\alpha$ replaced by $\alpha+1$.

\noindent When $q \in [1, \infty]$, $\alpha  = -\frac{1}{q}$ and $\beta   > -\frac{1}{q}$,  again via   \eqref{E:optHma}  and  \eqref{E:1pie}, one infers that
\begin{align*} 
\widehat\psi_m\left(r^{\frac{1}{n}} \right)  
& =   \  r^{\frac{1}{n}}\|s^{-1+\frac{m- 1}{n}}\chi_{(r,1)}(s)\|_{L\sp{\infty, q'; -\frac{1}{q'}, -\beta -1}(0,1)} \\ 
& \approx \,         r^{-1+\frac{m}{n}}\left\|\Psi_{0; \, q, \frac{1}{q'}, \beta+1}\right\|_{L\sp{q'}(0,r)} +    \ r^{\frac{1}{n}}\left\|\Psi_{\frac {m-1}{n} - 1; \, q, \frac{1}{q'}, \beta+1}\right\|_{L\sp{q'}(r,1-r)}    \quad \  \text{near} \ 0,
\end{align*}  whence an application -with $\beta$ replaced by $\beta+1$- of   \eqref{E:H10}\textsubscript{6} and either    \eqref{E:H11}\textsubscript{1} when $m<n-1$ or \eqref{E:H11}\textsubscript{6} when $m=n-1$
implies  \eqref{E:Ca*Bb}\textsubscript{2}. 

 \noindent Finally, for $q \in [1, \infty)$ and $\alpha = \beta= - \frac{1}{q}$, 
\begin{equation*}\begin{aligned} 
\widehat\psi_m\left(r^{\frac{1}{n}} \right)  
&   =  \,      r^{\frac{1}{n}}\|s^{-1+\frac{m-1}{n}}\chi_{(r,1)}(s)\|_{L\sp{\infty, q'; -\frac{1}{q'}, -\frac{1}{q'}, -1}(0,1)} \\ 
& \approx \, r^{-1+\frac{m}{n}}\left\|  \Psi_{0; \, q, \frac{1}{q'}, \frac{1}{q'}} \ell\ell\ell^{-1}\right\|_{L^{q'}(0,r)} \   + \ r^{\frac{1}{n}}\left\|  \Psi_{\frac{m-1}{n}-1; \, q, \frac{1}{q'}, \frac{1}{q'}} \ell\ell\ell^{-1}\right\|_{L^{q'} (r,1-r)} \quad \ \text{near} \ 0.
\end{aligned}\end{equation*}
Consequently, an application of Lemma~\ref{L:H2} yields the conclusion \eqref{E:Ca*Bb}\textsubscript{3}, ending the proof.
\qed

\vspace{4mm}
\section*{List of the principal notations}\label{sec7}
 
 \vspace{1mm}
{\it Extended norms and semi-norms}

\vspace{-7mm}
\begin{multicols}{2}
\begin{align*}
 &\| \cdot \|_{C^{0,\sigma(\cdot)}(\Omega)}     
 &\eqref{E:Hnorm}, &\, \mbox{p.}~\pageref{E:Hnorm} \\
&\| \cdot \|_{L^{p, q; \alpha, \beta}(0,1)}   
&\eqref{c:LZfunct}, &\, \mbox{p.}~\pageref{c:LZfunct} \\
&\| \cdot \|_{L^{(p, q; \alpha, \beta)}(0,1)}   
&\eqref{c:LZfunct2}, &\, \mbox{p.}~\pageref{c:LZfunct2} \\
&\| \cdot \|_{\Lambda_\phi(0,1)}  
&\eqref{EndLor_no}, &\, \mbox{p.}~\pageref{EndLor_no} \\
&\| \cdot \|_{\mathsf{M}_{\phi}(0,1)}  
&\eqref{EndMarc_no}, &\, \mbox{p.}~\pageref{EndMarc_no} \\
&\| \cdot \|_{\mathcal M^{\varphi(\cdot)}(\Omega)}   
&\eqref{morreydef}, &\, \mbox{p.}~\pageref{morreydef}
\end{align*}
\columnbreak
\begin{align*}
\\
&\| \cdot \|_{X(0,1)}   
& &\, \mbox{p.}~\pageref{F}\\
&\| \cdot \|_{X'(0,1)}   
&\eqref{n.assoc.}, &\, \mbox{p.}~\pageref{n.assoc.} \\
&\| \cdot \|_{X_{m, \text{opt}}'(0,1)}  
&\eqref{E:Opt_S}, &\, \mbox{p.}~\pageref{E:Opt_S} \\
&\| \cdot \|_{_{W^mX(\Omega)}}  
&\eqref{SOB_norm}, &\, \mbox{p.}~\pageref{SOB_norm} \\ 
& \, | \cdot |_{\mathcal L^{\varphi(\cdot)}(\Omega)}   
&\eqref{E:Campsem}, &\, \mbox{p.}~\pageref{E:Campsem}\end{align*}
\end{multicols}

 \vspace{-3mm}
 
{\it Functionals, operators and operations}  
\vspace{-7mm}
\begin{multicols}{2}
\begin{align*}
&\hookrightarrow     
&  &\, \mbox{p.}~\pageref{ARROW} \\
&S_m     
&\qquad \qquad\eqref{E:EOPJC}, &\, \mbox{p.}~\pageref{E:EOPJC} \\
\end{align*}
\columnbreak
\begin{align*}
\\
&u^*  
&\qquad \qquad\eqref{decreasing rearrangementE}, &\, \mbox{p.}~\pageref{decreasing rearrangementE} \\
&u^{**}
&\qquad \qquad\eqref{u^**}, &\, \mbox{p.}~\pageref{u^**}   
 \end{align*}
\end{multicols}
 
\vspace{-3mm}
{\it Functions}
\vspace{-7mm}
\begin{multicols}{2}
\begin{align*}
 &\ell, \ell\ell     
 &\qquad \qquad\eqref{c:convLeLL}, &\, \mbox{p.}~\pageref{c:convLeLL} \\
&\ell\ell\ell  
&\eqref{c:convLLL}, &\, \mbox{p.}~\pageref{c:convLLL} 
 \end{align*}
\columnbreak
\begin{align*}
\\
&\phi_X  
&\qquad \qquad\eqref{feb99}, &\,  \mbox{p.}~\pageref{feb99} \\
&\Psi_{\lambda; \, q, \alpha, \beta}
&\qquad \qquad\eqref{E:1pie}, &\, \mbox{p.}~\pageref{E:1pie}  
 \end{align*}
\end{multicols}

\vspace{-3mm}
 
{\it Other symbols}
\vspace{-7mm}
\begin{multicols}{2}
\begin{align*}
 & \lesssim     
 &\qquad \qquad &\qquad \quad \mbox{p.}~\pageref{A} \\
& \gtrsim  
&\qquad \qquad & \qquad \quad \mbox{p.}~\pageref{B} \\
& \approx  
&\qquad \qquad &\qquad \quad \mbox{p.}~\pageref{C} \\
& \  p^{\ast}  
&\qquad \qquad &\qquad \quad \mbox{p.}~\pageref{Sconj}
 \end{align*}
\columnbreak
\begin{align*}
\\
& \  p'  
&\qquad \qquad &\qquad \quad \mbox{p.}~\pageref{Hco} \\
&C^0_b(\Omega)  
&\qquad \qquad &\qquad \quad \mbox{p.}~\pageref{CON} \\
&L^0(E)  
&\qquad \qquad &\qquad \quad \mbox{p.}~\pageref{D} \\
&L^0_+(E)
&\qquad \qquad &\qquad \quad \mbox{p.}~\pageref{E}  
 \end{align*}
\end{multicols}

\noindent {\bf Data availability}

\smallskip

\noindent No data was used for the research described in the article.

 \vspace{3mm}
 
\noindent {\bf Conflict of Interest Statement}  
\smallskip

\noindent The authors declare  no potential conflict of interests.

\vspace{3mm}

\noindent {\bf ORCID} 
\smallskip

\noindent {\it Paola Cavaliere}
\orcidlink{0000-0002-7829-0015}.
 \smallskip
\noindent {\it Ladislav Drážný}
\orcidlink{0009-0003-1807-6003}

\end{document}